\documentclass[12pt,reqno]{amsart}
\usepackage[T1]{fontenc}
\usepackage{amssymb}
\usepackage[mathscr]{eucal}
\usepackage{xypic}
\usepackage{amsmath}
\usepackage[all]{xy}
\usepackage{lineno}

\usepackage{epsfig}
\usepackage{amscd}
\usepackage{mathrsfs}
\usepackage{enumerate}
\usepackage{amsfonts}
\usepackage{amsmath}

\usepackage{hyperref}
\usepackage[nameinlink]{cleveref}
\usepackage[thinc]{esdiff}
\usepackage{tikz-cd}
\usepackage[normalem]{ulem}

\usepackage{xcolor}

\theoremstyle{definition}
\newtheorem{theorem}{Theorem}[section]
\newtheorem{definition}{Definition}[section]
\newtheorem{lemma}[theorem]{Lemma}
\newtheorem{proposition}[theorem]{Proposition}
\newtheorem{corollary}[theorem]{Corollary}
\newtheorem{remark}[theorem]{Remark}
\newtheorem{example}[theorem]{Example}

\numberwithin{equation}{section}

\newcommand{\mc}{\mathcal}
\newcommand{\mb}{\mathbb}

\newcommand{\ra}{\rightarrow}
\newcommand{\rra}{\rightrightarrows}
\newcommand{\ghta}{(G,H,\tau, \alpha)}
\newcommand{\hrtag}{H \rtimes_{\alpha} G}
\newcommand{\Lrrw}{\Longrightarrow}

\DeclareMathOperator{\ad}{ad}

\def\og{\leavevmode\raise.3ex\hbox{$\scriptscriptstyle\langle\!\langle$~}}
\def\fg{\leavevmode\raise.3ex\hbox{~$\!\scriptscriptstyle\,\rangle\!\rangle$}}

\begin{document}
	
	\title
	{Parallel transport on a Lie 2-group bundle over a Lie groupoid along Haefliger paths}

	\author[Saikat Chatterjee]{Saikat Chatterjee}
	\author[Adittya Chaudhuri]{Adittya Chaudhuri}
	\address{School of Mathematics,
		Indian Institute of Science Education and Research Thiruvananthapuram,
		Maruthamala P.O., Vithura, Kerala 695551, India}
	\email{saikat.chat01@gmail.com,chaudhuriadittya@gmail.com, adittyachaudhuri15@iisertvm.ac.in}

	

	\begin{abstract}

		We prove a Lie 2-group torsor version of the well-known one-one correspondence between fibered categories and pseudofunctors. Consequently, we obtain a weak version of the principal Lie group bundle over a Lie groupoid. The correspondence also enables us to extend a particular class of principal 2-bundles to be defined over differentiable stacks. We show that the differential geometric connection structures introduced in the authors' previous work, combine nicely with the underlying fibration structure of a principal 2-bundle over a Lie groupoid. This interrelation allows us to derive a notion of parallel transport in the framework of principal 2-bundles over Lie groupoids along a particular class of Haefliger paths. The corresponding parallel transport functor is shown to be smooth. We apply our results to examine the parallel transport on an associated VB-groupoid.

%
	\end{abstract}

	\maketitle
	\addtocontents{toc}{\setcounter{tocdepth}{1}} 
	\tableofcontents
	
	\section{Introduction}
	For the last few decades or so, higher gauge theories provided frameworks for describing the dynamics of string-like extended "higher dimensional objects." Typically, they involve an appropriately categorified version of a smooth fiber bundle equipped with a connection structure that induces a notion of parallel transport consistent with this categorification. The precise description of the categorified objects depends largely on the framework. Our object of interest in this paper is, in particular, a categorified principal bundle that lives in the realm of Lie groupoids. In our earlier paper, we introduced such an object as a principal 2-bundle over a Lie groupoid [Definition 3.6,  \cite{chatterjee2022atiyah}]. More precisely, it consists of a morphism of Lie groupoids $\pi:=(\pi_1, \pi_0) \colon [E_1 \rra E_0] \ra [X_1 \rra X_0]$ with a functorial action of a Lie 2-group $\mb{G}:=[G_1 \rra G_0]$ on $\mb{E}:=[E_1 \rra E_0]$ such that both $\pi_1 \colon E_1 \ra X_1$ and $\pi_0 \colon E_0 \ra X_0$ are classical principal $G_1$-bundle and principal $G_0$-bundle over $X_1$ and $X_0$ respectively. In \cite{chatterjee2022atiyah}, we studied connection structures and gauge transformations on these objects, characterizing certain subclasses and constructing several interesting examples. Recently, Herrera-Carmona and Oritz in \cite{herreracarmona2023chernweillecomte}  introduced a Chern-Weil map for these objects by an application of an extended version of the Chern-Weil-Lecomte characteristic map in the setting of $L_{\infty}$-algebras using a similar connection structure as ours. \cite{MR4598921} refers to our principal 2-bundles as principal bundle groupoids and studied their associated nerves. We ought to mention here that Ginot and Stiénon in \cite{MR3480061}, introduced a slightly more general notion of a principal 2-bundle over a Lie groupoid interms of a generalized morphism, though their objective and approach are quite distinct from ours. Moreover, the definition of a principal Lie group bundle over a Lie groupoid considered in \cite{MR2270285} can be viewed as our special case.  
	
	Motivated by the  connection induced horizontal path lifting property of traditional principal bundles, in  \cite{chatterjee2022atiyah}, we introduced a structure on our principal 2-bundles by the name categorical connection [Definition 3.21,\cite{chatterjee2022atiyah}]. Although, when the base of a 2-bundle is a sort of "path groupoid," the notion, in particular, turned out to be quite interesting as we see in \cite{MR3126940} however, the categorical connection introduced in \cite{chatterjee2022atiyah} is analogous to the set-theoretic notion of a "splitting cleavage" in a fibered category. We in the present article introduce a weaker notion of a categorical connection and call it a \textit{quasi connection} whose set theoretic analog is a general cleavage, i.e., a cleavage not closed under composition, and not necessarily sends an identity to identity. Like the way a VB-groupoid is a groupoid object in the category of vector bundles (Proposition 3.5, \cite{MR3696590}), our principal 2-bundle is also a groupoid object in the category of principal bundles (Remark 1.16, \cite{MR4598921}). But they are still a much less explored and a newcomer in the literature, compared to a VB -groupoid which has already proven its worth through its intimate connection with Poisson geometry \cite{mackenzie1999symplectic, MR1262213, mackenzie2000integration} and representation theory of Lie groupoids \cite{MR3696590, MR4126305}. In particular, Gracia-Saz and Mehta in [Theorem 6.1,\cite{MR3696590}] showed that the 2-term representations up to homotopy of a Lie groupoid $\mb{X}$ is in one-one correspondence with the VB-groupoids over $\mb{X}$ equipped with right horizontal lifts (a term they used for unital linear cleavages). Further, del Hoyo and Oritz improved the result to a categorical equivalence in [Theorem 2.7, \cite{MR4126305}]. From the groupoid object perspective, a VB-groupoid over a Lie groupoid $\mb{X}$ with a linear cleavage corresponds to a principal 2-bundle over $\mb{X}$ equipped with a quasi connection in our framework (that is indeed the case as we show in the current paper through an associated bundle construction). We call them \textit{quasi-principal 2-bundles}. Natural questions arise, 
	
	"Can we realize a quasi-principal 2-bundle as a Grothendieck construction of some object? What is that particular object? Are their respective categories equivalent?"

	We constructed the object explicitly, which appears to be a weakened version of a principal Lie group bundle over a Lie groupoid (as studied in \cite{MR2270285}), whose underlying action of the base Lie groupoid on the total space is now replaced by a quasi-action(which is not closed under composition and unit map)  and "is an action upto a data of a Lie crossed module." Then, we produced a quasi-principal 2-bundle from it, and extended this association to an equivalence between their respective categories. Moreover, as a byproduct of this categorical equivalence, we obtained a subclass of principal 2-bundles which can be defined over a differentiable stack represented by its base Lie groupoid.
	In other words, we provided a Lie 2-group torsor version of the classic Grothendieck correspondence between general fibered categories and pseudofunctors, which, according to the best of our knowledge, is a new addition to the existing literature. This forms the first part of this paper. 
	
	 Rest of the paper is devoted to the development of a parallel transport theory on a quasi-principal 2-bundle (a principal 2-bundle over a Lie groupoid equipped with a quasi connection) induced from the differential geometric connection structures developed by us in \cite{chatterjee2022atiyah} along "certain class" of Haefliger paths in Lie groupoids (\cite{guruprasad2006closed, bridson2013metric, haefliger1982groupoides, MR1950946, colman20111, gutt2005poisson}). This also involves the construction of a parallel transport functor between appropriate source and target categories. Before we summarize our results and main achivements, we will provide a brief overview of some works in the development of parallel transport theory on a categorified framework of bundles and connections below:

Before delving into the existing works on principal bundles over appropriately categorified bases, we mention some works concerning higher principal bundles over manifolds.

 Baez\cite{baez2002higher}, Baez-Schreiber\cite{MR2342821}, Picken-Martins\cite{martins2010two}, Mackaay-Picken \cite{MR1932333}, Schreiber-Waldorf's \cite{MR2520993, schreiber2011smooth, MR3084724} along with some other papers cited therein comprises some of the earliest work in this area. In particular, Schreiber-Waldorf developed a general model-independent axiomatic approach to the theory of parallel transport. An interesting aspect of their approach is the axiomatic characterization of the smoothness condition for parallel transport functors/2-functors of geometric objects like a connection on principal bundles and non-abelian gerbes over a manifold. In \cite{MR3521476}, Collier, Lerman, and Wolbert introduced an alternative notion of smoothness for transport functors [Definition 3.5,\cite{MR3521476}] and argued its equivalence with the one mentioned in 
	\cite{MR2520993}. They also showed its validity by proving the smoothness (in their terms) of the  connection induced parallel transport functor of a classical principal bundle. More recently, in \cite{MR3894086}
	and \cite{MR3917427}, Waldorf introduced a notion of parallel transport on a model of principal Lie 2-group bundle over a manifold, induced from a global connection data. Also, a few years back, again in a framework of principal 2-bundle over a manifold, Kim and Saemann, via adjusted Weil algebra in \cite{kim2020adjusted},  introduced a notion of generalized higher holonomy functor that has been successfully kept free from the usual fake-flatness condition. One can find an aprroach to parallel transport in terms of  Lie crossed module cocycles over a manifold in \cite{MR3357822}, and \cite{MR3529236} explores its relation with the knot theory. A gluing algorithm for local 2-holonomies has been provided by \cite{MR3645839}. Also, double category theoretic approaches to higher gauge theories can be found in both Morton-Picken's \cite{MR4037666} and Zucchini-Soncini's \cite{MR3357822}. Although our overview list is far from complete, in all the approaches we mentioned so far, the base space of the concerned categorifed geometric object is manifold.
	However, since our current framework admits a categorified base space, i.e., a Lie groupoid, we must mention at least some works in this direction.

	Gengoux, Tu, and Xu in \cite{MR2270285} introduced a notion of holonomy map along a generalized pointed loop for a principal Lie group bundle over a Lie groupoid equipped with a flat connection. More recently, in \cite{MR3521476}, Collier, Lerman, and Wolbert studied parallel transport on a principal Lie group bundle over a differentiable stack. In particular, they defined their principal bundles, connection structures, and parallel transports as 1-morphisms of stacks. There are also other approaches, such as \cite{MR3126940} and \cite{MR3566125}, where authors considered the base as a path groupoid and as an affine 2-space, respectively. In order to contrast with our approach and motivation, we emphasize here that we consider parallel transport on a Lie 2-group bundle over a Lie groupoid along a certain class of Haefliger paths, which, despite of its interesting properties, seems to be unexplored so far.


	Now, let us come back to the current paper! 
	
	There is a notion of a path in a Lie groupoid due to Haefliger (\cite{haefliger1971homotopy, haefliger1990orbi, guruprasad2006closed}), along with a notion of homotopy between such paths, which yields a notion of fundamental groupoid of a Lie groupoid (for example see \cite{colman20111, mehta2011double, gutt2005poisson, MR2166453}). In the existing literature, these paths are usually known by the names \textit{Haefliger paths} or \textit{$\mc{G}$-paths in a Lie groupoid $\mc{G}$}. Consistent with our notations, we call them \textit{$\mb{X}$-paths} in a Lie groupoid $\mb{X}=[X_1 \rra X_0]$. Loosely speaking, an $\mb{X}$-path is a sequence of the form $(\gamma_0, \alpha_1,\gamma_1, \cdots,\alpha_n, \gamma_n)$ for some $n \in \mb{N}$, where each $\gamma_i \in X_1$ and each $\alpha_i$ is a path in $X_0$, such that they are compatible with each others' source-target, initial-finial point in an appropriate way. In this paper, we will introduce a notion of \textit{thin homotopy between "lazy Haefliger paths" or lazy $\mb{X}$-paths}, a suitable generalization of the classical notion of thin homotopy between paths with sitting instants, i.e., a smooth path in a manifold which is constant near the endpoints,  in the setting of Lie groupoids. Our notion of thin homotopy between lazy Haefliger paths is a minor variant of the existing notion of homotopy between Haefliger paths (\cite{colman20111, mehta2011double, gutt2005poisson, MR2166453}). We show that our notion of thin homotopy will naturally yield a notion of \textit{thin fundamental groupoid of a Lie groupoid}, which possesses a diffeological structure. Interestingly, we found out that the multiplicative nature of a connection 1-form(introduced by us in \cite{chatterjee2022atiyah})  and a quasi connection structure combine nicely to produce a reasonable notion of parallel transport on a quasi-principal 2-bundle along such lazy Haefliger paths. 
	This parallel transport turns out to be invariant under our notion of thin homotopy, and as a consequence produces
	 a functor from the thin fundamental groupoid of the base Lie groupoid to a quotient category of $\mb{G}$-torsors, where $\mb{G}$ is the structure Lie 2-group of the quasi-principal 2-bundle.
	 The parallel transport functor behaves naturally with respect to the pullback and connection preserving quasi-principal 2-bundle morphisms. Moreover, we showed this parallel transport functor is "smooth" in a way that generalizes the notion of smoothness introduced by Collier, Lerman, and Wolbert for the traditional principal bundle in [Definition 3.5,\cite{MR3521476}]. Fixing a Lie 2-group $\mb{G}$ and a Lie groupoid $\mb{X}$, we then extend this parallel transport functor to define a functor 
	\begin{equation}\label{Main functor}
		\mc{F} \colon \rm{Bun}_{\rm{quasi}}^{\nabla}(\mb{X}, \mb{G}) \ra \rm{Trans}(\mb{X},\mb{G}).
	\end{equation}
	In \Cref{Main functor}, $\rm{Bun}_{\rm{quasi}}^{\nabla}(\mb{X}, \mb{G})$ is the category of quasi-principal -$\mb{G}$-bundles equipped with strict connections (introduced in \cite{chatterjee2022atiyah}) over the Lie groupoid $\mb{X}$. On the other side, $\rm{Trans}(\mb{X},\mb{G})$ is the category of functors from the thin fundamental groupoid of $\mb{X}$ to the priorly mentioned quotiented category of $\mb{G}$-torsors. Finally, we developed a notion of induced parallel transport theory on VB-groupoids through associated bundle construction. 
	
	Except for a cursory mention of an idea in [Subsection 4.1.3, \cite{guruprasad2006closed}], where the setup and context are very different, according to the best of our knowledge, we believe this approach to parallel transport (especially along Haefliger paths in Lie groupoids) is new to the existing ones in higher gauge theory literature body. We hope our approach will bring a new insight to the parallel transport theory (induced from suitable connection data) on any geometric object that has an underlying Lie groupoid fibration structure with an appropriate cleavage. Nonetheless, we also must mention certain worthy topics which we have not covered in this paper such as  the construction of a quasi-principal 2-bundle with connection from the data of a parallel transport functor, a notion of parallel transport along higher dimensional objects like surfaces, etc. We will include these topics in a forthcoming paper that hopefully will make the relation with other existing notions of higher parallel transport theories more transparent.	
	\subsection*{The outline and organization of the paper}
	\Cref{Section 2} covers some standard results related to Lie 2-groups, fibered categories and diffeological spaces. We also recall some necessary results from our earlier work \cite{chatterjee2022atiyah}.
	
	The notions of quasi connection and a quasi-principal 2-bundle are introduced in \Cref{quasi-principal 2-bundles and their characterisations}. In \Cref{Examples of quasiprincipal 2-bundles} we construct some examples of the same. \Cref{Grothendieck construction subsection} establishes one of the key results where we realize a quasi-principal 2-bundle as a Grothendieck construction, and provide a Lie 2-group torsor version of the correspondence between fibered categories and pseudofunctors. In the following subsection, a characterization for quasi connections has been given in terms of retractions. We end the section by extending a class of principal 2-bundles over a differentiable stack.
	
	In \Cref{Section Lazy Haefliger paths and thin fundamental groupoid of a Lie groupoid}, we introduce the definitions of a lazy Haefliger path in a Lie groupoid and a  thin homotopy between them, resulting in a thin fundamental groupoid of a Lie groupoid. \Cref{Subsection Smoothness of thin fundamental groupoid of a Lie groupoid} provides a diffeological structure on the thin fundamental groupoid.
	
	\Cref{Section: Parallel Transport on principal 2-bundles} develops a notion of thin homotopy invariant parallel transport on a quasi-principal $\mb{G}$-bundle over a Lie groupoid $\mb{X}$ along a lazy Haefliger path, which results in a functor, namely the \textit{parallel transport functor}, from the thin fundamental groupoid of $\mb{X}$ to a quotiented category of $\mb{G}$-torsors. The main result of this section is \Cref{Equivalence of quasi and functors}, where we establish \Cref{Main functor}. We observe the naturality of the parallel transport functor with respect to pullback constructions and connection-preserving bundle morphisms in \Cref{Naturality of the parallel transport functor}. It is crucial to note that the parallel transport functor enjoys appropriate smoothness properties discussed in \Cref{Smoothness of parallel transport}.

	In \Cref{Associated PAPER VERSION}, we produce a VB-groupoid associated to a quasi-principal 2-bundle, and study the parallel transport on associated VB-groupoid along a lazy Haefliger path.

	\subsection{Notations and conventions}\label{Notations and Conventions} 
	Here, we fix some conventions and notations which we will follow throughout this paper. 
	
	We assume all our manifolds to be smooth, second countable and Hausdorff. We will denote the category of such manifolds by \rm Man. For any smooth map $f\colon M\to N,$ the differential at $m\in M$ will be denoted as
	$f_{*,m}\colon T_m M\ra T_{f(m)}N.$ A smooth right (resp. left) action of a Lie group $G$ on a smooth manifold $P$ is denoted by $(p, g)\mapsto p g$ (resp. $(g, p)\mapsto g p$), for $g\in G$ and $p\in P$. If the Lie group $G$ acts on a manifold $M$, such that the action is free and transitive, then we will call $M$ a $G$-torsor, and the groupoid of $G$-torsors will be denoted by $G$-Tor.  
	To make the concatenation of smooth paths in a manifold smooth, we will restrict ourselves to only \textit{paths with sitting instants} i.e. a smooth map $\alpha :[0,1] \rightarrow M$ to a manifold $M$, such that there exists an $\epsilon \in (0,1/2)$ and satisfying $\alpha(t)= \alpha(0)$ for $t \in [0,\epsilon)$ and $\alpha(t)=\alpha(1)$ for $t \in (1- \epsilon, 1]$, see [Definition 2.1, \cite{MR2520993}]. We will use the notation $PM$ to denote the set of smooth paths with sitting instants in the manifold $M$. For any path $\alpha$, the notation $\alpha^{-1}$ denotes the path defined by $t \mapsto \alpha(1-t)$ for all $t \in [0,1]$.

	Unless otherwise stated, for any category, we denote the source, target and unit maps by $s, t$ and $u$ respectively. For any object $p$ in a category, $1_P$ denotes the element $u(p)$. We write the composition of a pair of morphisms $f_2, f_1$ as $f_2\circ f_1,$ when $t(f_1)=s(f_2).$ In a groupoid, we denote the inverse map by $\mathfrak{i}$ and for any morohism $\gamma$, we shorten the notation by writing $\gamma^{-1}$ instead of $\mathfrak{i} (\gamma)$. 
	
	\textit{Lie groupoids} are groupoid object in $\rm Man,$ such that source and target maps are surjective submersions. We use the blackboard bold notation to denote a Lie groupoid; that is, $\mb{E}$ for $[E_1\rra E_0]$ and so on. A \textit{morphism of Lie groupoids} $F: \mb{X} \ra \mb{Y}$ is a functor such that both the object level and morphism level maps are smooth, and we denote them respectively by $F_0 \colon X_0 \ra Y_0$ and $F_1 \colon X_1 \ra Y_1$. A \textit{smooth natural transformation} from a morphism of Lie groupoids $F \colon \mb{X} \ra \mb{Y}$ to another $F' \colon \mb{X} \ra \mb{Y}$ is a natural transformation $\eta$ between the underlying functors such that associated map $X_0 \ra Y_1$  is smooth and we denote it by $\eta \colon F \Longrightarrow F'$. For a Lie groupoid $\mb{X}$, we denote the associated \textit{tangent Lie groupoid} by $T\mb{X}=[TX_1\rra TX_0],$ and define the structure maps by the differentials of the respective structure maps of $\mb{X}.$  For a compoasable pair of morphisms $(\gamma_2, X_2), (\gamma_1, X_1)$ in  $T\mb{X},$ the composition $(\gamma_2, X_2)\circ (\gamma_1, X_1)=\bigl(m(\gamma_2, \gamma_1), m_{* (\gamma_2, \gamma_1)}(X_2, X_1)\bigr)$  will be denoted by $(\gamma_2\circ \gamma_1, X_2\circ X_1),$ where $m$ is the composition map of the Lie groupoid $\mb{X}.$

\section{Some background materials}\label{Section 2}
To make this paper self-readable, in this section, we briefly recall some notions that are either already well established in the existing literature or we have already introduced such notions in our earlier paper \cite{chatterjee2022atiyah}.
\subsection{Lie 2-groups and Lie crossed modules}\label{subsection: Lie 2-groups and Lie crossed modules}
Here, we recall the basics of Lie 2-groups, Lie crossed modules, and related notions. The material in this subsection is standard, and we refer to the following papers \cite{{MR2342821}, {MR2068521}, {MR2068522}, {MR3126940}, {MR3504595}, {MR2805195}} for further reading in these topics.

A \textit{(strict) Lie 2-group} is a Lie groupoid $\mb{G}$, along with a morphism of Lie groupoids $$\otimes : \mb{G} \times \mb{G} \ra \mb{G}$$ inducing Lie group structures on both objects and morphisms. A \textit{ Lie crossed module},  is a  $4$-tuple $(G,H,\tau,\alpha)$ where $G,H$ are Lie groups, $\alpha:G\times H\ra H$ is a smooth action of $G$ on $H$ such that $\tau:H\ra G$ is a morphism of Lie groups  and for each $g \in G$, $\alpha(g,-):H \ra H$ is a Lie group homomorphism such that 
\begin{equation}\label{E:Peiffer}
	\begin{split}
		& \tau(\alpha(g,h))=g\tau(h)g^{-1} \, \textit {\rm for all} \, (g,h) \in G\times H,\\
		& \alpha(\tau(h),h')=hh'h^{-1} \, \textit {\rm for all}\, h,h'\in H.
	\end{split}
\end{equation}
It is well known that one can identify a Lie 2-group with a Lie crossed module. In particular, given a Lie crossed module $(G,H,\tau,\alpha)$ the associated  Lie $2$-group is given by the Lie groupoid $ \mb{G}=[H \rtimes_{\alpha} G \rightrightarrows G]$, where $\hrtag$ denotes the semidirect product of Lie groups $H$ and $G$ with respect to the action $\alpha$ of $G$ on $H$. We list the structure maps of this Lie $2$-group below:
\begin{itemize}
	\item the source map is given by $s(h,g)=g$,
	\item the target map is given by $t(h,g)=\tau(h)g$,
	\item the composition map is given by $m((h_2, g_2),(h_1, g_1))=(h_2 h_1,g_1)$,
	\item the identity map is given by $1_g=(e,  g)$,
	\item the inverse map is given by $\mathfrak{i}(h,g)=(h^{-1},\tau(h)g)$,
	\item the group properties of  $H\rtimes_{\alpha} G$ are those of the standard semidirect product of the groups, that is 
	the bi-functor $\otimes: \mb{G} \times\mb{G} \rightarrow \mb{G}$ is defined as 
	\begin{equation}\label{E:grouppro}
		\begin{split}
			&\otimes_0: (g_1, g_2) \mapsto g_1g_2,
			\\
			&\otimes_{1}:  ((h_2, g_2), (h_1, g_1)) \mapsto (h_2 \alpha(g_2,h_1), g_2 g_1),
		\end{split}
	\end{equation}
\end{itemize}
whereas the group inverse and the identity elements are respectively given as $(h,g)^{-1}=(\alpha(g^{-1}, h^{-1}), g^{-1})$ and  $(e,e)$. Note that we use the same notation $e$ for identity elements in both $H$ and $G$, and the distinction should be made according to the context.
On the other hand, given a Lie 2-group $\mb{G}=[G_1 \rra G_0]$, the associated Lie crossed module is given by the $4$-tuple
	$(G_0, \ker(s), t|_{\ker(s)}:\ker(s)\ra G_0,\alpha:G_0\times \ker(s)\ra \ker(s))$ where $\ker(s)=\{\gamma\in G_1: s(\gamma)=1_{G_0}\}$ and the morphism $\alpha:G_0\times \ker(s)\ra \ker(s)$ is defined as $(a,\gamma)\mapsto 1_a\cdot \gamma\cdot 1_{a^{-1}}$. One can show that the above association defines a one-one correspondence between Lie 2-groups and Lie crossed modules. 
	
	We will denote the Lie 2-group associated to a Lie crossed module $(G, H, \tau, \alpha)$ as $[H \rtimes_{\alpha}G \rra G]$.

Given a Lie 2-group $\mb{G}=[G_1 \rra G_0]$, we can associate the Lie groupoid $L(\mb{G})=[L(G_1) \rra L(G_0)]$ whose structure maps are obtained by taking differentials of the structure maps of $\mb{G}$ at the identity. 

A \textit{right action of a Lie 2-group $\mb{G}=[G_1 \rra G_0]$ on a Lie groupoid }$\mb{X}= [X_1 \rra X_0]$ is given by a morphism of Lie groupoids $\rho: \mb{X} \times \mb{G} \ra \mb{X}$, such that it induces classical Lie group actions on both objects and morphisms. Suppose $\mb{G}$ acts on a pair of Lie groupoids $\mb{X}$ and  $\mb{Y}.$ Then a morphism of Lie groupoids $F:=(F_1, F_0): \mb{X} \rightarrow \mb{Y}$ is said  to be \textit{$\mb{G}$-equivariant }if $F_0$ is $G_0$-equivariant and $F_1$ is $G_1$-equivariant.  A smooth natural transformation between two such functors 
$\eta\colon F\Lrrw F'$  is said to be a \textit{$\mb{G}$-equivariant natural transformation}  if $\eta(x g)=\eta(x)1_g$ for all $x \in X_0$ and $g \in G_0$.
\subsection{Fibered categories and pseudofunctors}\label{subsection fibered categories and pseudofunctors} 
We suggest \cite{MR2223406} for readers interested in further reading on this topic.

A category $\mathcal{E}$ equipped with a functor $\pi: \mathcal{E} \ra \mathcal{X}$ is called a \textit{category over $\mc{X}$.} A morphism $f:x \ra y$ in $\mathcal{E}$ is called \textit{cartesian} if for any morphism $h:z \ra y$ in $\mathcal{E}$ and any morphism $u:\pi(z) \ra \pi(x)$ in $\mathcal{X}$ with $\pi(h)=\pi(f) \circ u$ , there exists a unique morphism $\tilde{u}:z \ra x$ with $\pi(\tilde{u})=u$ and $f \circ \tilde{u}=h$. Observe that if both $\mc{E}$ and $\mc{X}$ are groupoids, then every morphism in $\mc{E}$ is cartesian. 

If $\pi: \mathcal{E} \ra \mathcal{X}$ has the property that for any morphism $\gamma$ in $\mc{X}$ and an element $p \in \mc{E}$ such that $\pi(p)=t(\gamma)$, there is a cartesian morphism $\tilde{\gamma} \in \mc{E}$ satisfying $\pi(\tilde{\gamma})= \gamma$ and $t(\tilde{\gamma})=p$, then it will be called as a \textit{ fibered category over $\mc{X}$}. For each object $x \in \mc{X}$, the \textit{fibre $\pi^{-1}(x)$ of $\pi$ over $x$} is defined as the subcategory of $\mc{E}$, whose objects are the objects $p$ of $\mc{E}$ such that $\pi(p)=x$ and whose morphisms are the morphism $\delta$ of $\mc{E}$ such that $\pi(\delta)=1_x$. 

A \textit{cleavage} on the fibered category $\pi: \mathcal{E} \ra \mathcal{X}$  consists of a class $\kappa$ of cartesian morphisms in $\mc{E}$ such that for each morphism $\gamma : x \ra y$ in $\mc{X}$ and each object $p \in \pi^{-1}(y)$, there exists a unique morphism in $\kappa$ with target $p$, mapping to $\gamma$ in $\mc{X}$. If the cleavage contains all the identities and is closed under composition, then it is called a \textit{splitting cleavage}. It is well known that fibered categories over $\mc{X}$ equipped with a splitting cleavage are in one-one correspondence with category valued contravariant functors over $\mc{X}$. The correspondence naturally extends to a one-one correspondence (due to Grothendieck) between fibered categories over $\mc{X}$ equipped with cleavage and \textit{pseudofunctors over $\mc{X}$} (see\textbf{ sections} \textbf{3.1.2} and \textbf{3.1.3} of \cite{MR2223406} for the proof). We define a pseudofunctor below:
\begin{definition}\label{pseudofunctor}
	A \textit{pseudofunctor $\mc{F} \colon \mc{X}^{\rm{op}} \ra \rm{Cat}$ over a category $\mc{X}$} consists of the  following data:
	\begin{itemize}
		\item[(i)] a category $\mc{F}(x)$ for each $x \in \mc{X}$,
		\item[(ii)] a functor  $\gamma^{*}: \mc{F}(y) \ra \mc{F}(x)$ for each morphism $x \xrightarrow {\gamma} y$  in $\mc{X}$,
		\item[(iii)] for each object  $x$ in $\mc{X}$, we have a  natural isomorphism $I_x: id^{*}_x \Longrightarrow \rm{id}_{\mc{F}(x)}$,
		\item[(iv)] for each pair of composable morphisms 
		\begin{tikzcd}
			x \arrow[r, "\gamma_{1}"] & y \arrow[r, "\gamma_2"] & z
		\end{tikzcd},
		we have a natural isomorphism $\alpha_{\gamma_1,\gamma_2}: \gamma_1^{*} \gamma_2^{*} \ra (\gamma_2 \gamma_1)^{*}$, where the adjacency of $\gamma_2, \gamma_1$ denotes the composition.
	\end{itemize}
	These data satisfy the following coherence laws:
	\begin{itemize}
		\item[(i)] If $x \xrightarrow {\gamma} y$  is a morphism in $\mc{X}$, and $p$ is an object in $F(y)$, then we have 
		\begin{equation}\label{Coherence diagram 1}
			\begin{split}
				& \alpha_{\rm{id}_x, \gamma}(p)= I_{x}(\gamma^{*}(p)) \colon \rm{id}^{*}_{x}(\gamma^{*}(p)) \Longrightarrow \gamma^{*}(p)\\
				& \alpha_{\gamma, \rm{id}_{y}}(p)= \gamma^{*}(I_{y}(p)) \colon \gamma^{*}\rm{id}^{*}_{y}(p) \Longrightarrow \gamma^{*}(p).
			\end{split}
		\end{equation}
		\item[(ii)] For composable morphisms 
		\begin{tikzcd}
			x \arrow[r, "\gamma_3"] & y \arrow[r, "\gamma_2"] & z \arrow[r, "\gamma_1"] & w
		\end{tikzcd}
		and any object $p$ in $F(x)$, the following diagram is commutative:
		\begin{equation}\label{Coherence diagram 2}
			\begin{tikzcd}
				\gamma_3^{*}\gamma_2^{*}\gamma_1^{*}(p) \arrow[d, "\gamma_3^{*}\big( \alpha_{\gamma_2,\gamma_1}(p) \big)"'] \arrow[r, "\alpha_{\gamma_3,\gamma_2}(\gamma_1^{*}(p))"] & (\gamma_2  \gamma_3)^{*}\gamma_1^{*} (p) \arrow[d, "\alpha_{ \gamma_2 \gamma_3,\gamma_1}(p)"] \\
				\gamma_3^{*}(\gamma_1  \gamma_2)^{*} (p) \arrow[r, "\alpha_{\gamma_3,  \gamma_1  \gamma_2 }(p)"']                & (\gamma_1  \gamma_2  \gamma_3)^{*} (p)               
			\end{tikzcd}
		\end{equation}
	\end{itemize}
\end{definition}
The construction of a fibered category from a pseudofunctor is usually known as the \textit{Grothendieck construction }in literature (see \textbf{Chapter 10} of \cite{MR4261588}).
\subsection{Diffeological spaces and some of its properties}\label{Diffeological spaces and some of its properties:}
This subsection will briefly review the notion of a diffeological space and some of its standard properties. Interested readers can find more details on diffeology in \cite{iglesias2013diffeology} and \cite{baez2011convenient}.
\begin{definition}
	A \textit{diffeology} on a set $S$ is a collection of functions $D_{S} \subseteq \lbrace p \colon U \ra S : U \subseteq \mb{R}^{n}$, where $U$ is an open subset of $\mb{R}^{n}, n \in \mb{N} \rbrace$ which staisfy the following conditions:
	\begin{itemize}
		\item[(i)]Every constant function is in $D_S$;
		\item[(ii)] If $V \subseteq \mb{R}^{n}$ is open, $p \colon U \ra S$ is in $D_S$ and $f \colon V \ra U$ is a smooth map, then $p  \circ f \colon V \ra S$ is in $D_S$;
		\item[(iii)] If $\lbrace U_i \rbrace_{i \in I}$ is an open cover of $U \subseteq \mb{R}^{n}$ and $p \colon U \ra S$ is a function such that $p |_{U_i} \colon U_i \ra S$ is in $D_S$ for all  $i \in I$, then $p  \colon U \ra S$ is in $D_S$.
	\end{itemize}
\end{definition}
The pair $(S,D_S)$ is called a \textit{diffeological space} and the elements of $D_S$ are called \textit{plots}. 
\begin{definition}
	A \textit{map of diffeological spaces} from a diffeological space $(X,D_X) $ to a diffeological space  $(Y,D_Y)$ is a map of sets $f \colon X \ra Y$,  such that for any  $p \in D_X$, $f \circ p$ is an element of $D_Y$.
\end{definition}
Collection of diffeological spaces and maps of diffeological spaces between them naturally forms a category, which we denote by \textbf{Diffeol}.
\begin{example}[Smooth manifiold]\label{manifold diffeology}
	Any smooth manifold $M$ has a natural diffeological space structure, given by $D_M:= \lbrace p \colon U \ra M : U$ is an open subset of $\sqcup^{\infty}_{n=0}\mb{R}^{n}$ and $p$ is smooth.$\rbrace$. Any smooth map $f \colon M \ra N$ between manifolds is a map of diffelogical spaces $f \colon (M,D_M) \ra (N, D_N)$.
\end{example}

\begin{example}\label{subsapce diffeology}
	Given a diffeological space $(X, D_{X})$ and a subset $S \subseteq X$, the \textit{subspace diffeology} $D_{S}$ on $S$ is defined as the set $D_{S}:= \lbrace (p \colon U \ra X) \in D_{X} : p(U) \subseteq S \rbrace$.
\end{example}
\begin{example}[Path space diffeology]\label{Path space diffeology}
	For any smooth manifold $M$, the set of paths with sitting instants $PM$ is a diffeological space, equipped with diffeology $D_{PM}:= \lbrace p \colon U \ra PM : \bar{p} \colon U \times [0,1] \ra M, (u,x) \mapsto p(u)(x)$, is smooth.$\rbrace$. $D_{PM}$ is called the \textit{path space diffeology}. The evaluation maps $ev_0,ev_1 \colon PM \ra M$ at $0$ and $1$ respectively, are maps of diffeological spaces.
\end{example}
\begin{example}[Fibre product diffeology]\label{Fibre product diffeology}
	Given maps of diffeological spaces $f \colon (Y,D_Y) \ra (X,D_X)$ and $g \colon (Z,D_Z) \mapsto (X,D_X)$, the fibre product $Y \times_{f,X,g} Z$ is a diffeological space with \textit{fiber product diffeology} $D_{Y \times_{f,X,g} Z}:= \lbrace(p_Y,p_Z) \in D_Y \times D_Z : f \circ p_Y= g \circ p_Z \rbrace$.  Note that the projection maps are the maps of diffeological spaces.
\end{example}
\begin{example}[Quotient diffeology ]\label{quotient diffeology}
	Given a diffeological space $(X, D_{X})$ and an equivalence relation $\sim$ on $X$, the quotient $q \colon X \ra \frac{X}{\sim}$ induces a diffeological structure with the diffeology as given below [\textbf{Construction A.15}, \cite{MR3521476}]:
	
	$D_{\frac{X}{\sim}}:= \lbrace p \colon U \ra \frac{X}{\sim}$: $U \subseteq  \mb{R}^{n}$ is ${\rm{open}}, n \in \mb{N}, p $ is a function such that for every $u \in U$, there is an open neighbourhood $V$ of $u$ in $U$ and a plot $\bar{p} \colon V \ra X$ with $q \circ \bar{p}=p|_{V}\rbrace$. $D_{\frac{X}{\sim}}$ is called the \textit{quotient diffeology}. Then, the quotient map becomes a map of diffeological spaces.
\end{example}
\begin{example}\label{sum diffeology}
	Let $(S_i, D_{S_i})_{i \in I}$ be an arbitrary family of diffeological spaces. Then the disjoint union $S=\sqcup_{i \in I} S_i$ is a diffeological space with the diffeology $D:= \lbrace p \colon U \ra  S: U \subseteq  \mb{R}^{n}$ is ${\rm{open}}, n \in \mb{N}, p $ is a function such that for any $x \in U$ there exists an open neighbourhood $U_x$ of $x$ and an index $i \in I$, with $P|_{U_x} \in D_{S_i}.  \rbrace$. The diffeology $D$ is called the \textit{sum diffeology} on the family $\lbrace S_{i} \rbrace_{i 
	\in I}$, (see \textbf{Section 1.39} of \cite{iglesias2013diffeology}).
	\end{example}
\begin{lemma}\label{Technical 1}
	Let $q \colon (A,D_A) \ra (B,D_B)$ be a quotient map between two diffeological spaces and $(C,D_C)$   another diffeological space. Then a map $f \colon (B,D_B) \ra (C,D_C)$ is a map of diffeological spaces if and only if for any plot $p \colon U \ra A$, the composite $f \circ q \circ p  \in D_C$.
\end{lemma}
\begin{proof}
	See \textbf{Lemma A.16} in \cite{MR3521476}.
\end{proof}
\subsection{A principal Lie 2-group bundle over a Lie groupoid}\label{subsection:A principal Lie 2-group bundle over a Lie groupoid}
In \cite{chatterjee2022atiyah}, we introduced a notion of a principal Lie 2-group bundle over a Lie groupoid and characterized a particular subfamily of them. In this subsection, we give a quick review of such.

\begin{definition}[Definition 3.6 , Definition 3.8, \cite{chatterjee2022atiyah}]\label{Definition Principal Lie 2-group bundle over a Lie groupoid}
	Let $\mb{G}$ be a Lie 2-group.
	\begin{itemize}
		\item[(i)] A \textit{principal $\mb{G}$-bundle over a Lie groupoid} $\mathbb{X}$ is given by a morphism of Lie groupoids $\pi: \mb{E} \rightarrow \mb{X}$ along with a right action $\rho: \mb{E} \times \mb{G} \rightarrow \mb{E}$ of Lie $2$-group $\mb{G}$ on $\mb{E}$ such that $\pi_0: E_0 \ra X_0$ and $\pi_1: E_1 \ra X_1$ are principal $G_0$-bundle and principal $G_1$-bundle respectively.
		\item[(ii)] A\textit{ morphism  of principal $\mb{G}$-bundles from $\pi: \mb{E} \ra \mb{X}$ to $\pi': \mb{E}' \ra \mb{X}$} is given by a smooth $\mb{G}$-equivariant morphism $F:\mb{E} \rightarrow \mb{E'}$ such that the following diagram commutes on the nose
		\[
		\begin{tikzcd}
			\mb{E} \arrow[r, "F"] \arrow[d, "\pi"'] & \mb{E'} \arrow[ld, "\pi'"] \\
			\mb{X}                               &                  
		\end{tikzcd}.\]
	\end{itemize}
\end{definition}
The Lie 2-group $\mb{G}$ above is said to be the \textit{structure 2-group of the principal $\mb{G}$-bundle.} Several examples of such principal 2-bundles have been given in \cite{chatterjee2022atiyah}. 

In particular when the structure 2-group is a discrete 2-group $[G \rra G]$, then a principal $[G \rra G]$-bundle over a Lie groupoid $\mb{X}$ coincides with a principal $G$-bundle over $\mb{X}$ (as in \cite{MR2270285, MR2119241, MR1950948}), that we showed in \textbf{Example 3.14}, \cite{chatterjee2022atiyah}.
\begin{definition}[Definition 2.2,\cite{MR2270285}]\label{Definition: Principal Lie group bundle over a Lie groupoid}
	For a Lie group $G$, a \textit{principal $G$-bundle over a Lie groupoid} $\mb{X}$ is given by a principal $G$-bundle $\pi\colon E_G \rightarrow X_0$ along with a  smooth 
	map $\mu\colon s^{*}E_G:= (X_1 \times_{s, X_0, \pi} E_G) \rightarrow E_G$, which satisfy the following conditions:
	
	\begin{itemize}
		\item[(i)] $\mu(1_{\pi(p)}, p)=p$ for all $p \in E_G$,
		\item[(ii)] for each $(\gamma, p) \in s^{*}E_G $, we have $\bigl(\gamma, \mu(\gamma,p)\bigr) \in t^{*}E_G$,
		\item[(iii)] if $\gamma_2 , \gamma_1 \in X_1$ such that $t(\gamma_1)=s(\gamma_2)$ and $(\gamma_1,p) \in s^{*}E_G$, then 
		$\mu(\gamma_2 \circ \gamma_1,p)=\mu(\gamma_2,\mu(\gamma_1,p))$,
		\item[(iv)] for all $p \in E_G, g \in G$ and $\gamma \in X_1$ we have $\mu(\gamma, p)g=\mu(\gamma, pg).$			
	\end{itemize}
\end{definition}
In particular the conditions(i)-(iii)  tells that $\mu$ is an \textit{action} of $\mb{X}$ on $E_G$. The fourth condition tells that it commutes with the right $G$-action on $E_G$. We denote a principal $G$-bundle over the Lie groupoid $\mathbb{X}= [X_1 \rightrightarrows X_0]$  by  $\bigl(\pi\colon E_G \rightarrow X_0, \mu, \mb{X} \bigr)$. Given a Lie group $G$ and a Lie groupoid $\mb{X}$, the collection of principal $G$-bundles over $\mb{X}$ form a groupoid which we denote by ${\rm{Bun}}(\mb{X},G)$. An element in ${\rm{Hom}} \bigl((\pi\colon E_G \rightarrow X_0, \mu, \mb{X}), (\pi'\colon E'_G \rightarrow X_0, \mu', \mb{X})  \bigr)$ is a morphism of principal $G$-bundles $f \colon E_G \ra E'_G$ over $X_0$ such that $f(\mu(\gamma,p))= \mu'(\gamma,f(p))$ for all $(\gamma,p) \in s^{*}E_G$. 

For a Lie crossed module $(G,H,\tau,\alpha)$ and a principal $G$-bundle over a Lie groupoid $\mb{X}$, we constructed in [Proposition 3.18, \cite{chatterjee2022atiyah}], a particular class of principal $[H \rtimes_{\alpha} G \rra G]$-bundle over $\mb{X}$ called \textit{decorated principal 2-bundles} and have extensively studied. Below, we briefly recall the construction:
\begin{proposition}[Proposition 3.18 ,\cite{chatterjee2022atiyah}]\label{Prop:Decoliegpd}
	Given a Lie crossed module $(G, H, \tau, \alpha)$ and a principal $G$-bundle $\bigl(\pi\colon E_G \rightarrow X_0, \mu, \mb{X} \bigr)$ over $\mb{X}$ we have the following:
	\begin{enumerate}
		\item the manifolds $(s^{*}E_G)^{\rm{dec}}:=s^{*}E_G \times H$ and $E_G$ determines a Lie groupoid $[(s^{*}E_G)^{\rm{dec}} \rightrightarrows E_G]$ whose structure maps are given as 
		\begin{itemize}
			\item[(i)] source map $s$: $(\gamma, p, h) \mapsto p$,
			\item[(ii)] target map $t$: $(\gamma, p, h) \mapsto \mu(\gamma, p) \tau(h^{-1})$,
			\item[(iii)] composition map $m \colon \big((\gamma_2, p_2, h_2), (\gamma_1, p_1, h_1) \big) \mapsto (\gamma_2 \circ \gamma_1, p_1 ,h_2h_1)$, 
			\item[(iv)] unit map $u : p \mapsto (1_{\pi(p)},p,e)$,
			\item[(v)] inverse map $ \mathfrak{i} \colon \bigl(\gamma, p, h) \mapsto (\gamma^{-1}, \mu(\gamma,p)\tau(h^{-1}), h^{-1}\bigr)$.
		\end{itemize}
		\item The Lie groupoid  $\mb{E}^{\rm{dec}}:=[(s^{*}E_G)^{\rm{dec}} \rightrightarrows E_G]$ forms a principal $\mb{G}$-bundle over $\mb{X}$ such that the action of the Lie 2-group $[H \rtimes_{\alpha}G \rra G]$ and bundle projection are respectively defined as
		\begin{equation}\label{E:Actionondeco}
			\begin{split}
				\rho\colon &\mb{E}^{\rm dec}\times \mb{G}\ra \mb{E}^{\rm dec}\\
				&(p, g) \mapsto p\, g\\
				\bigl((\gamma, p, h)&, (h', g)\bigr)\mapsto \bigl(\gamma, p g, \alpha_{g^{-1}}(h'^{-1}\, h)\bigr),
			\end{split}
		\end{equation}
		and
		\begin{equation}\label{E:Projondeco}
			\begin{split}
				\pi^{\rm{dec}} \colon &\mb{E}^{\rm dec}\ra \mb{X}\\
				& p \mapsto \pi(p),\\
				\bigl(\gamma,&  p, h\bigr) \mapsto \gamma.
			\end{split}
		\end{equation}
	\end{enumerate}
\end{proposition}
The principal $[H \rtimes_{\alpha}G \rra G]$-bundle $ \pi^{\rm{dec}} \colon \mb{E}^{\rm{dec}}\ra \mb{X}$ is said to be the \textit{decorated principal $\mb{G}$-bundle associated to $(\pi \colon E_G\ra X_0, \mu, \mb{X})$ and the  Lie crossed module $(G,H, \tau, \alpha)$}. They can be characterzed as prinicipal 2-bundles admitting  \textit{categorical connections} (\cite{chatterjee2022atiyah},\textbf{Proposition 3.25}), a notion we recall below:
\begin{definition}[Definition 3.21,\cite{chatterjee2022atiyah}]\label{Def:categorical connection}
	Let $\mb{G}$ be a Lie 2-group and $\pi\colon \mb{E} \ra \mb{X}$ a principal $\mb{G}$-bundle over $\mb{X}$. A \textit{categorical connection} $\mathcal{C}$ on $\pi: \mb{E} \ra \mb{X}$  is defined as a smooth map $\mathcal{C}\colon {s}^{*}E_0 \ra E_1$ which satisfies the following conditions:
	\begin{enumerate}[(i)]
		\item $s(\mathcal{C}(\gamma,p))=p$  for all $(\gamma,p) \in s^{*}E_0$;
		\item $\pi_1(\mathcal{C}(\gamma,p))= \gamma$ for all $(\gamma,p) \in s^{*}E_0$;
		\item $\mathcal{C}(\gamma, p g)= \mathcal{C}(\gamma, p) 1_g$ for all $(\gamma, p) \in {s}^{*}E_0$ and $g \in G_0$;
		\item $\mathcal{C}(1_x,p)=1_p$  for any $x\in X_0$ and $p\in \pi^{-1}(x)$;
		\item if $(\gamma_2, p_2), (\gamma_1, p_1) \in {s}^{*}E_0$ such that ${s}(\gamma_2)={t}(\gamma_1)$ and $p_2=t\bigl({\mathcal C}(\gamma_1, p_1)\bigr),$ then $\mathcal{C}(\gamma_2 \circ \gamma_1 , p_1)= \mathcal{C}(\gamma_2, p_2) \circ \mathcal{C}(\gamma_1, p_1)$.
	\end{enumerate}

\end{definition}

\begin{example}[Corollary 3.26, \cite{chatterjee2022atiyah}]\label{Cat connection over discrete base}
	Let $\mb{G}$ be a Lie 2-group. Any principal $\mb{G}$-bundle $\pi \colon \mb{E} \ra [M \rra M]$ over a discrete Lie groupoid $[M \rra M]$ admits a unique categorical connection given by  $(1_x, p)\mapsto 1_p$ for $p\in E_0, x=\pi(p).$
\end{example}
\subsection{Connection structures on a principal 2-bundle over a Lie groupoid}\label{subsection:Connection structures on a principal 2-bundle over a Lie groupoid}

In this subsection, we briefly recall the connection structures on a principal 2-bundle over a Lie groupoid (\Cref{Definition Principal Lie 2-group bundle over a Lie groupoid}), that we introduced in  \cite{chatterjee2022atiyah}.
\begin{proposition}[Proposition 3.2,\cite{chatterjee2022atiyah}]\label{Associated action on Lie 2-algebra and Tangent 2space}
	For a Lie 2-group $\mb{G}$, we have the following:
	\begin{itemize}
		\item[(i)] There is an action of $\mb{G}$ on the Lie groupoid $L(\mb{G})=[L(G_1) \rra L(G_0)]$ given by the adjoint action.
		\item[(ii)]  Suppose there is an action of $\mb{G}$ on a Lie groupoid $\mb{X}$. Then there is an action of $\mb{G}$ on the tangent Lie groupoid $T\mb{X}:=[TX_1 \rra TX_0]$, given by the differential of the action.
	\end{itemize}
	
\end{proposition}
\begin{definition}[Definition 5.4,\cite{chatterjee2022atiyah} ]
	\label{Definition:LGvaluedformOnLiegroupoid}
	Let $\mb{G}=[G_1\rra G_0]$ be a Lie $2$-group, and $\mb{E}=[E_1\rra E_0]$ a Lie groupoid. An \textit{$L(\mb{G})$-valued $1$-form on the Lie groupoid $\mb{E}$} is a morphism of Lie groupoids 
	$\omega:=(\omega_1,\omega_0)\colon T\mb{E} \rightarrow L(\mb{G})$ such that  $\omega_i$ is an $L(G_i)$-valued differential $1$-form on $E_i,$ for $i\in\{0, 1\}.$
	If $\mb{G}$ acts on $\mb{E}$ and $\omega\colon T\mb{E} \ra L(\mb{G})$ is  $\mb{G}$-equivariant with respect to the actions defined in \Cref{Associated action on Lie 2-algebra and Tangent 2space}, then $\omega$ is called a \textit{$\mb{G}$-equivariant $1$-form}.
\end{definition}
\begin{definition}[Proposition 5.11, \cite{chatterjee2022atiyah}]\label{strict and semistrict connections}
	For a Lie 2-group $\mb{G}$, let $\pi: \mb{E} \ra \mb{X}$ be a principal $\mb{G}$-bundle over a Lie groupoid $\mb{X}$. A $\mb{G}$-equivariant 1-form $\omega: T\mb{E} \ra L(\mb{G})$ on $\mb{E}$ is defined as a \textit{strict connection} (resp. \textit{semistrict connection}) if the following diagram of morphisms of Lie groupoids
	\begin{equation}\nonumber
		\begin{tikzcd}
			T{\mb{E}} \arrow[r, "\omega"]                  & L(\mb{G}) \\
			\mb{E} \times L(\mb{G}) \arrow[u, "\delta"] \arrow[ru, "\rm{pr_2}"'] &  
		\end{tikzcd},
	\end{equation}
	commutes on the nose (resp. upto a $\mb{G}$-equivariant, fibre-wise linear natural isomorphism). Here, $\delta$ and $\rm{pr_2}$ are, respectively, the functor induced by the vertical vector fields and the 2nd projection functor.
\end{definition}

As we see next, these connection structures behave well with the pullback along a morphism of principal 2-bundles over a Lie groupoid.
\begin{example}[Lemma 5.27, \cite{chatterjee2022atiyah}]\label{Lemma:Pullback connection}
	Given a Lie 2-group $\mb{G}$ and a Lie groupoid $\mb{X}$, let  $\pi \colon \mb{E} \ra \mb{X}$ and $\pi' \colon \mb{E}' \ra \mb{X}$ be a pair of principal $\mb{G}$ bundles over a Lie groupoid $\mb{X}$. Suppose $F:=(F_1, F_0): \mb{E} \ra \mb{E}'$ be a morphism of prinicpal $\mb{G}$-bundles over $\mb{X}$.  If $\omega:=(\omega_1,\omega_0):T\mb{E}' \ra L(\mb{G})$ is a strict connection on $\mb{E}'$ then $F^{*}\omega:=(F_1^{*}\omega_1, F_0^{*}\omega_0):T \mb{E} \ra L(\mb{G})$ is a strict connection on the principal $\mb{G}$-bundle $\pi \colon \mb{E} \ra \mb{X}$. 
\end{example}
We call $F^{*}\omega$ above as the \textit{pullback connection of $\omega$ along $F$}.
\begin{example}[\textbf{Example 5.16}, \cite{chatterjee2022atiyah}]\label{Classical as 2-connection}
		For a Lie group $G$, let $P\ra M$ be a classical principal $G$-bundle over a manifold $M.$  Then a  connection 1-form $\omega$ on $P\ra M$ defines a strict connection 1-form $(\omega, \omega)$  on the principal Lie $2$-group $[G\rra G]$-bundle $[P\rra P]\ra [M\rra M]$ over  the Lie groupoid $[M\rra M].$
\end{example}
\begin{example}[\textbf{Proposition 5.24}, \cite{chatterjee2022atiyah}]\label{Prop:ConOnDeco}
	Let $\bigl(\pi\colon E_G \rightarrow X_0, \mu, \mb{X} \bigr)$ 
	be a principal $G$-bundle over the Lie groupoid $\mb{X}$ and $\omega$ a connection on the principal $G$-bundle $\pi \colon E_G \ra X_0$ such that $s^{*}\omega=t^{*}\omega$. Then, for a Lie crossed module $\ghta$, the pair $(\omega^{\rm dec}, \omega)$ is a strict connection $1$-form on principal $[H \rtimes_{\alpha} G \rra G]$-bundle $\pi^{{\rm{dec}}} \colon  \mb{E}^{\rm dec} \ra \mb{X}$ (\Cref{Prop:Decoliegpd}), where $\omega^{\rm dec}$  is defined as $\omega^{\rm dec}{(\gamma,\, p, h)}=\ad_{(h, e)}\bigl((s^*\omega){(\gamma, p, h)}\bigr)-{\Theta}_h,$ and $\Theta_h$ is the Maurer-Cartan form on $H$.

\end{example}
We refer to our previous paper \cite{chatterjee2022atiyah} for several other examples of principal Lie 2-group bundles over Lie groupoids and connections.
\begin{remark}
	Originally in [\textbf{Definition 5.1}, \cite{chatterjee2022atiyah}],  strict and semistrict connections on a principal 2-bundle over a Lie groupoid were described in terms of the splitting of a short exact sequence of VB-groupoids called the \textit{Atiyah sequence associated to the principal 2-bundle}, which is equivalent to the \Cref{strict and semistrict connections} [\textbf{Proposition 5.11}, \cite{chatterjee2022atiyah}].
\end{remark}

\section{Quasi-principal 2-bundles and their characterizations}\label{quasi-principal 2-bundles and their characterisations}
In this section, we introduce a a \textit{quasi-principal 2-bundle over a Lie groupoid} (\Cref{Definition:Quasicategorical Connection}) and a \textit{pseudo-principal Lie crossed module-bundle over a Lie groupoid} (\Cref{pseudo principal Lie crossed module}). We show that the respective categories  are equivalent (\Cref{Grothendieck construction}) via a Lie 2-group torsor version of the classical Grothendieck construction (\Cref{subsection fibered categories and pseudofunctors}). Consequently we also extend a class of principal 2-bundle over a Lie groupoid to be defined over a differentiable stack (\Cref{Morita})

%

\subsection{Quasi-principal 2-bundles}
Let $\mb{G}:=[G_1 \rra G_0]$ be a Lie 2-group. Given a principal $\mb{G}$-bundle $\pi \colon \mb{E} \ra \mb{X}$ over a Lie groupoid $\mb{X}$, there is a canonical morphism $P \colon s^{*}E_0 \ra E_1$ of principal bundles, from the pull-back principal $G_0$-bundle $\pi_0^{*} \colon s^{*}E_0 \ra X_1$ to the principal $G_1$-bundle $\pi_1 \colon E_1 \ra X_1$ given as $\delta \mapsto (\pi_1(\delta), s(\delta)).$ Adapting same notation as above, we define the following:
\begin{definition}\label{Definition:Quasicategorical Connection}
	A \textit{quasi connection} on a principal $\mb{G}$-bundle $\pi: \mb{E} \ra \mb{X}$ is defined as a smooth section $\mc{C}: s^{*}E_0 \ra E_1$ of the morphism of principal bundles $P : E_1 \ra s^{*}E_0$, such that $\mc{C}$ is itself a morphism of principal bundles over $X_1$ along the unit map $u \colon G_0 \ra G_1$. We will call the pair $(\pi: \mb{E} \ra \mb{X}, \mc{C})$, a \textit{quasi principal $\mb{G}$-bundle over $\mb{X}$}.
	
	\[
	\begin{tikzcd}
		E_1 \arrow[d, "\pi_1"'] \arrow[r, "P", bend left] & s^{*}E_0 \arrow[ld, "\pi_0^{*}", bend left] \arrow[l, "\mc{C}"', bend left=49] \\
		X_1                                          &                                                           
	\end{tikzcd}\]
\end{definition}
%
An analogous notion to a quasi connection in the setup of VB-groupoids can be found in \cite{MR3696590} and \cite{MR4126305}. One main distinguishing feature of our setup is the Lie 2-group equivariance. In fact, both of these definitions are special cases of a "cleavage" (a smooth version of the one mentioned in \Cref{subsection fibered categories and pseudofunctors}) in the setup of Lie groupoid fibrations \cite{MR3968895}.

The following proposition is obvious:
\begin{proposition}\label{Proposition: Quasi-Cat}
	Let $\pi \colon \mb{E} \ra \mb{X}$ be a principal $\mb{G}$-bundle over a Lie groupoid $\mb{X}$. Every categorical connection $\mc{C} \colon s^*E_0 \ra E_1$ is a quasi connection. Conversely, any quasi connection $\mc{C} \colon s^{*}E_0 \ra E_1$, which satisfies the following two properties:
	\begin{enumerate}[(i)]
		\item $\mathcal{C}(1_x,p)=1_p$  for any $x\in X_0$ and $p\in \pi^{-1}(x)$,
		\item if $(\gamma_2, p_2), (\gamma_1, p_1) \in {s}^{*}E_0$ such that ${s}(\gamma_2)={t}(\gamma_1)$ and $p_2=t\bigl({\mathcal C}(\gamma_1, p_1)\bigr),$ then $\mathcal{C}(\gamma_2 \circ \gamma_1 , p_1)= \mathcal{C}(\gamma_2, p_2) \circ \mathcal{C}(\gamma_1, p_1)$,
	\end{enumerate}
	is a categorical connection \Cref{Def:categorical connection}.
\end{proposition}
\begin{definition}\label{Unital connection}
	A quasi connection $\mc{C} \colon s^{*}E_0 \ra E_1$ satisfying (i) of \Cref{Proposition: Quasi-Cat} will be called a \textit{unital connection} and a principal 2-bundle equipped with an unital connection will be called a \textit{unital-principal 2-bundle}. In the same way, a principal 2-bundle equipped with a categorical connection will be called a \textit{categorical-principal 2-bundle}. \\ We do not notationally distinguish between quasi, unital or categorical-principal 2-bundles. 
\end{definition}
It is evident that the category of quasi-principal 2-bundles over a Lie groupoid forms a groupoid.
\begin{proposition}\label{Groupoid of quasi principal 2-bundles}
	Given a Lie 2-group $\mb{G}$ and a Lie groupoid $\mb{X}$, the category $\rm{Bun}_{\rm{quasi}}(\mb{X}, \mb{G})$ whose objects are quasi-principal $\mb{G}$-bundles $(\pi \colon \mb{E} \ra \mb{X}, \mc{C})$ over $\mb{X}$ and an arrow from $(\pi \colon \mb{E} \ra \mb{X}, \mc{C})$ to $(\pi' \colon \mb{E}' \ra \mb{X}, \mc{C}')$ is a morphism of principal $\mb{G}$-bundles $F \colon \mb{E} \ra \mb{E}'$ satisfying $F_1(\mc{C}(\gamma,p))= \mc{C}'(\gamma, F_0(p))$ for all $(\gamma,p) \in s^{*}E_0$, forms a groupoid.
	Similarly, unital principal $\mb{G}$-bundles and categorical principal $\mb{G}$-bundles over $\mb{X}$ forms the respective groupoids $\rm{Bun}_{\rm{unital}}(\mb{X}, \mb{G})$ and $\rm{Bun}_{\rm{Cat}}(\mb{X}, \mb{G})$.
\end{proposition}
In the same spirit, we propose a weaker version of \Cref{Definition: Principal Lie group bundle over a Lie groupoid}:
\begin{definition}\label{quasi-principal G-bun dle}
	For a Lie group $G$, a \textit{quasi-principal $G$-bundle over a Lie groupoid} $\mb{X}$ is given by a principal $G$-bundle $\pi\colon E_G \rightarrow X_0$ along with a  smooth 
	map $\mu\colon s^{*}E_G \rightarrow E_G$ which satisfy the following conditions:
	\begin{enumerate}[(i)]
		\item for each $(\gamma, p) \in s^{*}E_G$, we have $\bigl(\gamma, \mu(\gamma,p)\bigr) \in X_1 \times_{t, X_0, \pi}E_G$,
		\item for all $p \in E_G, g \in G$ and $\gamma \in X_1$ we have $\mu(\gamma, p)g=\mu(\gamma, pg).$			
	\end{enumerate}
\end{definition}
Depending on the context, $\bigl(\pi\colon E_G \rightarrow X_0, \mu, \mb{X} \bigr)$ may either denote a quasi-principal $\mb{G}$-bundle or a principal $G$-bundle (\Cref{Definition: Principal Lie group bundle over a Lie groupoid}).

A similar notion in the linear framework appeared in \cite{MR3107517} and \cite{MR3696590}  by the name \textit{quasi-action of a Lie groupoid on a vector bundle}.

\begin{example}\label{underlying quasi-principal bundle}
	Given a quasi-principal $\mb{G}$-bundle $(\pi \colon \mb{E} \ra \mb{X},\mc{C})$ over a Lie groupoid $\mb{X}$,  $(\pi_0 \colon E_0 \ra X_0, \mu_{\mc{C}}:=t \circ \mc{C}, \mb{X})$ is a quasi-principal $G_0$-bundle over $\mb{X}$, which we call the \textit{underlying quasi-principal $G_0$-bundle of the quasi-principal $\mb{G}$-bundle $\pi \colon \mb{E} \ra \mb{X}$.}
\end{example}
\subsection{Examples of quasi-principal 2-bundles}\label{Examples of quasiprincipal 2-bundles}
By \Cref{Proposition: Quasi-Cat}, any categorical principal 2-bundle is a quasi-principal 2-bundle. Here, we construct under certain conditions, non-trivial examples of quasi-principal 2-bundles which fail to be categorical-principal 2-bundles.

%
%
%
\begin{lemma}\label{QuasiExamples}
 Let $(\pi \colon \mb{E} \ra \mb{X}, \mc{C})$ be a categorical-principal $[H \rtimes_{\alpha} G \rra G]$-bundle over a Lie groupoid $\mb{X}$. If there exists a smooth map $\mc{H} \colon s^{*}E_0 \ra H$ such that it satisfies $\alpha_{g}(\mc{H}(\gamma,pg))= \mc{H}(\gamma,p)$ for all $(\gamma,p) \in s^{*}E_0$ and $g \in G$, then for $\mc{C}_{\mc{H}}:= \mc{C}(\gamma,p) \big( \mc{H}(\gamma,p) ,e\big)$,  $(\pi \colon \mb{E} \ra \mb{X},\mc{C}_{\mc{H}})$ is a quasi-principal $[H \rtimes_{\alpha} G \rra G]$-bundle over $\mb{X}$. Moreover, $\mc{C}_{\mc{H}}$ is a categorical connection if and only if we have
	\begin{itemize}
		\item[(i)] $\mc{H}\big(1_{\pi(p)},p \big)=e$ for all $p \in E_0$ and 
		\item[(ii)] $\mc{H}(\gamma_2 \circ \gamma_1,p)= \mc{H}(\gamma_2,t(\mc{C}(\gamma_1,p))) \mc{H}(\gamma_1,p_1)$ for all $\gamma_2,\gamma_1 \in X_1$, such that $s(\gamma_2)=t(\gamma_1)$ and $(\gamma_1,p) \in s^{*}E_0$.
	\end{itemize}
\end{lemma}
\begin{proof}
	Since, for $(\gamma,p) \in s^{*}E_0, g \in G$, we have $\pi \Big(\mc{C}(\gamma,p) \big( \mc{H}(\gamma,p) ,e\big), s \Big( \mc{C}(\gamma,p) \big( \mc{H}(\gamma,p) ,e\big) \Big) \Big)= (\gamma,p)$ and $\mc{C}_{\mc{H}}(\gamma,pg) =\mc{C}(\gamma,p)\Big(\alpha_{g}\big(\mc{H}(\gamma,pg)  \big),g\Big)$, it follows immediately that $(\pi \colon \mb{E} \ra \mb{X}, \mc{C}_{\mc{H}})$ is a quasi-principal $[H \rtimes_{\alpha} G \rra G]$-bundle over $\mb{X}$. It is easy to verify that $\mc{C}_{\mc{H}}$ is a categorical connection if and only if (i) and (ii) holds.

\end{proof}
Using \Cref{QuasiExamples}, next, we will construct some concrete examples of quasi-principal 2-bundles, which are not categorical-principal 2-bundles. 

\begin{example}\label{concrete family of example}
	Let $\pi^{\rm{dec}} \colon \mb{E}^{\rm{dec}} \ra \mb{X}$ be the decorated principal $[H \rtimes_{\alpha}G \rra G ]$-bundle over a Lie groupoid $\mb{X}$, obtained from a Lie crossed module  $(G,H,\tau,\alpha)$ and a principal $G$-bundle $(\pi_G \colon E_G \ra X_0, \mu, \mb{X})$, such that there is a non-identity element $h$ in $H$ satisfying $\alpha(g)(h)=h$ for all $g \in G$. Define a map $\mc{H} \colon  s^{*}E_0 \ra H$ by $(\gamma,p) \mapsto h$ for all $(\gamma,p) \in s^{*}E_0$. Since, the assignment $(\gamma,p) \mapsto (\gamma,p,e)$ for all $(\gamma,p) \in s^{*}E_G$ defines a categorical connection on $\pi^{\rm{dec}} \colon \mb{E}^{\rm{dec}} \ra \mb{X}$, it follows immediately from \Cref{QuasiExamples}   $\mc{C}_h \colon s^{*}E_G \ra s^{*}E_G \times H$ defined by $(\gamma,p) \mapsto (\gamma,p, e)(h,e)$ is a quasi connection on $\pi^{\rm{dec}} \colon \mb{E}^{\rm{dec}} \ra \mb{X}$. As $h \neq e$, $\mc{C}_h$ is not a categorical connection.
\end{example}
As a special case of \Cref{concrete family of example}, we obtain the following example:

\begin{example}\label{Hquasi}
	Let $\mb{X}$ be a Lie groupoid. Observe that the identity map $ {\rm{id}} \colon X_0 \ra X_0$ is a principal $\lbrace e \rbrace$-bundle over $X_0$ under the natural action of the trivial Lie group $\lbrace e \rbrace $. Now define 
	\begin{equation}\nonumber
		\begin{split}
			& \mu \colon s^{*} X_0 \ra X_0\\
			& (\gamma,p) \mapsto t(\gamma).
		\end{split}
	\end{equation}
	Clearly, $({\rm{id}} \colon X_0 \ra X_0, \mu, \mb{X})$ defines a principal $\lbrace e \rbrace$-bundle over $\mb{X}$. For an abelian Lie group $H \neq \lbrace e \rbrace$, consider the decorated principal $[H \rra \lbrace e \rbrace]$-bundle over $\mb{X}$, obtained from the Lie crossed module  $(\lbrace e \rbrace, H, \tau ,\alpha)$ (where $\tau$ is trivial and $\alpha$ is ${\rm{id}}_{H}$) and the principal  $\lbrace e \rbrace$-bundle $({\rm{id}} \colon X_0 \ra X_0, \mu, \mb{X})$. Since $H$ is not trivial and $\alpha$ is ${\rm{id}}_H$, it follows from \Cref{concrete family of example} that for any non-identity $h$ in $H$, the map $\mc{C}_{h} \colon s^{*}X_0 \ra X_0$ given by $(\gamma,p) \mapsto (\gamma,p,e)(h,e)$ is a quasi connection which is not a categorical connection.
	
\end{example}
\begin{example}\label{qusi bundle over discrete groupoid}
 Consider a principal $[H \rtimes_{\alpha}G \rra G]$-bundle $\pi \colon \mb{E} \ra [M \rra M]$ over a discrete Lie groupoid $[M \rra M]$, such that there exists $h \in H$, $h \neq e$ and $\alpha(g)(h)=h$ for all $g \in G$. Then  it follows from \Cref{QuasiExamples} that the map $\mc{C}_{h} \colon s^{*}E_0 \ra E_1, (1_x,p) \mapsto 1_p(h,e)$ defines a quasi connection, which is not a categorical connection. Hence, contrast to the existence of a unique categorical connection (see \Cref{Cat connection over discrete base}), it may admit many quasi connections. 	

\end{example}
\subsection{A quasi-principal 2-bundle as a Grothendieck Construction}\label{Grothendieck construction subsection}
In this subsection, we will obtain the paper's first main result (\Cref{Grothendieck construction}). We begin by observing some properties of the underlying quasi-principal Lie group bundle of a quasi-principal Lie 2-group bundle (\Cref{underlying quasi-principal bundle}). 
\begin{proposition}\label{Lemma: cohenrence of canonical quasi action}
 Let $(\pi \colon  \mb{E} \ra \mb{X}, \mc{C})$ be a quasi-principal $[H \rtimes_{\alpha}G \rra G]$-bundle over a Lie groupoid $\mb{X}$. Consider the underlying quasi-principal $G$-bundle $(\pi_0 \colon E_0 \ra X_0, \mu_{\mc{C}}:=t \circ \mc{C}, \mb{X})$ over $\mb{X}$.
	Then there exist smooth maps $\mc{H}_{u,\mc{C}} \colon E_0 \ra H$ and $\mc{H}_{m ,\mc{C}} \colon X_1 \times_{s,X_0,t} X_1 \ra H$ satisfying the following properties:
	\begin{itemize}
		\item[(a)] $\mu_{\mc{C}}(1_{\pi(p)},p)=p \tau(\mc{H}_{u,\mc{C}}(p))$ for all $p \in E_0$.
		\item[(b)] $\mu_{\mc{C}}(\gamma_2, \mu_{\mc{C}}(\gamma_1,p))= \mu_{\mc{C}}(\gamma_2 \circ \gamma_1,p) \tau(\mc{H}_{m,\mc{C}}(\gamma_2, \gamma_1))$ for all appropriate $\gamma_2, \gamma_1 \in X_1, p \in E_0$.
		\item[(c)] [\textit{Right unitor}] $\mc{H}_{m,\mc{C}}(\gamma, 1_{\pi(p)})=  \mc{H}_{u,\mc{C}}(p)$ for all $\gamma \in X_1$ such that $s(\gamma)= \pi(p)$.
		\item[(d)][\textit{Left unitor}] $\mc{H}_{m,\mc{C}}(1_{\pi(\mu_{\mc{C}}(\gamma,p))},\gamma)=  \mc{H}_{u, \mc{C}}(\mu_{\mc{C}}(\gamma,p))$ for $(\gamma,p) \in s^{*}E_0$.
		\item[(e)] $\mc{H}_{u,\mc{C}}$ is $G$ invariant.
		\item[(f)] $\alpha_{g^{-1}}(\mc{H}_{u,\mc{C}}(p))=\mc{H}_{u,\mc{C}}(p)$ for all $g \in G$ and $p \in E_0$.
		\item[(g)] $\mc{H}_{u,\mc{C}}(p) \in Z(H)$ for all $p \in E_0$, where $Z(H)$ is the centre of $H$.
		\item[(h)] $\alpha_{g^{-1}}(\mc{H}_{m,\mc{C}}^{-1}(\gamma_2, \gamma_1))= \mc{H}_{m,\mc{C}}^{-1}(\gamma_2, \gamma_1)$ for all composable $\gamma_2, \gamma_1 \in X_1$.
		\item[(i)]  $\mc{H}_{m,\mc{C}}(\gamma_2, \gamma_1) \in Z(H)$ for all $\gamma_2, \gamma_1 \in X_1 \times_{s,X_0,t} X_1$.
		\item[(j)] [\textit{Associator}]For $\gamma_3, \gamma_2, \gamma_1 \in X_1$ such that $s(\gamma_3)= t(\gamma_2)$ and $s(\gamma_2)= t(\gamma_1)$, we have 
		$$\mc{H}_{m,\mc{C}}^{-1}(\gamma_3,\gamma_2) \mc{H}_{m,\mc{C}}^{-1}(\gamma_3 \circ \gamma_2, \gamma_1)= \mc{H}_{m,\mc{C}}^{-1}(\gamma_2, \gamma_1) \mc{H}_{m,\mc{C}}^{-1}(\gamma_3, \gamma_2 \circ \gamma_1).$$
		\item[(k)][\textit{Invertor}] If $(\gamma,p) \in s^{*}E_0$, then we have
		$$\mc{H}_{m,\mc{C}}({\gamma^{-1}, \gamma}) \mc{H}_{m,\mc{C}}(\gamma,\gamma^{-1})^{-1}= H_{u,\mc{C}}(p)^{-1}H_{u,\mc{C}}(\mu_{\mc{C}}(\gamma,p)).$$
	\end{itemize}
\end{proposition}
\begin{proof}
	Define $\mc{H}_{u,\mc{C}} \colon E_0 \ra H$  by $p \mapsto h_p$ and $\mc{H}_{m,\mc{C}} \colon X_1 \times_{s,X_0,t} X_1 \ra H$ by $(\gamma_2, \gamma_1) \mapsto h_{\gamma_2, \gamma_1}$, where $h_p$ and  $h_{\gamma_2, \gamma_1}$  are respectively unique elements in $H$ satisfying
	
	\begin{equation}\label{Equation: Construction 1}
		\mc{C}(1_{\pi(p)},p) = 1_p (h_p,e)
	\end{equation}
	for $p \in E_0$,
	\begin{equation}\label{Equation: Construction 2}
		\mc{C}(\gamma_2 , \mu_{\mc{C}}(\gamma_1,p)) \circ \mc{C}(\gamma_1,p)= \mc{C}(\gamma_2 \circ \gamma_1, p) (h_{\gamma_2, \gamma_1}, e) 
	\end{equation}
	for composable $\gamma_2, \gamma_1$.
	
	\subsection*{Proof of (a) and (b):} 
	(a) and (b) follow directly by taking target map $t$ on both sides of \Cref{Equation: Construction 1} and \Cref{Equation: Construction 2} respectively.
	\subsection*{Proof of (c) and (d):}
	To prove (c), note that from \Cref{Equation: Construction 2}, immediately we get $\mc{C}(\gamma,p) \big( h_{\gamma, 1_{\pi(p)}}, e \big)=\mc{C} \big( \gamma, \mu_{\mc{C}}(1_{\pi(p)},p) \big) \circ \mc{C} \big(1_{\pi}(p),p \big)$ for all $(\gamma,p) \in s^{*}E_0$. Then (c) follows easily from (a) and the freeness of the action of $H \rtimes_{\alpha}G$ on $E_1$. (d) can be proved using similar techniques as in (c).
	\subsection*{Proof of (e):}
	A straightforward consequence of (c). 
	\subsection*{Proof of (f):}
	(f) follows by observing that, using (e) we can rewrite the identity $\mc{C}(1_{\pi(pg^{-1})},pg^{-1})= 1_{pg^{-1}}(h_{pg^{-1}},e)$ (from \Cref{Equation: Construction 1}) as $1p(h_p,e) (e,g^{-1})=1_p(e,g^{-1})(h_p,e)$  for $p \in E_0$ and $g \in G$. 
	\subsection*{Proof of (g):}
	(g) follows from (f) using the Peiffer identity \Cref{E:Peiffer}. 
	\subsection*{Proof of (h):}
	For composable $\gamma_2, \gamma_1 \in X_1$, $g \in G$ and  $(\gamma_1,p) \in s^{*}E_0$, we have $\mc{C}(\gamma_2 , \mu_{\mc{C}}(\gamma_1,pg^{-1})) \circ \mc{C}(\gamma_1,pg^{-1})= \mc{C}(\gamma_2 \circ \gamma_1, pg^{-1}) (h_{\gamma_2, \gamma_1}, e)$. Then (h) follows straightforwardly from \Cref{Equation: Construction 2}. 
	\subsection*{Proof of (i):}
	(i)  follows from (h) and the Peiffer identity \Cref{E:Peiffer}. 
	\subsection*{Proof of  (j):}
	To prove (j), consider $\gamma_3, \gamma_2, \gamma_1 \in X_1$, a sequence of composable morphisms such that $(\gamma_1,p) \in s^{*}E_0$. Then we  have
	\begin{equation}\nonumber
		\begin{split}
			& \mc{C}(\gamma_3 \circ \gamma_2 \circ \gamma_1,p)(h_{\gamma_3 \circ \gamma_2, \gamma_1},e) \\
			&= \mc{C}(\gamma_3 \circ \gamma_2, \mu_{\mc{C}}(\gamma_1,p)) \circ \mc{C}(\gamma_1,p) [{\textit{by \Cref{Equation: Construction 2}}}].  \\
			&= \underbrace{\big( \mc{C}(\gamma_3, \mu_{\mc{C}}(\gamma_2 \circ \gamma_1,p)\tau(h_{\gamma_2,\gamma_1})) \circ \mc{C}(\gamma_2,\mu_{\mc{C}}(\gamma_1,p)) \big)(h^{-1}_{\gamma_3,\gamma_2},e)}_{\mc{C}(\gamma_3 \circ \gamma_2, \mu_{\mc{C}}(\gamma_1,p))[{\textit{by \Cref{Equation: Construction 2}}}]} \circ \big( \mc{C}(\gamma_1,p)(e,e) \big) \\
			&= \big( \mc{C}(\gamma_3, \mu_{\mc{C}}(\gamma_2 \circ \gamma_1,p)\tau(h_{\gamma_2,\gamma_1})) \circ \underbrace{\mc{C}(\gamma_2,\mu_{\mc{C}}(\gamma_1,p)) \circ \mc{C}(\gamma_1,p) \big) (h^{-1}_{\gamma_3,\gamma_2},e)}_{\textit{by functoriality of the action.}}  \\
			&= \underbrace{\big( \mc{C}(\gamma_3, \mu_{\mc{C}}(\gamma_2 \circ \gamma_1,p)) 1_{\tau(h_{\gamma_2,\gamma_1})}}_{[{\textit{by (iii), \Cref{Def:categorical connection}}}]} \circ \underbrace{\mc{C}(\gamma_2 \circ \gamma_1,p)(h_{\gamma_2,\gamma_1},e) \circ \mc{C}(\gamma_1,p)^{-1}}_{\mc{C}(\gamma_2,\mu_{\mc{C}}(\gamma_1,p))[{\textit{by \Cref{Equation: Construction 2}}}]}  \circ \mc{C}(\gamma_1,p) \big) (h^{-1}_{\gamma_3,\gamma_2},e) \\
			&= \big( \mc{C}(\gamma_3, \mu_{\mc{C}}(\gamma_2 \circ \gamma_1,p)) 1_{\tau(h_{\gamma_2,\gamma_1})} \circ \mc{C}(\gamma_2 \circ \gamma_1,p)(h_{\gamma_2,\gamma_1},e) \big)(h^{-1}_{\gamma_3,\gamma_2},e)  \\ 
			&= \underbrace{\bigg(\mc{C}(\gamma_3, \mu_{\mc{C}}(\gamma_2 \circ \gamma_1,p)) \circ \mc{C}(\gamma_2 \circ \gamma_1,p)(h_{\gamma_2,\gamma_1},e) \bigg)}_{\textit{by functoriality of the action.}}(h^{-1}_{\gamma_3,\gamma_2},e)  \\
			&= \underbrace{\bigg( \mc{C}(\gamma_3 \circ \gamma_2 \circ \gamma_1,p) (h_{\gamma_3,\gamma_2\circ \gamma_1 },e)\bigg)}_{[\textit{{by \Cref{Equation: Construction 2}}}]} (h_{\gamma_2,\gamma_1},e)(h^{-1}_{\gamma_3,\gamma_2},e) .
		\end{split}
	\end{equation}
	which concludes the proof of (j).
	\subsection*{Proof of (k)}
	In order to prove (k), note that we have $\mc{C}\big(\gamma,\mu_{\mc{C}}(\gamma^{-1}, \mu_{\mc{C}}(\gamma, p)) \big)\circ \mc{C}(\gamma^{-1}, \mu_{\mc{C}}(\gamma, p))= 1_{\mu_{\mc{C}}(\gamma,p)}(h_{\mu_{\mc{C}}(\gamma,p)}h_{\gamma,\gamma^{-1}},e)$ for $(\gamma,p) \in s^{*}E_0$. On the other hand, consider 
	\begin{equation}\label{Equation inverse 2}\nonumber
		\begin{split}
			& \mc{C}\big(\gamma,\mu_{\mc{C}}(\gamma^{-1}, \mu_{\mc{C}}(\gamma,p)) \big) \circ \mc{C}(\gamma^{-1}, \mu_{\mc{C}}(\gamma,p))\\
			&= \mc{C} \big( \gamma, \underbrace{\mu_{\mc{C}}(1_{\pi(p)},p)\tau(h_{\gamma^{-1},\gamma})}_{[{\textit{by (b), \Cref{Lemma: cohenrence of canonical quasi action}}}]} \big) \circ \mc{C}(\gamma^{-1}, \mu_{\mc{C}}(\gamma,p)) \\
			&= \underbrace{\mc{C} \big( \gamma, p.\tau(h_ph_{\gamma^{-1},\gamma}) \big)}_{[{\textit{by (a), \Cref{Lemma: cohenrence of canonical quasi action}}}]} \circ \mc{C}(\gamma^{-1}, \mu_{\mc{C}}(\gamma,p))\\
			&= \underbrace{\mc{C}(\gamma,p)1_{\tau(h_ph_{\gamma^{-1},\gamma})}}_{[{\textit{by (iii), \Cref{Def:categorical connection}}}]} \circ \underbrace{\big(\mc{C}(1_{\pi(p)},p)(h_{\gamma^{-1},\gamma}, e) \circ \mc{C}(\gamma,p)^{-1} \big)}_{[{\rm{by \Cref{Equation: Construction 2}}}].}\\
			&= \mc{C}(\gamma,p)1_{\tau(h_ph_{\gamma^{-1},\gamma})} \circ \underbrace{\big(1_p(h_ph_{\gamma^{-1},\gamma},e) \big)}_{[{\textit{by \Cref{Equation: Construction 1}}}].} \circ \mc{C}(\gamma,p)^{-1}\\
			&= \underbrace{\big( \mc{C}(\gamma,p)(h_ph_{\gamma^{-1},\gamma},e) \big)}_{\textit{by functoriality of the action.}} \circ  \mc{C}(\gamma,p)^{-1}\\
			&= 1_{(\mu_{\mc{C}}(\gamma,p))}(h_ph_{\gamma^{-1},\gamma},e)[\textit{by functoriality of the action}].
		\end{split}
	\end{equation}
	which completes the proof of (k).
\end{proof}
Properties (a)-(k) listed above will be called the \textit{coherence properties}.

Conversely, we have the following result.

\begin{proposition}\label{Theorem:quasi-bundle construction}
	For a Lie group $G$, let $(\pi \colon E_G \ra X_0, \mu, \mb{X})$ be a quasi-principal $G$-bundle over a Lie groupoid $\mb{X}$. Let $(G,H, \tau, \alpha)$  be a Lie crossed module along with a pair of smooth maps $\mc{H}_u \colon E_G \ra H$ and $\mc{H}_{m} \colon X_1 \times_{s,X_0,t} X_1 \ra H$ satisfying the coherence properties in \Cref{Lemma: cohenrence of canonical quasi action}, then we have the following:
	\begin{enumerate}
		\item	The manifolds $(s^{*}E_G)^{q-\rm{dec}}:= s^{*}E_G \times H $ and $E_G$ define a Lie groupoid $[(s^{*}E_G)^{\rm{q-dec}} \rightrightarrows E_G]$ whose structure maps are given as 
		\begin{itemize}
			\item[(i)] source map $s$: $(\gamma, p, h) \mapsto p$,
			\item[(ii)] target map $t$: $(\gamma, p, h) \mapsto \mu(\gamma, p) \tau(h^{-1})$,
			\item[(iii)] composition map $m \colon \big((\gamma_2, p_2, h_2), (\gamma_1, p_1, h_1) \big) \mapsto \big( \gamma_2 \circ \gamma_1, p_1 ,h_2h_1\mc{H}^{-1}_m({\gamma_2,\gamma_1}) \big)$, 
			\item[(iv)] unit map $u : p \mapsto (1_{\pi(p)},p, \mc{H}_{u}(p))$,
			\item[(v)] inverse map $ \mathfrak{i} \colon \bigl(\gamma, p, h) \mapsto (\gamma^{-1}, \mu(\gamma,p)\tau(h^{-1}), \mc{H}_{u}(p)\mc{H}_m({\gamma^{-1}, \gamma})h^{-1}\bigr)$.
		\end{itemize}
		\item The Lie groupoid  $\mb{E}^{\rm{q-dec}}:=[(s^{*}E_G)^{\rm{q-dec}} \rightrightarrows E_G]$ forms a quasi-principal $[H \rtimes_{\alpha}G \rra G]$-bundle $\pi^{\rm{q-dec}}\colon \mb{E}^{\rm q-dec}\ra \mb{X}$ over $\mb{X}$ equipped with the quasi connection $\mc{C}^{\rm{q-dec}} \colon s^{*}E_G \ra (s^{*}E_G)^{\rm{q-dec}}$, $(\gamma,p) \mapsto (\gamma,p, e)$. The action of $[H \rtimes_{\alpha}G \rra G]$ on $\mb{E}^{\rm{q-dec}}$ and the bundle projection coincide with that of the decorated case (See \Cref{Prop:Decoliegpd}).
		\item The quasi connection $\mc{C}^{\rm{q-dec}}$ is a categorical connection if and only if the maps $\mc{H}_u$ and $\mc{H}_m$ are constant maps to $e$.
		
	\end{enumerate}
	
\end{proposition}
\begin{proof}:
	
	\subsection*{Proof of (1):}
	From \Cref{Prop:Decoliegpd}, it follows that the source and target maps are surjective submersions. Let $(\gamma_2,p_2,h_2),(\gamma_1,p_1,h_1)$ be a composable pair of morphisms. To show that the source is compatible with the composition, observe that $s((\gamma_2,p_2,h_2) \circ (\gamma_1,p_1,h_1))= s \big( \gamma_2 \circ \gamma_1, p_1, h_2h_1 \mc{H}^{-1}_m({\gamma_2,\gamma_1}) \big)= p_1= s(\gamma_1, p_1, h_1)$, whereas the target consistency follows easily from the condition (b) in \Cref{Lemma: cohenrence of canonical quasi action}. For the unit map to make sense, note that we have $t(u(p))= t(1_{\pi(p)},p, \mc{H}_u(p))=\mu(1_{\pi(p)},p) \tau(\mc{H}_u(p)^{-1})=p \tau(\mc{H}_u(p))\tau(\mc{H}_u(p)^{-1})=p.$  
	The fact that $u$ is indeed a unit map follows from the right and left unitor (conditions (c) and (d) in \Cref{Lemma: cohenrence of canonical quasi action}). More precisely, right unitor implies $\big(\gamma,p,h \big) \circ (1_{\pi(p)},p, \mc{H}_u(p))= (\gamma, p, h \mc{H}_u(p) \mc{H}^{-1}_m(\gamma,1_{\pi(p)}))= \big(\gamma,p,h \big).$ On the other hand, $G$-invariance of $\mc{H}_u$ (condition (e) in \Cref{Lemma: cohenrence of canonical quasi action} ) and the left unitor property imply $\big( 1_{\pi(\mu(\gamma,p)\tau(h^{-1}))}, \mu(\gamma,p)\tau(h^{-1}), \mc{H}_u(\mu(\gamma,p)\tau(h^{-1})) \big) \circ \big((\gamma,p),h \big)= \big(\gamma,p, \mc{H}_u(\mu(\gamma,p))h \mc{H}^{-1}_m(1_{\pi(\mu(\gamma,p)}, \gamma) \big)= \big((\gamma,p),h \big)$. 
	
	To check the associativity of the composition, consider a sequence of composable morphisms $(\gamma_3,p_3,h_3), (\gamma_2,p_2,h_2), (\gamma_1,p_1,h_1) \in (s^{*}E_G)^{\rm{q-dec}}$. Now, $\big( (\gamma_3,p_3,h_3) \circ(\gamma_2,p_2,h_2) \big) \circ (\gamma_1,p_1,h_1)= \big( \gamma_3 \circ \gamma_2, p_2, h_3h_2 \mc{H}_m(\gamma_3, \gamma_2)^{-1} \big) \circ(\gamma_1, p_1 ,h_1)$ which is same as $\big( \gamma_3 \circ \gamma_2 \circ \gamma_1, p_1, h_3h_2 \mc{H}_m(\gamma_3, \gamma_2)^{-1}h_1 \mc{H}_m^{-1}(\gamma_3 \circ  \gamma_2,\gamma_1) \big)$. Whereas, $(\gamma_3,p_3,h_3) \circ \big((\gamma_2,p_2,h_2)  \circ  (\gamma_1,p_1,h_1) \big)= (\gamma_3,p_3,h_3) \circ \big( \gamma_2 \circ \gamma_1, p_1, h_2h_1 \mc{H}^{-1}_m(\gamma_2, \gamma_1) \big)$ which is equal to $\big( \gamma_3 \circ \gamma_2 \circ \gamma_1, p_1, h_3h_2h_1 \mc{H}^{-1}_m(\gamma_2, \gamma_1) \mc{H}^{-1}_m(\gamma_3,\gamma_2 \circ \gamma_1) \big)$. The associativity of the composition then follows from (condition (i) and (j) in \Cref{Lemma: cohenrence of canonical quasi action}). The compatibility of the inverse with the target is clear. For verifying $\mathfrak{i}$ is indeed the inverse, first observe
	$$\bigl( \gamma^{-1}, \mu(\gamma,p)\tau(h^{-1}), \mc{H}_{u}(p)\mc{H}_m({\gamma^{-1}, \gamma})h^{-1}\bigr) \circ \big(\gamma,p,h \big)= (1_{\pi(p)},p, \mc{H}_u(p)).$$ Then, from the invertor (condition (k) in \Cref{Lemma: cohenrence of canonical quasi action}), it follows $\big((\gamma,p),h \big) \circ \mathfrak{i}(\gamma,p,h)= u({\pi(\mu(\gamma,p) \tau(h^{-1}))})$. Finally, since its structure maps are smooth by definition, it follows $\mb{E}^{\rm{q-dec}}$ is a Lie groupoid.
	\subsection*{Proof of (2) and (3):}
	Since the action 
	\begin{equation}\nonumber
		\begin{split}
			\rho\colon &\mb{E}^{\rm q-dec}\times [H \rtimes_{\alpha}G \rra G] \ra \mb{E}^{\rm q-dec}\\
			&(p, g) \mapsto pg,\\
			\bigl((\gamma, p, h)&, (h', g)\bigr)\mapsto \bigl(\gamma, p g, \alpha_{g^{-1}}(h'^{-1}\, h)\bigr),
		\end{split}
	\end{equation}
	coincides with the decorated case as in \Cref{E:Actionondeco}, to prove $\rho$ defines a right action of $[H \rtimes_{\alpha}G \rra G]$ on $\mb{E}^{\rm q-dec}$, we just need to check the compatibility of $\rho$ with unit maps and the composition. But these are direct consequences of condition (f) and condition (h) in \Cref{Lemma: cohenrence of canonical quasi action} respectively, combined with the functoriality of the action in the decorated case. Also, since the bundle projection functor coincides with the decorated case as in \Cref{E:Projondeco}, it follows $\pi \colon \mb{E}^{\rm{q-dec}} \ra \mb{X}$ is a $[H \rtimes_{\alpha}G \rra G]$-bundle over $\mb{X}$. Now, as for each $p \in E_G$, we have
	$\mc{C}^{\rm{q-dec}}(1_{\pi(p)},p)= u(p)(\mc{H}_u(p),e)$ and for any composable sequence of morphisms $\gamma_2, \gamma_1 \in X_1$ such that $(\gamma_1,p) \in s^{*}E_G$, we have $\mc{C}^{\rm{q- dec}}(\gamma_2,t(\mc{C}(\gamma_1,p))) \circ \mc{C}^{\rm{q- dec}}(\gamma_1,p)= \mc{C}^{\rm{q- dec}}(\gamma_2 \circ \gamma_1,p)(\mc{H}_m(\gamma_2,\gamma_1),e)$, it follows that $\mc{C}^{\rm{q-dec}}$ is a quasi connection. Observe that \textbf{(3)} is an immediate consequence of the above two identities.
\end{proof}
\begin{definition}\label{pseudo principal Lie crossed module}
	Given a Lie crossed module $(G,H, \tau, \alpha)$, a \textit{pseudo-principal $(G,H, \tau, \alpha)$-bundle over a Lie groupoid $\mb{X}$} is defined as a quasi-principal $G$-bundle $(\pi_G \colon E_G \ra X_0, \mu, \mb{X})$ over the Lie groupoid $\mb{X}$ (\Cref{quasi-principal G-bun dle}), equipped with a pair of smooth maps $\mc{H}_{u} \colon E_0 \ra H$   and $\mc{H}_{m} \colon X_1 \times_{s,X_0,t} \times X_1 \ra H$, satisfying the coherence properties (a)-(k) in \Cref{Lemma: cohenrence of canonical quasi action}.
	
\end{definition}
We  denote a pseudo-principal $(G,H, \tau, \alpha)$-bundle by the notation $(\pi \colon E_G\ra X_0,\mu, \mc{H}_u, \mc{H}_m,\mb{X})$, and call the smooth maps $\mc{H}_{u}$ and $\mc{H}_m$ as  \textit{unital deviation} and \textit{compositional deviation} respectively.

\begin{example}\label{Definition makes sense}
	It follows directly from \Cref{Lemma: cohenrence of canonical quasi action}, that the underlying quasi-principal $G$-bundle $(\pi_0 \colon E_0 \ra X_0,  \mu_{\mc{C}}, \mb{X})$ (\Cref{underlying quasi-principal bundle})  of a quasi-principal-$[H \rtimes_{\alpha}G \rra G]$-bundle $(\pi \colon \mb{E} \ra \mb{X}, \mc{C})$, equipped with the pair of smooth maps $\mc{H}_{u,\mc{C}}$ and $\mc{H}_{m,\mc{C}}$ (as defined in \Cref{Lemma: cohenrence of canonical quasi action}) is a pseudo-principal $(G,H, \tau, \alpha)$-bundle over the Lie groupoid $\mb{X}$. We call $(\pi_0 \colon E_0 \ra X_0,  \mu_{\mc{C}}, \mc{H}_{u,\mc{C}}, \mc{H}_{m,\mc{C}}, \mb{X})$ as the \textit{underlying  pseudo-principal $(G,H, \tau, \alpha)$-bundle of the quasi-principal $[H \rtimes_{\alpha}G \rra G]$-bundle $(\pi \colon \mb{E} \ra \mb{X}, \mc{C})$.} Unital deviation $\mc{H}_{u,\mc{C}}$ and compositional deviation $\mc{H}_{m,\mc{C}}$ together precisely mesaure the amount by with the quasi connection $\mc{C}$ differs from being a categorical connection.
\end{example}

A quasi-principal $[H \rtimes_{\alpha}G \rra G]$-bundle $( \pi^{\rm{q-dec}} \colon \mb{E}^{\rm{q-dec}}\ra \mb{X}, \mc{C}^{\rm{q-dec}})$ in \Cref{Theorem:quasi-bundle construction} associated to a pseudo-principal $(G,H, \tau, \alpha)$\-bundle  $(\pi \colon E_G\ra X_0,\mu, \mc{H}_u, \mc{H}_m,\mb{X})$ will be called a \textit{quasi-decorated  $[H \rtimes_{\alpha}G \rra G]$-bundle over $\mb{X}$.} 

The following is evident:

\begin{proposition}\label{Groupoid of pseudo crossed module principal bundle}
	Given a Lie groupoid $\mb{X}$ and a Lie crossed module $(G,H, \tau, \alpha)$, there is a groupoid $\rm{Pseudo} (\mb{X}$,$(G,H,\tau, \alpha)$$)$, whose objects are pseudo-principal$(G,H, \tau, \alpha)$-bundles over $\mb{X}$  and morphisms are defined in the following way:
	
	For any pair of pseudo-principal$(G,H, \tau, \alpha)$-bundles $(\pi \colon E_G\ra X_0,\mu, \mc{H}_u, \mc{H}_m,\mb{X})$ and $(\pi' \colon E'_G\ra X_0,\mu', \mc{H}'_u, \mc{H}'_m,\mb{X})$,
	\begin{itemize}
		\item[(i)]if $\mc{H}_{m} \neq \mc{H}'_{m}$, then $${\rm{Hom}}\Big((\pi \colon E_G\ra X_0,\mu, \mc{H}_u, \mc{H}_m,\mb{X}),(\pi' \colon E'_G\ra X_0,\mu', \mc{H}'_u, \mc{H}'_m,\mb{X}) \Big) = \emptyset ,$$
		\item[(ii)] if $\mc{H}_{m}=\mc{H}'_{m}$, 
		
		then an element of $ {\rm{Hom}}\Big((\pi \colon E_G\ra X_0,\mu, \mc{H}_u, \mc{H}_m,\mb{X}),(\pi' \colon E'_G\ra X_0,\mu', \mc{H}'_u, \mc{H}'_m,\mb{X}) \Big)$ is defined as a morphism of principal $G$-bundles $f \colon E_G \ra E'_G$, satisfying the following conditions:
		\begin{itemize}
			\item[(a)] $f(\mu(\gamma,p))= \mu'(\gamma,f(p))$ and
			\item[(b)] $\mc{H}_u= \mc{H}'_u \circ f$.
		\end{itemize}
	\end{itemize}
\end{proposition}
Now, we are ready to state and prove our first main result of the paper.
\begin{theorem}\label{Grothendieck construction}
	For a Lie crossed module $(G,H,\tau, \alpha)$ and a Lie groupoid $\mb{X}$, the groupoid $\rm{Bun}_{\rm{quasi}}(\mb{X}$, $[H \rtimes_{\alpha}G \rra G])$ is equivalent to the groupoid $\rm{Pseudo} (\mb{X}$,$(G,H,\tau, \alpha)$$)$.
\end{theorem}
\begin{proof}
	Define
	\begin{equation}\nonumber
		\begin{split}
			\mc{F}\colon &{\rm{Bun}}_{\rm{quasi}}(\mb{X}, [H \rtimes_{\alpha}G \rra G])\ra {\rm{Pseudo}}\Big(\mb{X},(G,H, \tau, \alpha) \Big),\\
			& (\pi \colon \mb{E} \ra \mb{X}, \mc{C}) \mapsto (\pi_0 \colon E_0 \ra X_0,  \mu_{\mc{C}}, \mc{H}_{u,\mc{C}}, \mc{H}_{m,\mc{C}}, \mb{X})\\
			&(F \colon \mb{E} \ra \mb{E}') \mapsto (F_0 \colon E_0 \ra E'_0).
		\end{split}
	\end{equation}
	We claim  $\mc{F}$ is an essentially surjective, faithful, and full functor. 
	
	Suppose  $(\pi \colon \mb{E} \ra \mb{X}, \mc{C})$ is a quasi-principal $[H \rtimes_{\alpha}G \rra G])$-bundle over $\mb{X}$. Then from \Cref{Definition makes sense}, it directly follows that $(\pi_0 \colon E_0 \ra X_0,  \mu_{\mc{C}}, \mc{H}_{u,\mc{C}}, \mc{H}_{m,\mc{C}}, \mb{X})$ is indeed a pseudo-principal $(G,H, \tau, \alpha)$-bundle over $\mb{X}$.  Now, let $F \in {\rm{Hom}}\Big((\pi \colon \mb{E} \ra \mb{X}, \mc{C}), (\pi' \colon \mb{E}' \ra \mb{X})) \Big)$, then as a direct consequence of the identities 
	$F_1(\mc{C}(\gamma,p))= \mc{C}'(\gamma, F_0(p))$ and  $F_1 \big( \mc{C}(1_{\pi(p)},p) \big)= 1_{F_0(p)} \Big(\mc{H}_{u,\mc{C}'}\big(F(p)\big),e \Big)$, we have a well-defined $\mc{F}(F)$. Functoriality of $\mc{F}$ is a straightforward verification. By \Cref{Theorem:quasi-bundle construction}, $\mc{F}$ is essentially surjective. Now, consider $F,\bar{F} \in {\rm{Hom}} \Big( (\pi \colon \mb{E} \ra \mb{X}, \mc{C}) ,(\pi' \colon \mb{E}' \ra \mb{X}, \mc{C}')  \Big)$. Suppose $\mc{F}(F)= \mc{F}(\bar{F})$, that is $F_0=\bar{F}_0$. Now, for  any $\delta \in E_1$, there exists a unique $h_{\delta} \in H$, such that $\delta= \mc{C}\big(\pi_1(\delta),s(\delta)\big)(h_{\delta},e)$. Then the equality $F_1(\delta)=\bar{F}_1(\delta)$ follows from the compatiblity condition  of $F$ and $\bar{F}$ with quasi connections $\mc{C}$ and $\mc{C}'$ in \Cref{Groupoid of quasi principal 2-bundles}. Hence, $\mc{F}$ is faithful. To show $\mc{F}$ is full, consider an element $f$ in ${\rm{Hom}}\Big((\pi_0 \colon E_0 \ra X_0,  \mu_{\mc{C}}, \mc{H}_{u,\mc{C}}, \mc{H}_{m,\mc{C}}, \mb{X}), (\pi'_0 \colon E'_0 \ra X_0,  \mu_{\mc{C}'}, \mc{H}_{u,\mc{C}'}, \mc{H}_{m,\mc{C}'}, \mb{X}) \Big)$ for a pair of quasi-principal $[H \rtimes_{\alpha}G \rra G]$-bundles $(\pi \colon \mb{E} \ra \mb{X}, \mc{C})$ and $(\pi' \colon \mb{E}' \ra \mb{X}, \mc{C}')$ over $\mb{X}$. Now, define 
	\begin{equation}\nonumber
		\begin{split}
			&F \colon \mb{E} \ra \mb{E}'\\
			& p \mapsto f(p), p \in E_0,\\
			&\delta \mapsto \mc{C}'\big(\pi_1(\delta), f(s(\delta)) \big)(h_{\delta},e), \delta \in E_1,
		\end{split}
	\end{equation}
	where $h_{\delta}$ is the unique element in $H$, such that $\delta= \mc{C}\big(\pi_1(\delta),s(\delta)\big)(h_{\delta},e)$. We need to show $F:=(F_1,F_0)$ is a morphism of quasi-principal $[H \rtimes_{\alpha}G \rra G])$-bundles over $\mb{X}$. Note that $\pi'_0 \circ F_0= \pi'_0 \circ f=\pi_0$ and for any $\delta \in E_1$, $\pi'_1 \circ F_1(\delta)= \pi'_1 \Big(\mc{C}'\big(\pi_1(\delta), f(s(\delta)) \big)(h_{\delta},e) \Big)= \pi_1(\delta)$. $G$-equivariancy of $f$ implies the same for $F_0$. Now, observe that there exists a unique element $\bar{h} \in H$ such that $\mc{C}\big(\pi_1(\delta),s(\delta)\big)(h_{\delta},e)(h,g)= \mc{C} \big(\pi_1(\delta), s(\delta) \big)(e,g)(\bar{h},e)$ for $\delta \in E_1$ and $(h,g) \in H \rtimes_{\alpha} G$. Then, by a little calculation, we arrive at $$F_1 \big( \delta(h,g) \big)= \mc{C}' \big( \pi_1(\delta), f(s(\delta))\big)(\bar{h},g).$$ From the observation $\bar{h}=h_{\delta}h$, we get 	
	$F_1 \big( \delta(h,g) \big)= \mc{C}' \big( \pi_1(\delta), f(s(\delta))\big)(h_{\delta},e)(h,g)= F_1(\delta)(h,g)$. Hence, $F_1$ is $H \rtimes_{\alpha}G$-equivariant. Now, note that from the definition itself, it is evident that $F_1(\mc{C}(\gamma,p))= \mc{C}'(\gamma, F_0(p))$ for all $(\gamma,p) \in s^{*}E_0$. Since both $F_0$ and $F_1$ are smooth by definition, in order to prove $\mc{F}$ is full, it is sufficient to prove that $F \colon \mb{E} \ra \mb{E}'$ is a functor. Source map consistency is trivial, whereas the target consistency follows from the compatibility condition of $f$ with $\mu_{\mc{C}}$ and $\mu_{\mc{C}'}$ as mentioned in \Cref{Groupoid of pseudo crossed module principal bundle}. Now, for any $p \in E_0$, we have $F_1(1_p)= \mc{C}' \big(1_{\pi'_0 \circ f(p)}, f(p) \big) \big((\mc{H}_{u,\mc{C}}(p))^{-1},e \big)$. By appropriately using \Cref{Equation: Construction 1}, we obtain the compatibility of $\mc{F}$ with unit maps. For compositional compatibility, consider $\delta_2, \delta_1 \in E_1$ such that $s(\delta_2)=t(\delta_1)$. It is not difficult to observe that to show the required compatibility, it is sufficient to prove the following identity:
	\begin{equation}\nonumber
		\mc{H}_{m, \mc{C}}(\gamma_2, \gamma_1)h_{\delta_1}h_{\delta_2}=h_{\delta_2 \circ \delta_1}.
	\end{equation}
	To prove the above identity, consider
	\begin{equation}\nonumber
		\begin{split}
			&\delta_2 \circ \delta_1\\
			&= \underbrace{\mc{C}\big(\pi_1(\delta_2),s(\delta_2)\big)(h_{\delta_2},e)}_{\delta_2} \circ \underbrace{\mc{C}\big(\pi_1(\delta_1),s(\delta_1)\big)(h_{\delta_1},e)}_{\delta_1} \\
			&= \underbrace{\mc{C}\Big(\pi_1(\delta_2), \underbrace{\mu_{\mc{C}} \big( \pi_1(\delta_1),s(\delta_1)}_{s(\delta_2)=t(\delta_1)} \big)\Big)(e, \tau(h_{\delta_1}))}_{[{\textit{by (iii), \Cref{Def:categorical connection}}}]}(h_{\delta_2},e) \circ \mc{C}\big(\pi_1(\delta_1),s(\delta_1)\big)(h_{\delta_1},e)\\
			&= \underbrace{ \bigg( \mc{C}\Big(\pi_1(\delta_2), \mu_{\mc{C}} \big( \pi_1(\delta_1),s(\delta_1) \big)\Big)\circ  \mc{C}\big(\pi_1(\delta_1),s(\delta_1)\big) \bigg) \bigg(\Big(\alpha_{\tau(h_{\delta_1})}(h_{\delta_2}),\tau(h_{\delta_1}) \Big) \circ(h_{\delta_1},e) \bigg)}_{\textit{by functoriality of the action.}} \\
			&= \bigg( \mc{C}\Big(\pi_1(\delta_2), \mu_{\mc{C}} \big( \pi_1(\delta_1),s(\delta_1) \big)\Big)\circ  \mc{C}\big(\pi_1(\delta_1),s(\delta_1)\big) \bigg)  \bigg( \underbrace{\Big(h_{\delta_1}h_{\delta_2}h_{\delta_1}^{-1}, \tau(h_{\delta_1}) \Big)}_{[{\textit{by \Cref{E:Peiffer}}}]}  \circ \Big(h_{\delta_1}, e \Big) \bigg)\\
			&= \mc{C}\Big( (\pi_1(\delta_2 \circ \delta_1),s(\delta_2 \circ \delta_1) \big)\Big) \Big(\mc{H}_{m, \mc{C}}(\gamma_2, \gamma_1)h_{\delta_1}h_{\delta_2},e \Big) [{\textit{by \Cref{Equation: Construction 2}}}].
		\end{split}	
	\end{equation}
	Hence, $\mc{H}_{m, \mc{C}}(\gamma_2, \gamma_1)h_{\delta_1}h_{\delta_2}=h_{\delta_2 \circ \delta_1}$. 
	
	Thus, we complete the proof of $\mc{F}$ is an equivalence of categories.
\end{proof}

 A crucial consequence of \Cref{Grothendieck construction}, is a complete characterization of quasi-principal 2-bundles. The proof is not so difficult, though tedious.

\begin{corollary}\label{Splitting fibration}
	 Any quasi-principal $[H \rtimes_{\alpha}G \rra G]$-bundle $(\pi \colon \mb{E} \ra \mb{X}, \mc{C})$ over a Lie groupoid $\mb{X}$ is canonically isomorphic to the quasi-decorated $[H \rtimes_{\alpha}G \rra G]$-bundle $(\pi^{\rm{q-dec}} \colon \mb{E}^{\rm{q-dec}} \ra \mb{X}, \mc{C}^{\rm{q-dec}})$ associated to the underlying pseudo-principal $(G,H,\tau,\alpha)$-bundle $(\pi_0 \colon E_0 \ra X_0,  \mu_{\mc{C}}, \mc{H}_{u,\mc{C}}, \mc{H}_{m,\mc{C}}, \mb{X})$. The canonical  isomorphism is explicitly given by 
	
	\begin{equation}\nonumber
		\begin{split}
			\theta_{\mb{E}} \colon & \mb{E}^{\rm{q-dec}} \ra \mb{E}\\
			& p \mapsto p,  p \in E_0 \\
			&\Big((\gamma,p),h \Big) \mapsto \mc{C}(\gamma,p)(h^{-1},e), \Big(\gamma,p,h \Big) \in s^{*}E_0 \times H.
		\end{split}
	\end{equation}
	Moreover, the functor $\mc{F}$ in \Cref{Grothendieck construction} restricts to an essentially surjective, full and faithful functor to the subcategory ${\rm{Bun}}_{\rm{Cat}}(\mb{X}, [H \rtimes_{\alpha}G \rra G]) )$ of ${\rm{Bun}}_{{\rm{quasi}}}(\mb{X}, [H \rtimes_{\alpha}G \rra G])$ and hence yielding an equivalence of categories between ${\rm{Bun}}_{\rm{Cat}}(\mb{X}, [H \rtimes_{\alpha}G \rra G]) )$ and ${\rm{Bun}}(\mb{X}, G)$.
\end{corollary}

\subsection{Quasi connections as retractions}\label{Homotopy}Recall, given morphisms of Lie groupoids $\phi \colon \mb{Y} \ra \mb{X}$ and $\psi \colon \mb{Z} \ra \mb{X}$, there is a topological groupoid  $\mb{Y} \times_{\phi, \mb{X}, \psi}^{h} \mb{Z}$. An object of this groupoid is a triple $(y, \gamma: \psi(z) \ra \phi(y) ,z )$ for $y \in Y_0, z \in Z_0$, whereas a  morphism from $(y, \gamma: \psi(z) \ra \phi(y) ,z )$ to $(y', \gamma': \psi(z') \ra \phi(y') ,z')$ is given by a pair $(\Gamma \colon y \ra y', \delta \colon z \ra z')$  such that $\phi(\Gamma) \gamma = \gamma' \circ \psi(\delta)$.
$\mb{Y} \times_{\phi, \mb{X}, \psi}^{h} \mb{Z}$ satisfies the usual universal property of a fiber product (but up to an isomorphism). Structure maps are the natural ones. If either of $\phi_0 \colon Y_0 \ra X_0$ or $\psi_0 \colon Z_0 \ra X_0$ is a submersion, then  $\mb{Y} \times_{\phi, \mb{X}, \psi}^{h} \mb{Z}$ has a natural Lie groupoid structure and is known as  \textit{weak pull-back} or \textit{weak fibered product}. Now, a morphism of Lie groupoids $F \colon \mb{X} \ra \mb{Y}$ has a canonical factorization through $\mb{Y} \times_{\mb{Y},F}^{h} \mb{X}:= \mb{Y} \times_{{\rm{id}}_{\mb{Y}}, \mb{Y}, F}^{h} \mb{X}$ as $F= F_\mb{Y} \circ F_{\mb{X}}$, where at the object level, the maps $F_{\mb{X}} \colon \mb{Y} \times_{\mb{Y},F}^{h} \mb{X}$ and $F_{\mb{Y}} \colon \mb{Y} \times_{\mb{Y},F}^{h} \mb{X} \ra \mb{Y}$ are given as $x \mapsto (F_(x), {\rm{id}}_{F(x)},x)$ and $(\eta \colon y \ra y', \delta \colon p \ra p') \mapsto \eta$ respectively, and can be extended to the morphism level in a natural way. For a detailed discussion on these, readers can look at  \cite{MR2012261}.
%
%
%
As a consequence of this factorization, we will characterize unital (\Cref{Unital connection}) and categorical connections (\Cref{Def:categorical connection}) in terms of certain retractions. The proof is a suitable adaptation of the one given in \textbf{Proposition 2.2.3} and \textbf{Corollary 2.2.4 } of  \cite{del2008homotopy} in the set-theoretic setup, and we skip it here.
\begin{proposition}
	For a Lie 2-group $\mb{G}$, let $\pi \colon \mb{E} \ra \mb{X}$ be a principal $\mb{G}$-bundle over a Lie groupoid $\mb{X}$. Then 
	\begin{itemize}
		\item[(a)] there is an induced right action of $\mb{G}$ on the Lie groupoid $\mb{X} \times_{\mb{X}, \pi}^{h} \mb{E}$, given by
		\begin{equation}\nonumber
			\begin{split}
				\rho\colon &(\mb{X} \times_{\mb{X}, \pi}^{h} \mb{E}) \times \mb{G} \ra \mb{X} \times_{\mb{X}, \pi}^{h} \mb{E}\\
				&\big((x, \gamma, p), g \big) \mapsto \big(x, \gamma, pg \big), \\
				& \big((\Gamma \colon x \ra x', \delta \colon p \ra p'), \phi \big) \mapsto (\Gamma, \delta \phi),
			\end{split}
		\end{equation}
		\item[(b)] the set of $\mb{G}$-equivariant morphisms of Lie groupoids $r \colon \mb{X} \times_{\rm{id}_{\mb{X}}, \mb{X}, \pi}^{h} \mb{E} \ra \mb{E}$ satisfying $\pi \circ r= \pi_{\mb{X}}$, $r \circ \pi_{\mb{E}}= \rm{id}_{\mb{E}}$, is in one-one correspondence with the set of unital connections $\mc{C}$ on $\pi \colon  \mb{E} \ra \mb{X}$, 
		\item[(c)] a unital connection $\mc{C}$ on $\pi \colon \mb{E} \ra \mb{X}$ is a categorical connection if and only if the image of morphisms of the form  $(\Gamma, 1_p)$ under the associated map $r_{\mc{C}} \colon \mb{X} \times_{\mb{X}, \pi}^{h} \mb{E} \ra \mb{E}$, lies in the image of $\mc{C}$.
	\end{itemize}
\end{proposition}

\subsection{Towards a principal 2-bundle over a differentiable stack}\label{Morita}
According to the \textbf{Corollary 2.12} of \cite{MR2270285}, if the Lie groupoids $\mb{X}$ and $\mb{Y}$ are Morita equivalent, then for any Lie group $G$, the categories ${\rm{Bun}}(\mb{X},G)$ and ${\rm{Bun}}(\mb{Y},G)$ are equivalent.Then, \Cref{Splitting fibration} yields the following result:
\begin{proposition}\label{stack}
	Given a Lie 2-group $\mb{G}$, if Lie groupoids $\mb{X}$ and $\mb{Y}$ are Morita equivalent, then the category $\rm{Bun}_{\rm{Cat}}(\mb{X}, \mb{G})$ is equivalent to the category $\rm{Bun}_{\rm{Cat}}(\mb{Y}, \mb{G})$.
\end{proposition}
\Cref{stack} implies that the definition of a principal 2-bundle over a Lie groupoid, equipped with a categorical connection, can be extended over a differentiable stack represented by the base Lie groupoid. One can consult \cite{MR2817778} for a detailed discussion on differentiable stacks and Morita equivalent Lie groupoids. Our upcoming paper will deal with quasi-principal 2-bundle over a differentiable stack.

\section{Lazy Haefliger paths and thin fundamental groupoid of a Lie groupoid}\label{Section Lazy Haefliger paths and thin fundamental groupoid of a Lie groupoid} 
In this section, we introduce a notion of thin fundamental groupoid of a Lie groupoid, a generalization of the classical one, and will impose a diffeological structure on it.

We start by defining the notion of a lazy Haefliger path.

\begin{definition}\label{Definition:Haefliger path}
	Let $\mb{X}$ be a Lie groupoid and let $x,y \in X_0$. A \textit{lazy $\mb{X}$-path} or a \textit{lazy Haefliger path} $\Gamma$ \textit{from $x$ to $y$} is a sequence $\Gamma :=(\gamma_0, \alpha_1,\gamma_1, \cdots,\alpha_n, \gamma_n)$ for some $n \in \mb{N}$ where 
	\begin{itemize}
		\item[(i)] $\alpha_i:[0,1] \ra X_0$ is a path with sitting instants for all $1 \leq i \leq n$ and 
		\item[(ii)] $\gamma_i \in X_1$ for all $0 \leq i \leq n$, 
	\end{itemize}
	such that the following conditions hold:
	\begin{itemize}
		\item[(a)] $s(\gamma_0)=x$ and $t(\gamma_n)=y$;
		\item[(b)] $s(\gamma_i)= \alpha_i(1)$ for all $0 < i \leq n$;
		\item[(c)] $t(\gamma_i)= \alpha_{i+1}(0)$ for all $0 \leq i < n$.
	\end{itemize}
\end{definition}

We will say that $\Gamma $ is a \textit{lazy $\mb{X}$-path of order $n$}. We define the \textit{source of $\Gamma$} as  $s(\gamma_0)=x$ and the \textit{target of $\Gamma$} as $t(\gamma_n)=y$. For a given Lie groupoid $\mb{X}$, we shall denote by $P\mb{X}$ the set of all lazy $\mb{X}$-paths of all orders. Observe that if we remove the sitting instants condition from (i), then we recover existing notion a Haefliger path as given in \cite{guruprasad2006closed, haefliger1982groupoides, colman20111, gutt2005poisson}.

We are interested in certain equivalence classes of lazy Haefliger paths. Such equivalences or its minor variations have already been studied, for example in \cite{guruprasad2006closed, gutt2005poisson, colman20111}.
\begin{definition}\label{Definition: Equivalence of X-paths}
	A lazy $\mb{X}$-path $\Gamma :=(\gamma_0, \alpha_1,\gamma_1, \cdots,\alpha_n, \gamma_n)$ is said to be \textit{equivalent} to another lazy $\mb{X}$-path $\bar{\Gamma}$, if one is obtained from the other by a finite sequence of all or some of the following operations:
	\begin{itemize}
		\item[(1)] \textit{Removing/adding a constant path}, that is if $\alpha_{i+1}$ is a constant path in the lazy $\mb{X}$-path $\Gamma$, then by removing it we obtain the lazy $\mb{X}$-path $(\gamma_0, \alpha_1,\gamma_1,..,\gamma_{i+1} \circ \gamma_i, \cdots,\alpha_n, \gamma_n)$, where $i \in \lbrace 1,2, \cdots n-1 \rbrace$. Replacing the word "removing" by "adding" one obtains the condition for "adding a constant path".
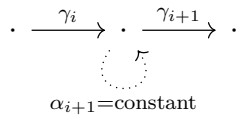
		\[\begin{tikzcd}
			\cdot \arrow[r, "\gamma_i"] & \cdot \arrow["\alpha_{i+1}=\rm{constant}"', dotted, loop, distance=2em, in=305, out=235] \arrow[r, "\gamma_{i+1}"] & \cdot
		\end{tikzcd}\]
		\item[(2)] \textit{Removing/adding an idenitity morphism}, that is if $\gamma_i$ is an identity morphism in the lazy $\mb{X}$-path $\Gamma$, then by removing it we obtain a lazy $\mb{X}$-path $(\gamma_0, \alpha_1,\gamma_1,..,\alpha_{i+1} * \alpha_i, \cdots,\alpha_n, \gamma_n)$, where $*$ is the concatenation of paths and $i \in \lbrace 1,2,..,n-1 \rbrace$. Replacing the word "removing" by "adding" one obtains the condition for "adding an identity morphism".
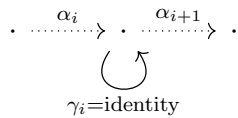
		\[\begin{tikzcd}
			\cdot \arrow[r, "\alpha_i", dotted] & \cdot \arrow["\gamma_i=\rm{identity}"', loop, distance=2em, in=305, out=235] \arrow[r, "\alpha_{i+1}", dotted] & \cdot
		\end{tikzcd}\]
		
		\item[(3)] \textit{Replacing $\alpha_i$ by $t \circ \zeta_i$, replacing $\gamma_{i-1}$ by $\zeta_i (0) \circ \gamma_{i-1}$ and $\gamma_{i}$ by $\gamma_i \circ (\zeta_i(1))^{-1}$ for a given path $\zeta_i \colon [0,1] \ra X_1$ with sitting instants, such that $s \circ \zeta_i= \alpha_i$} and $i \in \lbrace 1,2, \cdots, n \rbrace$, that is the portion $(\gamma_{i-1}, \alpha_i, \gamma_i)$ of the lazy $\mb{X}$-path $\Gamma$ is replaced by the portion $\big( \zeta_i(0) \circ \gamma_{i-1}, t \circ \zeta_i, \gamma_i \circ (\zeta(1))^{-1} \big)$ to obtain a lazy $\mb{X}$-path $(\gamma_0, \alpha_1,\gamma_1,..,\alpha_{i-1},\underbrace{\zeta_i(0) \circ \gamma_{i-1}, t \circ \zeta_i, \gamma_i \circ (\zeta(1))^{-1}},\alpha_{i+1}, \cdots,\alpha_n, \gamma_n).$ 
				See the diagram below: 
		\[
		\begin{tikzcd}
			\cdot \arrow[r, "\gamma_{i-1}"] & \cdot \arrow[r, "\alpha_i", dotted] \arrow[d, "\zeta_i(0)"'] & \cdot \arrow[r, "\gamma_i"]  & . \\
			& \cdot \arrow[r, "t \circ \zeta_i"', dotted]                & \cdot \arrow[u, "\zeta_i(1)^{-1}"'] &  
		\end{tikzcd}\]
	\end{itemize}
	
\end{definition}
Adapting the same conventions as in \cite{colman20111} and \cite{gutt2005poisson}, we  call the operations in \Cref{Definition: Equivalence of X-paths} as \textit{equivalences}.

Next, we introduce a notion of \textit{thin deformation of lazy $\mb{X}$-paths}, which can be thought of as a modified thin homotopy analog of the existing notion of deformation of $\mb{X}$-paths ( \cite{colman20111}, \cite{gutt2005poisson}). We start by briefly recalling the standard notion of thin homotopy equivalence class of paths on a manifold.
\begin{definition}[Definition 2.2, \cite{MR2520993}]\label{Classical thin homotopy}
	Let $\alpha , \beta \colon [0,1]  \ra M$ be elements of $PM$, the set of paths with sitting instants in the manifold $M$. Then $\alpha$ is said to be \textit{thin homotopic} to $\beta$ if there exists a smooth map $H:[0,1]^2 \rightarrow M$ with the following properties:		
	\begin{itemize}
		\item[(i)] $H$ has sitting instants i.e there exists an $\epsilon \in (0,1/2)$ with 
		\begin{itemize}
			\item[$\bullet$]  $H(s,t)=x$ for $t \in [0, \epsilon)$ and $H(s,t)=y$ for $t \in (1- \epsilon,1])$ and 
			\item[$\bullet$]$H(s,t)= \alpha(t)$ for $s \in [0,\epsilon)$ and $H(s,t)= \beta(t)$ for $s \in (1- \epsilon, 1]$;
		\end{itemize}
		\item[(ii)] the differential of $H$ has atmost rank 1 at all points.
	\end{itemize}
\end{definition}
It is well-known that under the thin homotopy quotient, $\Pi_{\rm{thin}}(M):= [\frac{PM}{\sim} \rra M]$ is a diffeological groupoid (see \textbf{Proposition A.25}, \cite{MR3521476}), known as the \textit{thin fundamental groupoid of $M$}.
\begin{definition}\label{Definition: Thin deformation}
	A \textit{thin deformation} from a lazy $\mb{X}$-path $\Gamma :=(\gamma_0, \alpha_1,\gamma_1, \cdots,\alpha_n, \gamma_n)$ to a lazy $\mb{X}$-path $\Gamma' :=(\gamma'_0, \alpha'_1, \gamma'_1 \cdots, \alpha'_n, \gamma'_n)$ of the same order is given by a sequence of smooth paths $ \lbrace \zeta_i: [0,1] \ra X_1 \rbrace_{i =0,1, \cdots,n}$ with $\zeta_i(0)= \gamma_i$ and $\zeta_i(1)= \gamma'_i$  such that 
	\begin{enumerate}[(i)]
		\item$\lbrace \zeta_i: [0,1] \ra X_1 \rbrace_{i =0,1, \cdots,n}$ are  paths with sitting instants;
		\item$\alpha_i$ is thin homotopic to $(s \circ \zeta_i)^{-1} * \alpha_i' * (t \circ \zeta_{i-1})$ for all $i=1,2 \cdots .,n$, where $*$ is the concatenation of paths, as illustrated by the following diagram:		
		\begin{equation}\nonumber
			\begin{tikzcd}
				\cdot \arrow[r] \arrow[r, "\gamma_{i-1}"] & \cdot \arrow[r, "\alpha_i", dotted] \arrow[d, "t \circ \zeta_{i-1}"', dotted] & \cdot \arrow[r, "\gamma_i"]                          & \cdot \\
				\cdot \arrow[r, "\gamma{\rm{'}}_{i-1}"']          & \cdot \arrow[r, "\alpha_i{\rm{'}}"', dotted]                        & \cdot \arrow[u, "(s \circ \zeta_i)^{-1}"', dotted] \arrow[r, "\gamma{\rm{'}}_i"'] & \cdot
			\end{tikzcd};
		\end{equation}
		\item$s \circ \zeta_0$ and $t \circ \zeta_n$ are constant paths in $X_0$.
	\end{enumerate}
\end{definition}
\begin{proposition}\label{thin homotopy of X-paths is an equivalence relation}
	Given a Lie groupoid $\mb{X}$, \Cref{Definition: Thin deformation} induces an equivalence relation on $P\mb{X}$.
\end{proposition}
\begin{proof} 
	A thin deformation from a lazy $\mb{X}$-path $\Gamma$ to itself is given by the sequence of constant paths on  the elements of $X_1$ in $\Gamma$. Now,  if  $ \lbrace \zeta_{i} \colon [0,1] \ra X_1 \rbrace_i$  is a thin deformation from $\Gamma$ to another lazy $\mb{X}$-path $\Gamma'$, then the idenitites $s \circ \zeta_i^{-1}=(s \circ \zeta_i)^{-1}$ and $t \circ \zeta_i^{-1}=(t \circ \zeta_i)^{-1}$ imply that $ \lbrace \zeta_{i}^{-1} \colon [0,1] \ra X_1 \rbrace_{i=0,1 \cdots,n}$  defined by $r \mapsto \zeta_i(1-r)$, is a thin deformation from $\Gamma'$ to $\Gamma$.
			To show the transitivity, suppose $ \lbrace \delta_{i} \colon [0,1] \ra X_1 \rbrace_i$ is a thin deformnation from $\Gamma'$ to $\Gamma''$, then one verifies that $ \lbrace \delta_{i} * \zeta_i \colon [0,1] \ra X_1 \rbrace_i$  is a thin deformation from $\Gamma$ to $\Gamma''$, where $*$ is the concatenation of paths.
	
%
%
%
%

\end{proof} 
\begin{definition}\label{Definition: Thin homotopy of X-paths}
	A \textit{lazy $\mb{X}$-path thin homotopy} is defined as the equivalence relation on $P\mb{X}$ generated by the equivalence relations in \Cref{Definition: Equivalence of X-paths} and \Cref{thin homotopy of X-paths is an equivalence relation}.
\end{definition}
We denote the corresponding quotient set by $\frac{P\mb{X}}{\sim}$. Explicitly,	a pair of lazy $\mb{X}$-paths (with fixed endpoints) is related by a lazy $\mb{X}$-path thin homotopy if one is obtained from the other by a finite sequence of equivalences and thin deformations. 
\begin{proposition}\label{Propostioni:thin fundamental groupoid of a Lie groupoid}
	Given a Lie groupoid $\mb{X}= [X_1 \rra X_0]$, there is a groupoid $\Pi_{\rm{thin}}(\mb{X})$,  whose object set is $X_0$ and the morphism set is $\frac{P\mb{X}}{\sim}$. The structure maps are given as follows:
	\begin{enumerate}[(i)]
		\item Source: $s: \frac{P\mb{X}}{\sim} \ra X_0$ is defined by $[\Gamma = (\gamma_0, \alpha_1,\gamma_1, \cdots,\alpha_n, \gamma_n)] \mapsto s(\gamma_0);$ 
		\item Target: $t: \frac{P\mb{X}}{\sim} \ra X_0$ is defined by $[\Gamma = (\gamma_0, \alpha_1,\gamma_1, \cdots, \alpha_n, \gamma_n)] \mapsto t(\gamma_n);$ 
		\item Composition: if $s([\Gamma'= (\gamma'_0, \alpha'_1,\gamma'_1, \cdots,\alpha'_n, \gamma'_n)])= t([\Gamma = (\gamma_0, \alpha_1,\gamma_1, \cdots,\alpha_m, \gamma_m)])$, then define
		\begin{equation}\nonumber
			[(\gamma'_0, \alpha'_1,\gamma'_1, \cdots,\alpha'_n, \gamma'_n)] \circ [(\gamma_0, \alpha_1,\gamma_1, \cdots,\alpha_m, \gamma_m)] := [(\gamma_0, \alpha_1,\gamma_1, \cdots,\alpha_m, \gamma'_0 \circ  \gamma_m, \alpha'_1,\gamma'_1, \cdots,\alpha'_n, \gamma'_n)];
		\end{equation}
		\item Unit: $u : X_0 \ra  \frac{P\mb{X}}{\sim}$ is given by $x \mapsto [(1_x, c_x,1_x)]$ where $c_x \colon [0,1] \ra X_0$ is the constant path at $x \in X_0$;
		\item Inverse: $\mathfrak{i} \colon  \frac{P\mb{X}}{\sim} \ra  \frac{P\mb{X}}{\sim}$ is given by
		\begin{equation}\nonumber
			[(\gamma_0, \alpha_1,\gamma_1, \cdots,\gamma_{n-1},\alpha_n, \gamma_n)] \mapsto [(\gamma_n^{-1}, \alpha^{-1}_{n},\gamma^{-1}_{n-1}, \cdots,\gamma^{-1}_1, \alpha^{-1}_1, \gamma^{-1}_0)].
		\end{equation}
	\end{enumerate}
\end{proposition}
\begin{proof}
	A direct consequence of the definition is that $s$ and $t$ are well-defined. One can verify the well-definedness of the composition by considering the following four cases separately.
	\begin{itemize}
		\item[(i)] If $\tilde{\Gamma}'$ is obtained from $\Gamma'$ by an equivalence and if $\tilde{\Gamma}$ is obtained from $\Gamma$ by an equivalence, then $\Gamma' \circ \Gamma$ is lazy $\mb{X}$-path thin homotopic to $\tilde{\Gamma}' \circ \tilde{\Gamma}$.
		\item[(ii)] If $\tilde{\Gamma}'$ is obtained from $\Gamma'$ by a thin deformation and if $\tilde{\Gamma}$ is obtained from $\Gamma$ by a thin deformation, then $\Gamma' \circ \Gamma$ is lazy $\mb{X}$-path thin homotopic to $\tilde{\Gamma}' \circ \tilde{\Gamma}$.
		\item[(iii)] If $\tilde{\Gamma}'$ is obtained from $\Gamma'$ by an equivalence and if $\tilde{\Gamma}$ is obtained from $\Gamma$ by a thin deformation, then $\Gamma' \circ \Gamma$ is lazy $\mb{X}$-path thin homotopic to $\tilde{\Gamma}' \circ \tilde{\Gamma}$.
		\item[(iv)]If $\tilde{\Gamma}'$ is obtained from $\Gamma'$ by a thin deformation and if $\tilde{\Gamma}$ is obtained from $\Gamma$ by an equivalence, then $\Gamma' \circ \Gamma$ is lazy $\mb{X}$-path thin homotopic to $\tilde{\Gamma}' \circ \tilde{\Gamma}$.
	\end{itemize}
	\subsection*{Case (i):} 
	Straightforward succesive execution of operations on $\Gamma' \circ \Gamma$ that produce $\tilde{\Gamma}'$  and $\tilde{\Gamma}$ from $\Gamma'$  and $\Gamma$ respectively. 
	\subsection*{Case (ii):} Observe that if $ \lbrace \zeta_{i} \colon [0,1] \ra X_1 \rbrace_{i=0,1 \cdots, n}$ and $ \lbrace \zeta_{i}' \colon [0,1] \ra X_1 \rbrace_{i=0,1 \cdots, n}$ are thin deformations from $\Gamma$ to $\tilde{\Gamma}$ and $\Gamma'$ to $\tilde{\Gamma}'$ respectively, then $ \lbrace \zeta_0, \zeta_1, \cdots,\zeta_{m-1}, \zeta_0' \circ \zeta_m, \zeta'_1,\zeta'_2, \cdots,\zeta_n \rbrace$ is a thin deformation from $\Gamma' \circ \Gamma$ to $\tilde{\Gamma}' \circ \tilde{\Gamma}$. 
	
	\subsection*{Case (iii):}  Let $\lbrace \zeta_{i} \colon [0,1] \ra X_1 \rbrace_{i=0,1 \cdots,n}$ be a thin deformation from $\Gamma$ to $\tilde{\Gamma}$ and let $\epsilon$ be an equivalence operation on $\Gamma'$ to obtain $\tilde{\Gamma}'$. Then, $d:=\lbrace \zeta_0, \zeta_1, \cdots ,\zeta_m, c_{\gamma_0'},c_{\gamma_1'}, \cdots, c_{\gamma_n'}   \rbrace $ is a thin deformation from $\Gamma' \circ \Gamma$ to $\Gamma' \circ \tilde{\Gamma}$, where $\lbrace c_{{\gamma_i'}} \rbrace_{i=0,1 \cdots n}$ are constant paths in $X_1$ defined by $c_{\gamma_i'}(r)=\gamma_i'$ for all $r \in [0,1]$. Applying the equivalence operation $\epsilon$ on $\Gamma' \circ \tilde{\Gamma}$  we will obtain $\tilde{\Gamma}' \circ \tilde{\Gamma}$.

\subsection*{Case (iv):} It can be verified in the same way as Case(iii).

The verification of associativity of the composition and its compatibility with the unit and inverse map are straightforawrd.
	
%
	
\end{proof}

\begin{definition} \label{Thin homotopy groupod of a Lie groupoid}
	Given a Lie groupoid $\mb{X}= [X_1 \rra X_0]$, the groupoid $\Pi_{\rm{thin}}(\mb{X})$ is called the \textit{thin fundamental groupoid of the Lie groupoid $\mb{X}$.}
\end{definition}
As an element $x$ and a path $\alpha$ in a manifold $M$ can respectively be identified with lazy $[M \rra M]$-paths $(1_x, c_x, 1_x)$ and $(1_{\alpha(0)}, \alpha, 1_{\alpha(1)})$ for a constant path $c_x$ at $x$, then $\Pi_{{\rm{thin}}}(M \rra M)$ reduces to $\Pi_{{\rm{thin}}}(M)$.

In the following subsection, we will show that $\Pi_{\rm{thin}}(\mb{X})$ is a diffeological groupoid for any Lie groupoid $\mb{X}$.

\subsection{Smoothness of the thin fundamental groupoid of a Lie groupoid}\label{Subsection Smoothness of thin fundamental groupoid of a Lie groupoid}
Given a Lie groupoid $\mb{X}$, define an infinite sequence of sets as $\lbrace P\mb{X}_n) \rbrace_{n \in \mb{N} \cup \lbrace 0 \rbrace}$, where $P\mb{X}_0:=X_1$ and $P\mb{X}_n:= X_1 \times_{t,X_0,ev_0}PX_0 \times_{ev_1,X_0,s}X_1 \times_{t,X_0,ev_0} \cdots \times_{t,X_0,ev_0}PX_0 \times_{ev_1,X_0,s} \times X_1$, for $n \in \mb{N}$. It is clear from the definition itself that $P\mb{X}_n$ has a natural identification with the set of lazy $\mb{X}$-paths of order $n$ for each $n \in \mb{N}$. Hence, as a set
$P\mb{X}= \cup_{i \in \mb{N} \cup \lbrace 0 \rbrace}P\mb{X}_i=\sqcup_{i \in \mb{N} \cup \lbrace 0 \rbrace}P\mb{X}_i$. 
\begin{proposition}\label{PX is a diffeological space}
	For any Lie groupoid $\mb{X}$, the set of lazy $\mb{X}$-paths $P\mb{X}$ is a diffeological space.
\end{proposition}
\begin{proof} 
	Since by \Cref{manifold diffeology} and \Cref{Path space diffeology} respectively, source-target and evaluation maps are maps of diffeological spaces, the fiber product diffeology (\Cref{Fibre product diffeology}) ensures that for each $n \in \mb{N}$, $P\mb{X}_n$ is a diffeologial space with diffeology given by
	$D_{P\mb{X}_n}:= \biggl\{(p^0_{X_1},p^{1}_{PX_0},p^1_{X_1}, \cdots,p^{n}_{PX_0},p^n_{X_1}) \in D_{X_1} \times D_{PX_0} \times D_{X_1} \times \cdots \times D_{PX_0} \times D_{X_1}: t \circ p^0_{X_1}=ev_0 \circ p^{1}_{PX_0}, ev_1 \circ p^{1}_{PX_0}=s \circ p^1_{X_1}, \cdots, ev_1 \circ p^{n}_{PX_0} =s \circ p^n_{X_1} \biggr\}$. Then, \Cref{sum diffeology}, induces  the sum diffeology on $P\mb{X}$.
\end{proof}
\begin{corollary}\label{Diffeology of quotient of Xpath space}
	$\frac{P\mb{X}}{\sim}$ is a diffeological space.	
\end{corollary}
\begin{proof}
	A direct consequence of \Cref{PX is a diffeological space} and \Cref{quotient diffeology}.
\end{proof}

\begin{lemma}\label{composition of PX}
	For any Lie groupoid $\mb{X}$, 
	\begin{itemize}
		\item[(a)] the multiplication map 
		\begin{equation}\nonumber
			\begin{split}
				\tilde{m} \colon  & P\mb{X} \times_{s,X_0,t} P\mb{X} \ra P\mb{X}\\
				& \Big((\gamma'_0, \alpha'_1, \cdots,\alpha'_n, \gamma'_n), (\gamma_0, \alpha_1, \cdots,\alpha_m, \gamma_m) \Big) \mapsto (\gamma_0, \alpha_1, \cdots,\alpha_m, \gamma'_0 \circ  \gamma_m, \alpha'_1, \cdots,\alpha'_n, \gamma'_n),
			\end{split}
		\end{equation}
		\item[(b)] the unit map 
		\begin{equation}\nonumber
			\begin{split}
				\tilde{u} \colon & X_0 \ra P\mb{X}\\
				& x \mapsto (1_x,c_x,1_x),	
			\end{split}
		\end{equation}
		where $c_x$ is the constant path at $x$,
		\item[(c)] the inverse map
		\begin{equation}\nonumber
			\begin{split}
				\tilde{\mathfrak{i}} \colon & P\mb{X} \ra P\mb{X}\\
				& (\gamma_0, \alpha_1,\gamma_1, \cdots,\gamma_{n-1},\alpha_n, \gamma_n) \mapsto (\gamma_n^{-1}, \alpha^{-1}_{n},\gamma^{-1}_{n-1}, \cdots,\gamma^{-1}_1, \alpha^{-1}_1, \gamma^{-1}_0),
			\end{split}
		\end{equation}
	\end{itemize}
	are maps of diffeological spaces.		
\end{lemma}

\begin{proof}:
Let $(p,p') \colon U \ra P\mb{X} \times_{s,X_0,t} P\mb{X}$ be a plot in $P\mb{X} \times_{s,X_0,t} P\mb{X}$ and $x \in U$. By the definition of sum diffeolgy \Cref{sum diffeology}, there exist open neighbourhoods $U^{n}_x$ and $U^{n'}_x$ of $x$ and indexes $n,n' \in \mb{N} \cup \lbrace 0 \rbrace$ such that $p|_{U^{n}_x} \in D_{P\mb{X}_n}$ and $p|_{U^{n'}_x} \in D_{P\mb{X}_{n'}}$. Hence, by the definition of $\tilde{m}$, it is clear that $ \big(m \circ (p,p') \big)|_{U_x} \in D_{P\mb{X}_{n+n'}}$, where $U_x=U^{n'}_x \cap U^{n}_x$, and hence, $\tilde{m}$ is smooth, and this proves (a). \\(b) and (c) can be proved using similar techniques used in the proof of (a).

\end{proof}
We show next that the thin fundamental groupoid of a Lie groupoid (\Cref{Thin homotopy groupod of a Lie groupoid}) is a \textit{diffeological groupoid} i.e., a groupoid object in the category of diffeological spaces, (\textbf{8.3}, \cite{iglesias2013diffeology}).
\begin{theorem}\label{Thin fundamental groupoid of a Lie groupoid is a diffeological groupoid}
	The thin fundamental groupoid $\Pi_{\rm{thin}}(\mb{X})$ of a Lie groupoid $\mb{X}$ is a diffeological groupoid.
\end{theorem}
\begin{proof}
	Since we have already shown $\frac{P\mb{X}}{\sim}$ is a diffeological space (\Cref{Diffeology of quotient of Xpath space}), it remains to be shown that the structure maps descent to maps of diffeological spaces.
	
	\Cref{Technical 1} guarantee that the source-target are maps of diffeologial spaces. Now, suppose $(p_1,p_2) \colon U \ra \frac{P {X}}{\sim} \times_{s,X_0,t} \frac{P \mb{X}}{\sim}$ is a plot of $\frac{P \mb{X}}{\sim} \times_{s,X_0,t} \frac{P \mb{X}}{\sim}$. Hence, by the definition of quotient diffeology, there exists a cover $ \lbrace U_i \rbrace$ of $U$ such that for each $i$, we have 
	\begin{itemize}
		\item a plot $\bar{p}^{i}_1 \colon U_i \ra P \mb{X}$ and $q \circ \bar{p}^{i}_1 =p_1|_{U_i}$ and
		\item  a plot $\bar{p}^{i}_2 \colon U_i \ra P\mb{X}$ and $q \circ \bar{p}^{i}_2 =p_2|_{U_i}$,
	\end{itemize}
	where $q$ is the quotient map. Hence,  clearly $(\bar{p}^{i}_1, \bar{p}^{i}_2 ) \colon U_i \ra P\mb{X} \times_{s,X_0,t} P\mb{X}$ is a plot of $P\mb{X} \times_{s,X_0,t} P\mb{X}$. The commutativity of the diagram below
	\[
	\begin{tikzcd}
		P\mb{X} \times_{s,X_0,t} P\mb{X} \arrow[d, "{(q,q)}"'] \arrow[r, "\tilde{m}"] & P\mb{X} \arrow[d, "{q}"] \\
		\frac{P \mb{X}}{\sim} \times_{s,X_0,t} \frac{P \mb{X}}{\sim} \arrow[r, "m"']                & \frac{P \mb{X}}{\sim}               
	\end{tikzcd}\]
	ensures that the composition is a map of diffeological spaces. Here, $\tilde{m}$ is the multiplication map defined in \Cref{composition of PX}.
	
 Smoothness of the unit and the inverse maps can be verified in a similar fashion.
\end{proof}
For a morphism of Lie groupoids $F \colon \mb{X} \ra \mb{Y}$, one has the induced morphism of diffeological groupoids (see \textbf{8.3}, \cite{iglesias2013diffeology} for the definition of the morphism of diffeological groupoids) between the respective thin fundamental grouipoids  $F_{\rm{thin}} \colon \Pi_{\rm{thin}}(\mb{Y}) \ra \Pi_{\rm{thin}}(\mb{X})$.
\begin{lemma}\label{Morphism of thin fundamental groupoid}
	 $$F_{\rm{thin}} \colon \Pi_{\rm{thin}}(\mb{Y}) \ra \Pi_{\rm{thin}}(\mb{X})$$ is defined as $y \mapsto F(y)$ for each $y \in Y_0$, and a class of lazy $\mb{Y}$-path $[\Gamma]:=[(\gamma_0, \alpha_1, \gamma_1, \cdots \alpha_n,\gamma_n)]$ goes to the class of lazy $\mb{X}$-path $[F(\Gamma)]:=[ \big( F(\gamma_0), F \circ \alpha_1, F(\gamma_1), \cdots,F \circ \alpha_n, F(\gamma_n) \big)]$. 
\end{lemma}
A similar result exists for fundamental groupoids of Lie groupoids (See \cite{colman20111}).

\section{Parallel Transport on quasi-principal 2-bundles}\label{Section: Parallel Transport on principal 2-bundles}
In this section, we investigate the  parallel transport along lazy Haefliger paths in the framework of quasi-principal 2-bundles. More interestingly, this construction leads to a parallel transport functor defined on the thin homotopy classes of lazy Haefliger paths. Moreover, this parallel transport functor enjoys certain smoothness properties (\Cref{Smoothness of parallel transport functor}), as well as it is well behaved with respect to the pullback (\Cref{Pullback naturality} and the connection preserving morphims of quasi-principal 2-bundles (\Cref{Naturality of parallel transport}).

We will develop the notion of parallel transport in three steps.

Consider a strict connection $\omega$ on a quasi-principal $\mb{G}$-bundle $(\pi \colon \mb{E} \ra \mb{X} ,\mc{C})$ and a lazy $\mb{X}$-path $\Gamma= (\gamma_0, \alpha_1,\gamma_1, \cdots,\alpha_n, \gamma_n)$. 
\begin{itemize}
	\item[Step 1:] For every element $\gamma_i \colon x_i \ra y_i$ in $X_1$, we will define a $\mb{G}$-equivariant isomorhism of Lie groupoids $T_{\mc{C}, \pi} \colon \pi^{-1}(y_i) \ra \pi^{-1}(x_i)$ induced by the quasi connection $\mc{C}$.
	\item[Step 2:] For every path $\alpha_i \colon x'_i \ra y'_i $,  we define a $\mb{G}$-equivariant isomorphism of Lie groupoids $T_{\omega}^{\alpha_i} \colon \pi^{-1}(x'_i) \ra \pi^{-1}(y'_i)$ induced from the strict connection $\omega$.
	\item[Step 3:] We will compose the above $\mb{G}$-equivariant isomorphisms of Lie groupoids succesively to get $$T_{(\Gamma, \mc{C}, \omega)}: = T_{\mc{C}, \pi}(\gamma_n^{-1}) \circ  T_{\omega}^{\alpha_n} \circ \cdots \circ T_{\omega}^{\alpha_1} \circ  T_{\mc{C}, \pi}(\gamma_0^{-1}).$$
\end{itemize}
The novelty of this approach lies in showing how a differential geometric notion of connection induced horizontal path lifting property combines with a purely categorical notion of cartesian lifts in a fibered category to produce a notion of a higher differential geometric version of the classical parallel transport functor (\Cref{Theorem: Parallel transport on 2-bundles}). The non-trivial aspect of the construction lies in proving its lazy $\mb{X}$-path thin homotopy invariance (\Cref{Proposition: Equivalence invariance of parallel transport} and \Cref{Proposition:  X-path thin homotopy invariance of parallel transport}).

\textbf{Step 1:} \label{Classical fibrations}

We start with the following straightforward observation.
\begin{proposition}\label{Lemma: Quasiprincipal 2-bundle is a fibered category}
	The underlying projection functor of a quasi-principal $\mb{G}$-bundle $(\pi \colon \mb{E} \ra \mb{X}, \mc{C})$ is a fibered category over $\mb{X}$ equipped with a cleavage $\mc{K}_{\mc{C}}:= \lbrace \mc{C}(\gamma^{-1},p)^{-1} \colon (\gamma,p) \in t^{*}E_0 \rbrace$,(see \Cref{subsection fibered categories and pseudofunctors}).
\end{proposition}
\begin{definition}\label{Lie 2-group torsor}
	For a Lie 2-group $\mb{G}$, a \textit{$\mb{G}$-torsor} is defined as a Lie groupoid $\mb{X}$ with an action of $\mb{G}$ such that manifolds $X_0$ and $X_1$ are $G_0$-torsor and $G_1$-torsor respectively. Collection of $\mb{G}$-torsors, $\mb{G}$-equivariant morphisms of Lie groupoids and $\mb{G}$-equivariant natural transformations form a 2-groupoid which we denote by $\mb{G}$-Tor.
\end{definition}
\begin{example}
If $\pi \colon \mb{E} \ra \mb{X}$ is a principal $\mb{G}$-bundle over a Lie groupoid $\mb{X}$, then for any $x \in X_0$, the \textit{fibre} $\pi^{-1}(x):=[\pi_1^{-1}(1_x) \rra \pi_0^{-1}(x_0)]$ is a $\mb{G}$-torsor.
\end{example}

Now, as a direct consequence of \Cref{Lemma: Quasiprincipal 2-bundle is a fibered category}, we get the following:
\begin{proposition}\label{T_C}
	For a quasi-principal $\mb{G}$-bundle $(\pi \colon \mb{E} \ra \mb{X}, \mc{C})$ over a Lie groupoid $\mb{X}$, there is an associated \textit{$\mb{G}$-${\rm{Tor}}$ valued pseudofunctor } $T_{\mc{C}} \colon \mb{X}^{\rm{op}} \ra \mb{G}$-Tor. Precisely,
	\begin{itemize}
		\item[(a)]each $x \in X_0$ is assigned to the $\mb{G}$-Torsor $T_{\mc{C}}(x):= \pi^{-1}(x)$,
		\item[(b)]each morphism $x \xrightarrow {\gamma} y$ is assigned  to an isomorphism of $\mb{G}$-torsors
		\begin{equation}\label{Pseudomor}
			\begin{split}
				T_{\mc{C}}(\gamma) \colon & \pi^{-1}(y) \ra \pi^{-1}(x)\\
				& p \mapsto \mu_{\mc{C}}(\gamma^{-1},p));\\
				& (p \xrightarrow[]{\zeta}q) \mapsto \mc{C}(\gamma^{-1},q) \circ \zeta \circ \big(\mc{C}(\gamma^{-1},p) \big)^{-1},
			\end{split}
		\end{equation} 
		\item[(c)]for each $x \in X_0$, we have a smooth $\mb{G}$-equivariant natural isomorphism 
		\begin{equation}\label{Pseudonat1}
			\begin{split}
				I_{x} \colon & T_{\mc{C}}(1_x) \Longrightarrow 1_{\pi^{-1}(x)}\\
				& p \mapsto \bigg( \mu_{\mc{C}}(1_x,p) \xrightarrow[]{\mc{C}(1_x,p)^{-1}}p \bigg),
			\end{split}
		\end{equation}
		\item[(d)]for each pair of composable arrows 
		\begin{tikzcd}
			x \arrow[r, "\gamma_{1}"] & y \arrow[r, "\gamma_2"] & z
		\end{tikzcd},
		we have a smooth $\mb{G}$-equivariant natural isomorphism 
		\begin{equation}\label{pseudonat2}
			\begin{split}
				\alpha_{\gamma_1, \gamma_2} \colon & T_{\mc{C}}(\gamma_1) \circ T_{\mc{C}}(\gamma_2) \Longrightarrow T_{\mc{C}}(\gamma_2 \circ \gamma_1)\\
				& p \mapsto \mc{C}(\gamma_1^{-1} \circ \gamma_2^{-1},p) \circ \mc{C}(\gamma_2^{-1},p)^{-1} \circ \mc{C}\big(\gamma_1^{-1}, t(\mc{C}(\gamma_2^{-1},p)) \big)^{-1},
			\end{split}
		\end{equation}
	\end{itemize}
	such that $\alpha_{\gamma_1,\gamma_2}$ and $I_x$ satisfy the necessary coherence laws of \Cref{Coherence diagram 2} and \Cref{Coherence diagram 1} respectively.
\end{proposition}

\textbf{Step 2:} 

Let $A$ be a connection on an ordinary principal $G$-bundle $\pi : E \ra M$ over a smooth manifold $M$. Then,  given a  smooth path  $\alpha:[0,1] \ra M$, for each point $p \in \pi^{-1}(\alpha(0))$, the unique horizontal lift of the path $\alpha$ starting from $p$ will be denoted by $\tilde{\alpha}_{A}^{p}$. In turn we have a map ${\rm{Tr}}_A^{\alpha} \colon \pi^{-1}(\alpha(0)) \ra \pi^{-1}(\alpha(1))$ defined by ${\rm{Tr}}_A^{\alpha}(p):=\tilde{\alpha}_{A}^{p}(1)$, the so called \textit{parallel transport map}, given by $p \mapsto \tilde{\alpha}_{A}^{p}(1)$.

Next, we observe a consequence of the functoriality of strict connections (\Cref{strict and semistrict connections}).
\begin{lemma}\label{Lemma: source-target strict transport}
	For a Lie 2-group $\mb{G}$, let $\pi: \mb{E} \ra \mb{X}$ be a prinicpal $\mb{G}$-bundle over a Lie groupoid $\mb{X}$. Any strict connection $\omega: T\mb{E} \ra L(\mb{G})$ induces the follwoing:
	
	For any path $\zeta \colon [0,1] \ra X_1$ and $\alpha:[0,1] \ra X_0$, we have the following identities:
	\begin{itemize}
		\item[(1)] ${\rm{Tr}}_{\omega_0}^{s \circ \zeta}(s (\delta)) = s ({\rm{Tr}}_{\omega_1}^{\zeta}( \delta))$ for each $\delta \in \pi_1^{-1}(\zeta(0))$.
		\item[(2)]  ${\rm{Tr}}_{\omega_0}^{t \circ \zeta}(t (\delta)) = t ({\rm{Tr}}_{\omega_1}^{\zeta}( \delta))$ for each $\delta \in \pi_1^{-1}(\zeta(0))$.
		\item[(3)] ${\rm{Tr}}_{\omega_1}^{u \circ \alpha}(u(p))= u ({{\rm{Tr}}_{\omega_0}^{\alpha}(p)})$ for each $p \in \pi_0^{-1}(\alpha(0))$.
	\end{itemize}
\end{lemma}
\begin{proof}
	Observe that to prove the above identities, it is sufficient to show the following:
		\begin{enumerate}[(i)]
			\item $\widetilde{s \circ \zeta}_{\omega_0}^{s(\delta)}= s \circ \tilde{\zeta}_{\omega_1}^{\delta}$,
			\item $\widetilde{t \circ \zeta}_{\omega_0}^{t(\delta)}= t \circ \tilde{\zeta}_{\omega_1}^{\delta}$,
			\item $\widetilde{u \circ \alpha}_{\omega_1}^{1_{p}}= u \circ {\tilde{\alpha}_{\omega_0}^{p}}$,
	\end{enumerate}
	which one can verify from the functoriality of $\omega$ and $\pi$. 
%
\end{proof}
We obtain the following consequence of the above lemma.
\begin{proposition}\label{Proposition: Parallel transport on principal 2-bundles}
	For a Lie 2-group $\mb{G}$, let $\pi: \mb{E} \ra \mb{X}$ be a prinicpal $\mb{G}$-bundle over a Lie groupoid $\mb{X}$  with a strict connection $\omega$. Then for any given  path $\alpha: [0,1] \ra X_0$, there is a $\mb{G}$-equivariant isomorphism of Lie groupoids
	\begin{equation}\nonumber
		\begin{split}
			T^{\alpha}_{\omega} \colon & \pi^{-1}(x) \ra \pi^{1}(y)\\
			& p \mapsto {\rm{Tr}}_{\omega_0}^\alpha(p)\\
			& \gamma \mapsto {\rm{Tr}}_{\omega_1}^ {u \circ \alpha}(\gamma)
		\end{split}
	\end{equation}
	for all $p \in \pi_{0}^{-1}(x)$ and $\gamma \in \pi_{1}^{-1}(1_{x})$, where $\alpha(0)=x$ and $\alpha(1)=y$.
\end{proposition}

\textbf{Step 3:}
Combining the results of step 1 and step 2, we arrive at the following definition:
\begin{definition}\label{Definition: Parallel transport of X-paths}
	Given a quasi-principal $\mb{G}$-bundle  $(\pi: \mb{E} \ra \mb{X}, \mc{C})$, a strict connection $\omega$ and a lazy $\mb{X}$-path $\Gamma :=(\gamma_0, \alpha_1, \gamma_1, \cdots, \alpha_n, \gamma_n)$ from $x$ to $y$,  the $\mb{G}$-equivariant isomorphism of Lie groupoids $T_{(\Gamma, \mc{C}, \omega)}: = T_{\mc{C}, \pi}(\gamma_n^{-1}) \circ  T_{\omega}^{\alpha_n} \circ \cdots \circ T_{\omega}^{\alpha_1} \circ  T_{\mc{C}, \pi}(\gamma_0^{-1}) $ is defined as the \textit{$(\mc{C}, \omega)$-parallel transport along the lazy $\mb{X}$-path $\Gamma$.}
\end{definition}
This parallel transport enjoys the crucial properties of functoriality and the lazy $\mb{X}$-path thin homotopy, which we will establish in the following subsection.
\begin{example}[Classical principal bundle]
		Let $\pi \colon [E \rra E] \ra [M \rra M]$ be a principal $[G \rra G]$ bundle  over a discrete Lie groupoid $[M \rra M]$, equipped with the strict connection $\omega:= (\omega,\omega)$ (\Cref{Classical as 2-connection}) and the unique categorical connection $\mc{C}$ (\Cref{Cat connection over discrete base}). Then $T_{\Gamma, \mc{C}, \omega}= T^{\alpha}_{\omega}$.
		
	\end{example}
	\begin{example}[Principal 2-bundle over a manifold]
		One can show that the principal 2-bundle defined in \Cref{Definition Principal Lie 2-group bundle over a Lie groupoid} coincides with the definition of principal 2-bundle over a manifold $M$ as defined in \textbf{Definition 3.1.1}, \cite{MR3894086} when our base Lie groupoid is of the form $[M \rra M]$. Further in \cite{MR3917427} a notion of parallel transport along a path $\alpha$ on a manifold $M$ was introduced in terms of a Lie 2-group equivariant anafunctor. In our setup the Lie 2-group equivariant anafunctor corresponding to the parallel transport on a quasi-principal $\mb{G}$-bundle (\Cref{Definition: Parallel transport of X-paths}) for a lazy $[M \rra M]$-path $(1_{\alpha(0)}, \alpha, 1_{\alpha(1)})$ relates with that of \textbf{Propoition 3.26}, \cite{MR3917427}.

\end{example}
\subsection{Parallel transport functor on a quasi-principal 2-bundle}\label{Parallel transport functor on a quasi-principal 2-bundle}
 We start by showing a lazy $\mb{X}$-path thin homotopy invariance of the parallel transport, defined in \Cref{Definition: Parallel transport of X-paths}.

%

\begin{proposition}\label{Proposition: Equivalence invariance of parallel transport}
	 Let $(\pi: \mb{E} \ra \mb{X}, \mc{C})$ be a quasi-principal $\mb{G}:=[H \rtimes_{\alpha} G \rra G]$-bundle with a strict connection $\omega: T\mb{E} \ra L(\mb{G})$. If a lazy $\mb{X}$-path $\Gamma  :=(\gamma_0, \alpha_1,\gamma_1, \cdots,\alpha_n, \gamma_n)$ is equivalent (see \Cref{Definition: Equivalence of X-paths}) to a lazy $\mb{X}$-path $\Gamma'$, then there is a smooth $\mb{G}$-equivariant natural isomorphism between $T_{(\Gamma, \mc{C}, \omega)}$ and $T_{(\Gamma', \mc{C}, \omega)}$.
\end{proposition}
\begin{proof}
	For our purpose, it is sufficient to verify only the following three cases, i.e :

	(A) if $\Gamma'$ is obtained from $\Gamma $ by the $(1)$ of \Cref{Definition: Equivalence of X-paths}, then there is a smooth $\mb{G}$-equivariant natural isomorphism between $T_{(\Gamma, \mc{C}, \omega)}$ and $T_{(\Gamma', \mc{C}, \omega)}$. (B) and (C) are likewise for respectively (2) and (3) in \Cref{Definition: Equivalence of X-paths}.
	
	(A) and (B) directly follows from \Cref{pseudonat2} and \Cref{Pseudonat1} respectively. (C) demands a little more work.
	\subsection*{Proof of(C)} 
	Let $\Gamma'$ be obtained from $\Gamma :=(\gamma_0, \alpha_1,\gamma_1, \cdots,\alpha_n, \gamma_n)$  by $(3) $ of \Cref{Definition: Equivalence of X-paths}. That is, given a path $\zeta_i \colon [0,1] \ra X_1$ with sitting instants, such that $s \circ \zeta_i= \alpha_i$, we replace $\alpha_i$ by $t \circ \zeta_i$,  $\gamma_{i-1}$ by $\zeta_i (0) \circ \gamma_{i-1}$ and $\gamma_{i}$ by $\gamma_i \circ (\zeta_i(1))^{-1}$, $i \in \lbrace 1,2, \cdots, n \rbrace$ to obtain $\Gamma'$. We have to show that
	\begin{equation}\label{Magic}
		T_{\omega}^{\alpha_i} \cong T_{\mc{C}}(\gamma') \circ T_{\omega}^{t \circ \zeta_i} \circ T_{\mc{C}}(\gamma^{-1}),
	\end{equation}
	where $\cong$ is a smooth $\mb{G}$-equivariant natural isomorphism, and $\gamma:= x \xrightarrow {\zeta_{i}(0)} y$, $\gamma':= x' \xrightarrow {\zeta_{i}(1)} y'$ are elements of $X_1$. By the repetitive use of \Cref{T_C}, this is equivalent to showing that  given a square
	\[
	\begin{tikzcd}
		x \arrow[r, "s \circ \zeta_i", dotted] \arrow[d, phantom] \arrow[d, "\gamma"'] & x' \arrow[d, "\gamma'"] \arrow[d] \arrow[d] \arrow[d] \\
		y \arrow[r, "t \circ \zeta_i"', dotted]                                   & y'                                             
	\end{tikzcd}\]
	
	the following square
	\[
	\begin{tikzcd}
		\pi^{-1}(y) \arrow[r, "T_{\omega}^{t \circ \zeta_i}"] \arrow[d, phantom] \arrow[d, "T_{\mc{C}}(\gamma)"'] & \pi^{-1}(y') \arrow[d, "T_{\mc{C}}(\gamma')"] \arrow[d] \arrow[d] \arrow[d] \\
		\pi^{-1}(x) \arrow[r, "T_{\omega}^{s \circ \zeta_i}"']                                   & \pi^{-1}(x')                                            
	\end{tikzcd}\]
	
	commutes upto a smooth $\mb{G}$-equivariant natural isomorphism. Here, the dotted lines represent the paths $s \circ \zeta_i: [0,1] \ra X_0$ and $t \circ \zeta_i: [0,1] \ra X_0$. We claim that the desired smooth $\mb{G}$-equivariant natural isomorphism $\eta \colon  T_{\mc{C}}(\gamma') \circ T_{\omega}^{t \circ \zeta_i} \Longrightarrow T_{\omega}^{s \circ \zeta_i} \circ T_{\mc{C}}(\gamma)$ is given by
	\begin{equation}\label{equivalence invariance}
		p \mapsto \eta_{p} := 1_{\mu_{\mc{C}} \big( \gamma'^{-1}, {\rm{Tr}}_{\omega_0}^{t \circ \zeta_i}(p) \big)} (h_p,e),
	\end{equation}
	where $h_p$ is the unique element in $H$ such that 
	\begin{equation}\label{Equivalence h}
		\mc{C}\big( \gamma'^{-1}, {\rm{Tr}}_{\omega_0}^{t \circ \zeta_i}(p) \big)(h_p,e)= {\rm{Tr}}^{\mathfrak{i} \circ \zeta_{i}}_{\omega_1}(\mc{C}(\gamma^{-1},p)),
	\end{equation}
	where $\mathfrak{i} \colon X_1 \ra X_1$ is the inverse map. 
	 Smoothness of the map $p \mapsto \eta_p$ and the source consistency are obvious. 
	
	We verify the target consistency by observing
	\begin{equation}\nonumber
		\begin{split}
			&t(\eta_p)\\
			&= t(\mc{C}(\gamma'^{-1}, {\rm{Tr}}_{\omega_0}^{t \circ \zeta_i}(p))(h_p,e))\\
			&= t({\rm{Tr}}^{\mathfrak{i} \circ \zeta_{i}}_{\omega_1}(\mc{C}(\gamma^{-1},p))) \quad [\rm{by} \Cref{Equivalence h}].
		\end{split}
	\end{equation}
	Hence, using \Cref{Lemma: source-target strict transport}, we get 
	\begin{equation}\label{Target}
		t(\eta_p)= t({\rm{Tr}}^{\mathfrak{i} \circ \zeta_{i}}_{\omega_1}(\mc{C}(\gamma^{-1},p)))= {\rm{Tr}}_{\omega_0}^{s \circ \zeta_i}(\mu_{\mc{C}}(\gamma^{-1},p)).
		\end{equation}
	 Since
	\begin{equation}\label{E11}
		(h_p,e)=1_g(h_{pg},e)1_g^{-1}
	\end{equation}
	by \Cref{Equivalence h} and $\eta_{pg}= 1_{\mu_{\mc{C}}(\gamma'^{-1}, {\rm{Tr}}_{\omega_0}^{t \circ \zeta_i}(p))}1_g (h_{pg},e)$ by \Cref{equivalence invariance}, we have $\eta_{pg}=\eta_p1_g$ by \Cref{E11}.
	
	Next, we ensure that $\eta$ satisfies the naturality square, that is for every $  p \xrightarrow {\delta} q \in \pi^{-1}(y)$, 
	
			$\begin{tikzcd}\label{Natural diagram}
					T_{\mc{C}}(\gamma{\rm{'}}) \circ T_{\omega}^{t \circ \zeta_i}(p) \arrow[r, "\eta_{p}"] \arrow[d, " T_{\mc{C}}(\gamma{\rm{'}}) \circ T_{\omega}^{t \circ \zeta_i}(\delta)"'] & T_{\omega}^{s \circ \zeta_i} \circ T_{\mc{C}}(\gamma)(p) \arrow[d, "T_{\omega}^{s \circ \zeta_i} \circ T_{\mc{C}}(\gamma)(\delta)"] \\
					T_{\mc{C}}(\gamma{\rm{'}}) \circ T_{\omega}^{t \circ \zeta_i}(q) \arrow[r, "\eta_{q}"']                & T_{\omega}^{s \circ \zeta_i} \circ T_{\mc{C}}(\gamma)(q)              
				\end{tikzcd}\,\, {\rm{commutes,}}$
		\begin{equation}\label{121}
			\eta_q \circ \big( T_{\mc{C}}(\gamma') \circ T_{\omega}^{t \circ \zeta_i}(\delta) \big)= \big(T_{\omega}^{s \circ \zeta_i} \circ T_{\mc{C}}(\gamma)(\delta) \big) \circ \eta_p.
			\end{equation}
Since $\delta=1_p(h,e)$ for a unique $h \in H$, we have
\begin{equation}\label{111}
	q= p \tau(h),
\end{equation}
\begin{equation}\label{211}
	T_{\mc{C}}(\gamma') \circ T_{\omega}^{t \circ \zeta_i}(\delta) \big)= 1_{\mu_{C}(\gamma'^{-1}, {\rm{Tr}}_{\omega_0}^{t \circ \zeta_i}(p))}(h, e),
\end{equation} and
\begin{equation}\label{311}
	{\rm{Tr}}_{\omega_0}^{s \circ \zeta_i} \circ T_{\mc{C}}(\gamma)(\delta)= 1_{{\rm{Tr}}_{\omega_0}^{s \circ \zeta_i}(\mu_{\mc{C}}(\gamma^{-1},p))}(h,e) .
\end{equation} 
	By comparing the left-hand 
	\begin{equation}\nonumber
		\begin{split}
			&\eta_q \circ \big( T_{\mc{C}}(\gamma') \circ T_{\omega}^{t \circ \zeta_i}(\delta) \big)\\
			&= \underbrace{1_{\mu_{\mc{C}}(\gamma'^{-1}, {\rm{Tr}}_{\omega_0}^{t \circ \zeta_i}(p))} (h_p,e) (e, \tau(h))}_{\quad \eta_q= \eta_p 1_{\tau(h)} [\textit{\Cref{111}}]} \circ \underbrace{1_{\mu_{C}(\gamma'^{-1}, {\rm{Tr}}_{\omega_0}^{t \circ \zeta_i}(p))}(h, e)}_{{\textit{by \Cref{211}}.}} \\
			&= 1_{\mu_{\mc{C}}(\gamma'^{-1}, {\rm{Tr}}_{\omega_0}^{t \circ \zeta_i}(p))}(h_ph,e) \quad [\textit{by functoriality of the action}]
		\end{split}
	\end{equation}
	and the right-hand sides
	\begin{equation}\nonumber
		\begin{split}
			& {\rm{Tr}}_{\omega_0}^{s \circ \zeta_i} \circ T_{\mc{C}}(\gamma)(\delta) \circ \eta_p\\
			&= \underbrace{1_{{\rm{Tr}}_{\omega_0}^{s \circ \zeta_i}(\mu_{\mc{C}}(\gamma^{-1},p))}(h,e)}_{ {\textit{by \Cref{311}}}} \circ \underbrace{1_{\mu_{\mc{C}}(\gamma'^{-1}, {\rm{Tr}}_{\omega_0}^{t \circ \zeta_i}(p))} (h_p,e)}_{\textit{by the definition in \Cref{equivalence invariance}}} \\
			&= 1_{t(\eta(p))}(h,e) \circ 1_{\mu_{\mc{C}}(\gamma'^{-1}, {\rm{Tr}}_{\omega_0}^{t \circ \zeta_i}(p))} (h_p,e)  \quad [\textit{As} \, \, t(\eta_p)= {\rm{Tr}}_{\omega_0}^{s \circ \zeta_i}(\mu_{\mc{C}}(\gamma^{-1},p)), \textit{by \Cref{Target}} ]\\
			&=  \underbrace{1_{ \mu_{\mc{C}}(\gamma'^{-1}, \rm{Tr}_{\omega_0}^{t \circ \zeta_i}(p))) \tau(h_p)}}_{{\textit{by \Cref{equivalence invariance}}}}(h,e) \circ 1_{\mu_{\mc{C}}(\gamma'^{-1}, \rm{Tr}_{\omega_0}^{t \circ \zeta_i}(p))} (h_p,e) \\
			&= \underbrace{1_{ \mu_{\mc{C}}(\gamma'^{-1}, \rm{Tr}_{\omega_0}^{t \circ \zeta_i}(p)))}(e,\tau(h_p))}_{\textit{by functoriality of the action.}}(h,e) \circ 1_{\mu_{\mc{C}}(\gamma'^{-1}, \rm{Tr}_{\omega_0}^{t \circ \zeta_i}(p))}(h_p,e) \\
			&= 1_{ \mu_{\mc{C}}(\gamma'^{-1}, \rm{Tr}_{\omega_0}^{t \circ \zeta_i}(p)))}\underbrace{(h_phh_p^{-1}, \tau(h_p))}_{{\textit{by \Cref{E:Peiffer}.}}} \circ 1_{\mu_{\mc{C}}(\gamma'^{-1}, \rm{Tr}_{\omega_0}^{t \circ \zeta_i}(p))} (h_p,e)\\
			&= 1_{\mu_{\mc{C}}(\gamma'^{-1}, \rm{Tr}_{\omega_0}^{t \circ \zeta_i}(p))}(h_ph,e) [\textit{by functoriality of the action}]
		\end{split}
	\end{equation}
	of the \Cref{121}, we conclude the naturality of $\eta$.
\end{proof}
To establish the lazy $\mb{X}$-path thin homotopy invariance of the parallel transport (\Cref{Definition: Parallel transport of X-paths}), it remains to show that it is thin deformation invariant (\Cref{Definition: Thin deformation}).
\begin{proposition}\label{Proposition:  X-path thin homotopy invariance of parallel transport}
	Let $(\pi: \mb{E} \ra \mb{X}, \mc{C})$ be a quasi-principal $\mb{G}:=[H \rtimes_{\alpha} G \rra G]$-bundle with a strict connection $\omega: T\mb{E} \ra L(\mb{G})$. If a lazy $\mb{X}$-path $\Gamma'  :=(\gamma_0, \alpha_1,\gamma_1,\cdots, \alpha_n, \gamma_n)$  is obtained from a lazy $\mb{X}$-path $\Gamma$ via thin deformation, then there is a smooth $\mb{G}$-equivariant natural isomorphism between $T_{(\Gamma, \mc{C}, \omega)}$ and $T_{(\Gamma', \mc{C}, \omega)}$.
\end{proposition}
\begin{proof}
	Consider the case when $\Gamma'$ is obtained from $\Gamma$ by a thin deformation. Let $ \lbrace \zeta_i: I \ra X_1 \rbrace_{i =0,1,\cdots, n}$  be a thin deformation from lazy $\mb{X}$-path $\Gamma =(\gamma_0, \alpha_1,\gamma_1,\cdots, \alpha_n, \gamma_n)$ to a lazy $\mb{X}$-path $\Gamma' =(\gamma_0', \alpha_1',\gamma_1',\cdots, \alpha_n', \gamma_n')$ such that $s(\Gamma)=s(\Gamma')=x$ and $t(\Gamma)=t(\Gamma')=y$, represented by the following diagram:
	
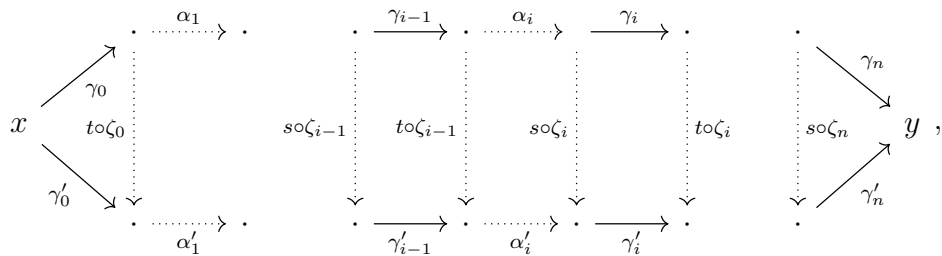
\[\begin{tikzcd}
		& {\cdot} \arrow[r, "\alpha_1", dotted] \arrow[dd, "t \circ \zeta_0"', dotted] & {\cdot} & {\cdot} \arrow[dd, "s \circ \zeta_{i-1}"', dotted] \arrow[r, "\gamma_{i-1}"] & {\cdot} \arrow[r, "\alpha_i", dotted] \arrow[dd, "t \circ \zeta_{i-1}"', dotted] & {} \arrow[dd, "s \circ \zeta_i"', dotted] \arrow[r, "\gamma_i"] & {\cdot} \arrow[dd, "t \circ \zeta_i", dotted] & {\cdot} \arrow[rd, "\gamma_n"] \arrow[dd, "s \circ \zeta_n", dotted] &   \\
		x \arrow[ru, "\gamma_0"'] \arrow[rd, "\gamma_0'"'] &                                                    &    &                                            &                                                    &                                            &                            &                                            & y \\
		& {\cdot} \arrow[r, "\alpha_1'"', dotted]                         & {\cdot} & {\cdot} \arrow[r, "\gamma_{i-1}'"']                         & {\cdot} \arrow[r, "\alpha_i'"', dotted]                         & {\cdot} \arrow[r, "\gamma_i'"']                         & {\cdot}                         & {\cdot} \arrow[ru, "\gamma_n'"']                        &  
	\end{tikzcd},\]
	
	where the solid arrows are elements of $X_1$, and the dotted arrows are paths in $X_0$.  
	
	Let $H_i: I \times I \ra X_0$  be thin homotopies from $\alpha_i$ to $ (s \circ \zeta_i)^{-1}* \alpha'_i *(t \circ \zeta_{i-1})$  for all $i=1,\cdots, n$. Consider, $$u \circ H_i : I \times I \ra X_1,$$ which is a thin homotopy from $u \circ \alpha_i$ to $u \circ \big( (s \circ \zeta_i)^{-1}* \alpha'_i *(t \circ \zeta_{i-1}) \big)$ in $X_1$ for each $i$, as the rank of $u \circ H_i $ is less than rank of $H_i$ at all points. From the thin homotopy invariance of  the parallel transport in classical principal bundles and by the same argument as in for deriving \Cref{Magic} in \Cref{Proposition: Equivalence invariance of parallel transport}, we immediately obtain the following family of equations and smooth $\mb{G}$-equivariant  natural isomorphisms:
	\begin{equation}\label{Magic1}
		T_{\omega}^{\alpha'_i}=T_{\omega}^{( s \circ \zeta_i)} \circ T_{\omega} ^{\alpha_i} \circ T_{\omega}^{(t \circ \zeta_{i-1})^{-1}}
	\end{equation} 
	for all $i=1,2\cdots, .,n,$ and 
	\begin{equation}\label{Magic2} 
		T_{\mc{C}}(\gamma_i'^{-1}) \cong T_{\omega}^{t \circ \zeta_{i}} \circ T_{\mc{C}}(\gamma_i^{-1}) \circ T_{\omega}^{(s \circ \zeta_i)^{-1}},  
	\end{equation} 
	for all $i=1,..,n-1$ respectively. As a consequence of \Cref{Magic1} and \Cref{Magic2}, we conclude $T_{(\Gamma, \mc{C}, \omega)} \cong T_{(\Gamma', \mc{C}, \omega)}$.
\end{proof}
To define our desired parallel transport functor for a quasi principal 2-bundle, we need to introduce a quotient category $\overline{\mb{G} {\rm{-Tor}}}$ of $\mb{G}$-Tor.

\begin{definition}\label{Quotiented G-tor}
	For a Lie 2-group $\mb{G}$, we define the category	$\overline{\mb{G} {\rm{-Tor}}}$ as the quotient category of $\mb{G}$-Tor obtained from the congruence relation given as follows: For each pair of $\mb{G}$-torsors $\mb{X}, \mb{Y}$, the equivalence relation on  $\rm{Hom}_{\mb{G} {\rm{-}} {\rm{Tor}}}(\mb{X}, \mb{Y})$ is given by the existence of a smooth $\mb{G}$-equiavriant natrural isomorphism.
\end{definition}
With the aid of \Cref{Proposition: Equivalence invariance of parallel transport}, \Cref{Proposition:  X-path thin homotopy invariance of parallel transport}, we derive the parallel transport functor.

\begin{theorem}\label{Theorem: Parallel transport on 2-bundles}
	Given a quasi-principal $\mb{G}:=[H \rtimes_{\alpha} G \rra G]$-bundle $(\pi: \mb{E} \ra \mb{X}, \mc{C})$  with a strict connection $\omega: T\mb{E} \ra L(\mb{G})$, there is a functor
	\begin{equation}\nonumber
		\begin{split}
			\mc{T}_{\mc{C}, \omega} \colon & \Pi_{\rm{thin}}(\mb{X}) \ra \overline{\mb{G} \rm{-Tor}}\\
			& x \mapsto \pi^{-1}(x),\\
			& [\Gamma] \mapsto [T_{(\Gamma, \mc{C}, \omega)}].
		\end{split}
	\end{equation}
\end{theorem}
\begin{proof}
	Well-definedness of $\mc{T}_{\mc{C}, \omega}$ follows from  \Cref{Proposition: Equivalence invariance of parallel transport} and \Cref{Proposition:  X-path thin homotopy invariance of parallel transport}. Source and target consistencies of $\mc{T}_{\mc{C}, \omega} $ are obvious. Compatibility with the unit map and the composition follows from  \Cref{Pseudonat1} and  \Cref{pseudonat2}, respectively.
\end{proof}
\begin{definition}\label{Comega parallel transport}
	Given a quasi-principal $\mb{G}:=[H \rtimes_{\alpha} G \rra G]$-bundle $(\pi: \mb{E} \ra \mb{X}, \mc{C})$  with a strict connection $\omega: T\mb{E} \ra L(\mb{G})$, the functor $\mc{T}_{\mc{C}, \omega}$ is defined  as the $(\mc{C}, \omega)$-\textit{parallel transport functor of $(\pi \colon \mb{E} \ra \mb{X}, \mc{C})$}.
\end{definition}
\begin{remark}\label{Classical transport}
For a principal $[G \rra G]$ bundle $\pi \colon [E \rra E] \ra [M \rra M]$ over a discrete Lie groupoid $[M \rra M]$, equipped with the strict connection of the form $\omega:= (\omega,\omega)$ (\Cref{Classical as 2-connection}) and the unique categorical connection $\mc{C}$ (\Cref{Cat connection over discrete base}) the functor $\mc{T}_{\mc{C}, \omega}$, coincides with the classical one.
	
\end{remark}
\subsection{Naturality of the parallel transport functor on a quasi-principal 2-bundle}\label{Naturality of the parallel transport functor}
We start by showing the naturality of  \Cref{Comega parallel transport} with respect to the connection preserving morphisms. 
\begin{proposition}\label{Naturality of parallel transport}
	For any morphism of quasi-principal $\mb{G}$-bundles $$F \colon (\pi \colon \mb{E}' \ra \mb{X}, \mc{C}' ) \ra (\pi \colon \mb{E} \ra \mb{X}, \mc{C})$$ over a Lie groupoid $\mb{X}$, equipped with strict conections $\omega$ and pullback connection $F^{*}\omega$ (See \Cref{Lemma:Pullback connection}) respectively, the functors $\mc{\tau}_{\mc{C}, \omega}$ and $\mc{\tau}_{\mc{C}', F^{*}\omega}$ are naturally isomorphic.
\end{proposition}
\begin{proof}
	Follows from the observation that for every $x \xrightarrow {\gamma} y \in X_1$ and for every path $\alpha$ in $X_0$ (with sitting instants) from $p$ to $q$ respectively, the following two diagrams commute in the category of $\mb{G}$-torsors:
	
	\begin{tikzcd}\label{47}
		\pi'^{-1}(y) \arrow[r, "T_{\mc{C}'}(\gamma)"] \arrow[d, "F|_{\pi'^{-1}(y)}"'] & \pi'^{-1}(x) \arrow[d, "F|_{\pi'^{-1}(x)}"] \\
		\pi^{-1}(y) \arrow[r, "T_{\mc{C}}(\gamma)"']                & \pi^{-1}(x)               
	\end{tikzcd}
	\begin{tikzcd}\label{48}
		\pi'^{-1}(p) \arrow[r, "T_{F^{*}\omega}^{\alpha}"] \arrow[d, "F|_{\pi'^{-1}(p)}"'] & \pi'^{-1}(q) \arrow[d, "F|_{\pi'^{-1}(q)}"] \\
		\pi^{-1}(p) \arrow[r, "T_{\omega}^{\alpha}"']                & \pi^{-1}(q).               
	\end{tikzcd}
	
	The square on the right commutes as classical parallel transports are well behaved with respect to a pullback (For instance, see \textbf{Lemma 3.11}, \cite{MR3521476} ). Whereas the compatibility of $F$ with $\mc{C}$ and $\mc{C}'$ (see \Cref{Groupoid of quasi principal 2-bundles}) ensures that the square on the left commutes.
\end{proof}
For a Lie 2-group $\mb{G}$ and a Lie groupoid $\mb{X}$, let $\rm{Bun}_{\rm{quasi}}^{\nabla}(\mb{X}, \mb{G})$ be the category whose objects are quasi-principal $\mb{G}$-bundles equipped with strict connections over the Lie groupoid $\mb{X}$, and arrows are connection preserving morphisms. Let $\rm{Trans}(\mb{X},\mb{G})$ be the category whose objects are functors $T \colon \Pi_{\rm{thin}}(\mb{X}) \ra \overline{\mb{G} {\rm{-Tor}}}$ and arrows are natural transformations. Then the following is a direct consequence of  \Cref{Naturality of parallel transport}.
\begin{theorem}\label{Equivalence of quasi and functors}
	The map $\big((\pi \colon \mb{E} \ra \mb{X}, \mc{C} \big), \omega ) \mapsto \mc{T}_{\mc{C}, \omega}$ defines a functor
	\begin{equation}\nonumber
			\mc{F} \colon  \rm{Bun}_{\rm{quasi}}^{\nabla}(\mb{X}, \mb{G}) \ra \rm{Trans}(\mb{X},\mb{G}),
		\end{equation}
	where $\omega$ is the strict connection on $\pi \colon \mb{E} \ra \mb{X}$.

	\end{theorem}
 For any principal $\mb{G}$-bundle  $\pi: \mb{E} \ra \mb{X}$ and a morphism of Lie groupoids $F : \mb{Y} \ra \mb{X}$, the pair of pullback projections defines a principal $\mb{G}$ bundle $F^{*}\mb{E} \colon \mb{Y} \times_{F,\mb{X},\pi} \mb{E} \ra \mb{Y}$ over $\mb{Y}$, where $\mb{Y} \times_{F,\mb{X},\pi} \mb{E}$ is the usual strong fibered products of Lie groupoids along the morphisms $F$ and $\pi$ (see \textbf{Section 5.3}, \cite{MR2012261} for details on strong fibered products of Lie groupoids). Now, the following observation is obvious.
	

\begin{lemma}
	If $(\pi \colon \mb{E} \ra \mb{X}, \mc{C})$ is a quasi-principal $\mb{G}$-bundle equipped with a strict connection $\omega$ and $F \colon \mb{Y} \ra \mb{X}$ is any morphism of Lie groupoids, then $(F^{*} \pi \colon \mb{F^{*}\mb{E}} \ra \mb{Y}, F^{*}\mc{C})$ is a quasi-principal $\mb{G}$-bundle with strict connection $\rm{pr}_2^{*}\omega$, where $F^{*}\mc{C} \colon  s^{*}(F_0^{*}E_0) \ra F_1^{*}E_1$ is given by $(\gamma,(x,p)) \ra (\gamma, C \big( F_1(\gamma),p) \big)$ for $\gamma \in X_1$ and $p \in E_0$ such that $F_0(s(\gamma))= \pi_0(p)$.
\end{lemma}
The result below establishes the naturality of  \Cref{Comega parallel transport} with respect to the pullback.
\begin{proposition}\label{Pullback naturality}
	Given a quasi-principal $\mb{G}$-bundle $(\pi \colon \mb{E} \ra \mb{X}, \mc{C})$ equipped with a strict connection $\omega$ and a morphism of Lie groupoids $F \colon \mb{Y} \ra \mb{X}$, the functors $\mc{T}_{F^{*}\mc{C}, \rm{pr_2}^{*}\omega}$ and  $\mc{T}_{\mc{C}, \omega} \circ F_{\rm{thin}}$ (see \Cref{Morphism of thin fundamental groupoid}) are naturally isomorphic.
\end{proposition}
\begin{proof}
	We claim that $\eta \colon Y_0 \ra (\overline{\mb{G} {\rm{-Tor}}})_1$ defined by $y \mapsto \eta_y := [{\rm{pr}}_2|_{(F^{*}\pi)^{-1}(y)}]$ is the required natural isomorphism, where ${\rm{pr}}_2 \colon F^{*}\mb{E} \ra \mb{E}$ is the 2nd projection functor from the pull-back Lie groupoid. Our claim is a consequence of the following two easy observations:
	\begin{itemize}
		\item[(i)] For every $x \xrightarrow {\gamma} y \in Y_1$, we have
		\begin{equation}\nonumber
			[T_{\mc{C}}(F(\gamma))] \circ \eta_{y}= \eta_x \circ [T_{F^{*}\mc{C}}(\gamma)],
		\end{equation}
		\item[(ii)]for every path (with sitting instants)  $\alpha \colon [0,1] \ra Y_0$ such that $\alpha(0)=a$ and $\alpha(1)=b$, we have 
		\begin{equation}\nonumber
			[T_{\omega}^{F(\alpha)^{-1}}] \circ \eta_{b}=\eta_a \circ [T_{{\rm{pr}}_2^{*}\omega}^{\alpha^{-1}}].
		\end{equation}
	\end{itemize}
\end{proof}
\subsection{Smoothness of the Parallel transport functor on a quasi-principal 2-bundle}\label{Smoothness of parallel transport}
We begin with the following observation.
\begin{lemma}\label{diffeology on AutE}
	For any $\mb{G}:=[H \rtimes_{\alpha}G \rra G]$-torsor $\mb{E}$, the group of automorphisms ${\rm{Aut}}(\mb{E}):= {\rm{Hom}}_{\mb{G}-{\rm{Tor}}} ( \mb{E}, \mb{E})$ is  canonicially isomorphic to the Lie group $G$.
\end{lemma}
\begin{proof}
	As for any  Lie group $G$, a $G$-torsor $E$ (\Cref{Notations and Conventions}) and a point $z \in G$, we have a group isomorphism given by 
	\begin{equation}\label{canonical Lie group structure on Aut}
		\begin{split}
			\psi_z \colon & {\rm{Aut}}(E):= {\rm{Hom}}_{G-{\rm{Tor}}}( E, E) \ra G\\
			& f \mapsto \delta(z,f(z)),
		\end{split}
	\end{equation}
	where $\delta \colon E \times E \ra G$ is a smooth map defined implicitly as $x \cdot \delta(x,y)=y$. The isomorphism does not depend on the choice of $z$, and thus ${\rm{Aut}}(E)$ can be canonically identified as a Lie group (see \textbf{Lemma 3.4} in \cite{MR3521476}). Hence, it is sufficient to show that the map 
	\begin{equation}\label{AutMbE}
	\begin{split}
		\theta \colon& {\rm{Aut}}(\mb{E}) \ra {\rm{Aut}}(E_0)\\
		& F:=(F_1,F_0) \mapsto F_0
	\end{split}
	\end{equation}
	is an isomorphism of groups. $\theta$ is obviously a group homomorphism.
	Now, let $\theta(F)=\theta(F')$ for $F, F' \in {\rm{Aut}}(\mb{E})$. Let $\delta \in E_1$. Then there exists unique $h_{\delta} \in H$, such that $\delta=1_{s(\delta)}(h_{\delta},e)$. Hence, $F_1(\delta)=F_1(1_{s(\delta)}(h_{\delta},e))=1_{F'_0(s(\delta))}(h_{\delta},e)=F'_1(\delta)$. So, $\theta$ is injective. Now, suppose $f \in {\rm{Aut}}(E_0)$. For $\delta \in E_1$, define $F_1(\delta) :=1_{f(s(\delta))}(h_{\delta},e)$. Observe that as for any $(h,g) \in H \rtimes_{\alpha}G$ the following identity holds
	\begin{equation}\nonumber
	(h_{\delta (h,g)},e)=(\alpha_{g^{-1}}(h_{\delta}h),e).
	\end{equation}
	Then it follows $F_1$ is a  morphism of $H \rtimes_{\alpha}G$-torsor.
	Hence, to show $\theta$ is onto, it is enough to prove $(F_1,f)$ is a functor. Consistencies with the source, target and unit maps are obvious.
	
Since for any composable $\delta_2,\delta_1 \in E_1$, we have 
	\begin{equation}\nonumber
	h_{\delta_2 \circ \delta_1} = h_{\delta_1} h_{\delta_2},
	\end{equation}
	it follows $(F_1, f)$ is consitent with the composition and hence, $\theta$ is onto.	
	\end{proof}

Let $\overline{\rm{Aut}(\mb{E})}$ denote the automorphism group of the $\mb{G}$-torsor $\mb{E}$ in the groupoid $\overline{\mb{G} {\rm{-Tor}}}$ (\Cref{Quotiented G-tor}). Observe that the quotient functor $\mb{G}-{\rm{Tor}} \ra \overline{\mb{G}-{\rm{Tor}}}$ descends to a quotient map $q \colon {\rm{Aut}}(\mb{E}) \ra \overline{\rm{Aut}(\mb{E})}$.
\begin{proposition}
For any $\mb{G}:=[H \rtimes_{\alpha}G \rra G]$-torsor $\mb{E}$, the group $\overline{\rm{Aut}}(\mb{E})$ is isomorphic to the quotient group $G/\tau(H)$. Hence, $\overline{\rm{Aut}}(\mb{E})$ is a diffeologial group.
\end{proposition}
\begin{proof}

		Consider the quotient map $q \colon {\rm{Aut}}(\mb{E}) \ra \overline{\rm{Aut}(\mb{E})}$. Note that to show $\overline{{\rm{Aut}}(\mb{E})} \cong G/ \tau(H)$,  by the first isomorphism theorem it is sufficient to prove
		\begin{equation}\nonumber
			\psi_z \circ \theta (\ker(q))= \tau(H),
		\end{equation}
		for some $z \in E_0$, where $\psi_z$ and $\theta$ are maps as defined in \Cref{diffeology on AutE}.
		The inclusion $\psi_z \circ \theta (\ker(q)) \subseteq \tau(H)$ follows, since for any $F \in \ker(q)$, there is a smooth $\mb{G}$-equivariant natural isomorphism $\eta \colon \rm{Id}_{\mb{E}} \Longrightarrow F$ and in turn we get the unique element $h_z \in H$ such that $\eta(z)=1_z(h_z,e)$, for which $\psi_z \circ \theta(F)= \tau(h_z)$. On the otherhand, for any $h \in H$, one can define $f \colon E_0 \ra E_0$ as $z.g \mapsto z\tau(h)g$ for each $g \in G$, and thus we get an element $(F_1, f)  \in \rm{Aut}(\mb{E})$ (as in \Cref{AutMbE}). Then one sees that $(F_1, f) \in \ker(q)$ as the prescription $z.g \mapsto 1_z(h,e)(e,g)$ for each $g \in G$ defines a smooth $\mb{G}$-equivariant natural isomorphim $\eta \colon {\rm{id}}_{\mb{E}} \Longrightarrow (F_1, f)$.
Finally, as $G$ is a Lie group, $\overline{\rm{Aut}(\mb{E})}$ is a diffeological group (see \textbf{7.3}, \cite{iglesias2013diffeology}) equipped with quotient diffeology (\Cref{quotient diffeology}). 
\end{proof}
Now, we are ready to show that parallel transport functor \Cref{Comega parallel transport} of a quasi-principal 2-bundle is smooth in an appropriate sense.
\begin{theorem}\label{Smoothness of parallel transport functor}
	Let $(\pi: \mb{E} \ra \mb{X}, \mc{C})$ be a quasi-principal $\mb{G}:=[H \rtimes_{\alpha} G \rra G]$-bundle with a strict connection $\omega: T\mb{E} \ra L(\mb{G})$. Then for each $x \in X_0$, the restriction map $\mc{T}_{{\mc{C}, \omega}}|_{{\Pi_{\rm{thin}}(\mb{X},x)}} \colon \Pi_{\rm{thin}}(\mb{X},x) \ra \overline{\rm{Aut}(\pi^{-1}(x))}$ is a map of diffeological spaces, where $\Pi_{\rm{thin}}(\mb{X},x)$ is the automorphiosm group of $x$ in the diffeological groupoid $\Pi_{\rm{thin}}(\mb{X})$.
\end{theorem}
\begin{proof}
Let $P\mb{X}_x$ denote the set of lazy $\mb{X}$-paths which start and end at $x \in X_0$. $P\mb{X}_x$ being a subset of $P\mb{X}$, is a diffeological space by (\Cref{subsapce diffeology}). Similarly, $\Pi_{\rm{thin}}(\mb{X},x)$ is also equipped with the subspace diffeology induced from the diffeology on $\frac{P\mb{X}}{\sim}$ (see  \Cref{quotient diffeology}). Let $q^{P\mb{X}_x} \colon P\mb{X}_x  \ra \Pi_{\rm{thin}}(\mb{X},x)$ be the quotient map. Note that from \Cref{Technical 1}, it suffices to show that for any plot $ \big( p \colon U \ra P\mb{X}_x \big) \in D_{P\mb{X}_x}$, $\mc{T}_{{\mc{C}, \omega}}|_{\Pi_{\rm{thin}}(\mb{X},x)} \circ  q^{P\mb{X}_x} \circ p \in  D_{\overline{\rm{Aut}(\pi^{-1}(x))}}$.  Suppose $x \in U$, then by \Cref{sum diffeology}, there exists an open neighbourhood $U_x$ around $x$ such that $p|_{U_x}$ is of the form 
		\begin{equation}\nonumber
		p|_{U_x}=(p^0_{X_1}, p^{1}_{PX_0},p^1_{X_1},\cdots, p^{n}_{PX_0}, p^n_{X_1}) \colon U \ra P\mb{X}_n 
	\end{equation}
	for some $n \in \mb{N} \cup \lbrace 0 \rbrace$. Observe that the smoothness of the map 
	\begin{equation}\nonumber
		\begin{split}
			\theta \colon & U_x \ra {\rm{Aut}}(\pi^{-1}(x))\\
			& u \mapsto T_{ \big( p|_{U_x}(u), \mc{C}, \omega \big)} \quad [\rm{see} \,\, \Cref{Definition: Parallel transport of X-paths}.]
		\end{split}
		\end{equation}
will imply $\mc{T}_{{\mc{C}, \omega}}|_{\Pi_{\rm{thin}}(\mb{X},x)} \circ  q^{P\mb{X}_x} \circ p \in  D_{\overline{\rm{Aut}}(\pi^{-1}(x))}$.		
		
Due to the smooth structure on ${\rm{Aut}}(\pi^{-1}(x))$ (\Cref{diffeology on AutE}), $\theta$ is smooth if and only if the map 
\begin{equation}\nonumber
	\begin{split}
		\bar{\theta} \colon & U_x \ra \pi_{0}^{-1}(x)\\
		& u \mapsto  \Big( T_{ \big( p|_{U_x}(u), \mc{C}, \omega \big)} \Big)_{0}(z)
	\end{split}
\end{equation}
is smooth for some choice of $z \in \pi^{-1}(x)$. But, the smoothness of $\bar{\theta}$ follows from the following sequence of facts:
\begin{equation}\nonumber
	\begin{split}
		& 	U_x \ra \pi_{0}^{-1} \big( t(p^0_{X_1}(u)) \big), u \mapsto t \big( \mc{C}(p^0_{X_1}(u),z \big) \,\, {\rm{is}} \, \, {\rm{smooth}},\,\, {\rm{and}}\\
		& U_x \ra \pi_0^{-1} \big(ev_0(p^{1}_{PX_0}) \big), u \mapsto {\rm{Tr}}_{\omega}^{p^{1}_{PX_0}(u)}\Big(t \big( \mc{C}(p^0_{X_1}(u),z \big) \Big)
		\end{split}
	\end{equation}
	is smooth due to \textbf{Lemma 3.13,} \cite{MR3521476}. Proceeding in this fashion for the sequence of maps in $p|_{U_x}=(p^0_{X_1}, p^{1}_{PX_0},p^1_{X_1},\cdots, p^{n}_{PX_0}, p^n_{X_1}) \colon U \ra P\mb{X}_n $, we complete the proof.
\end{proof}

\begin{remark}
The smoothness of $\mc{T}_{\mc{C}, \omega}$ in \Cref{Classical transport} obtained from \Cref{Smoothness of parallel transport functor} coincides with that of (\textbf{Theorem 3.9} of \cite{MR3521476})  for the parallel transport functor of the classical principal $G$-bundle $\pi \colon E \ra M$ over the manifold $M$. 	Recall in \Cref{Equivalence of quasi and functors}, we defined a functor $\mc{F} \colon \rm{Bun}_{\rm{quasi}}^{\nabla}(\mb{X}, \mb{G}) \ra \rm{Trans}(\mb{X},\mb{G})$. At the moment, it is not conclusive whether $\mc{F}$ provides a higher analog of \textbf{Theorem 4.1 }of \cite{MR3521476} or not. In an ongoing work, we are investigating in the said direction.
\end{remark}
\section{Induced parallel transport on VB-groupoids along lazy Haefliger paths}\label{Associated PAPER VERSION}
As an application of the theory developed in the precdeding sections, we investigate parallel transports on VB-groupoids along lazy Haefliger paths. For that we consider the associated VB-groupoid of a quasi-principal 2-bundle with respect to an action of the Lie 2-group on a Baez-Crans 2-vector space. For a detailed account on VB-groupoiods  we refer to  \cite{Bursztyn2016163, MR2157566, MR3696590} and for the 2-vector spaces to \cite{MR2068522}.
\begin{definition}[Definition 3.1, \cite{MR3696590}]\label{Definition of VB groupoidsPaper}
	A \textit{VB-groupoid} over  a Lie groupoid $\mb{X}$ is given by a morphism of  Lie groupoids $\pi \colon \mb{D} \ra \mb{X}$
	\[
	\begin{tikzcd}[sep=small]
		D_1 \arrow[rr,"\pi_1"] \arrow[dd,xshift=0.75ex,"t_D"]
		\arrow[dd,xshift=-0.75ex,"s_D"'] &  & X_1 \arrow[dd,xshift=0.75ex,"t_X"].  		\arrow[dd,xshift=-0.75ex,"s_X"'] \\
		&  &                \\
		D_0 \arrow[rr,"\pi_0"]            &  & X_0           
	\end{tikzcd},\]
	such that the following conditions are satisfied:
	\begin{enumerate}[(i)]
		\item the maps $\pi_1\colon D_1\ra X_1$ and $\pi_0\colon D_0\ra X_0$ are vector bundles,
		\item the maps $(s_D,s_X),(t_D, t_X)$ are morphisms of vector bundles,
		\item  for appropriate $\gamma_1,\gamma_2,\gamma_3,\gamma_4\in V_1$, we have 
		$(\gamma_3\circ \gamma_1)+(\gamma_4\circ \gamma_2)=
		(\gamma_3+\gamma_4)\circ (\gamma_1+\gamma_2)$.
	\end{enumerate}
\end{definition}
A \textit{(linear) cleavage} (see \cite{MR4126305}) on a VB-groupiod $\pi \colon \mb{D} \ra \mb{X}$ is a smooth section $\mc{C}$ of the map $P^{\mb{D}} \colon D_1 \ra X_1 \times_{s,X_0, \pi_0} D_0$, defined by $\delta \mapsto \big( \pi_1(\delta), s(\delta) \big)$, such that $\mc{C}$ is also a morphism of vector bundles. A linear cleavage that satisfies the condition $\mc{C}(1_{\pi(p)},p)=1_p$, for all $p \in D_0$ is either called \textit{unital} \cite{MR4126305} or \textit{right-horizontal lift} \cite{MR3696590}. A (linear) cleavage is called \textit{flat} if it satisfies the condition that for any pair $(\gamma_2, p_2), (\gamma_1, p_1) \in X_1 \times_{s,X_0, \pi_0} D_0$ satisfying ${s}(\gamma_2)={t}(\gamma_1)$ and $p_2=t\bigl({\mathcal C}(\gamma_1, p_1)\bigr)$, we have $\mathcal{C}(\gamma_2 \circ \gamma_1 , p_1)= \mathcal{C}(\gamma_2, p_2) \circ \mathcal{C}(\gamma_1, p_1)$. 
\begin{definition}[Definition 3.1, \cite{MR2068522}]\label{Vector 2-spaces}
	A \textit{2-vector space} is defined as a category $\mb{V}:=[V_1 \rra V_0]$ internal to Vect, the category of finite dimensional vector spaces over a field $\mb{R}$.
\end{definition}
In other words, a 2-vector space $\mb{V}$ is a category such that both $V_1$, $V_0$ are vector spaces and all structure maps are linear. Likewise, we have the notion of a functor internal to Vect between a pair of vector 2-spaces and a natural transformation internal to Vect between a pair of such functors internal to Vect. These data form a strict 2-category in the usual way and is denoted by 2Vect, see \textbf{Section 3} of \cite{MR2068522}. In literature, a different notion of a 2-vector space exists, namely \textit{Kapranov-Voevodsky 2-vector space} \cite{kapranov19942}. 
\begin{example}
	Given a VB-groupiod $\pi \colon \mb{D} \ra \mb{X}$, the groupoid $\pi^{-1}(x):=[\pi_1^{-1}(1_x) \rra \pi_0^{-1}(x)]$ is a 2-vector space for every $x \in X_0$.
\end{example}
\begin{example}\label{Lie 2-algebra}
	Given a Lie 2-group $\mb{G}$, the Lie groupoid $L(\mb{G}):=[L(G_1) \rra L(G_0)]$ is a 2-vector space.
\end{example}
Next, we prescribe a construction of a VB-groupoid from a principal 2-bundle over a Lie groupoid. We define a notion of the left action of a Lie 2-group on a  2- vector space.
\begin{definition}[Section 11, \cite{MR3126940}]\label{Action of a Lie 2-group on a vector 2-space}
	A \textit{left action of a Lie 2-group 	$\mb{G}:=[G_1 \rra G_0]$ on a vector 2-space $\mb{V}:=[V_1 \rra V_0]$} is defined as a functor $\rho \colon \mb{G} \times \mb{V} \ra \mb{V}$, such that the maps $\rho_1 \colon G_1 \times V_1 \ra V_1$ and $\rho_0 \colon G_0 \times V_0 \ra V_0$ are left Lie group actions which induce linear representations of $G_1$ and $G_0$ on $V_1$ and $V_0$ respectively. $\rho_i(g,v)$ will be be denoted as $gv$ for all $g \in G_i, v \in V_i$ and $i=0,1$.
\end{definition}
A weaker version of this action was studied in \cite{MR3556124, Baez_2012}, and for the representation theory of 2-groups we refer to 
\cite{elgueta2007representation, MR3213404, huan20222representations}.
\subsection*{Construction of a VB-groupoid associated to a principal 2-bundle over a Lie groupoid}
For a Lie 2-group $\mb{G}:=[G_1 \rra G_0]$, let $\pi \colon \mb{E} \ra \mb{X}$ be a principal $\mb{G}$-bundle over a Lie groupoid $\mb{X}$. Suppose there is a left action of $\mb{G}$ on a 2-vector space $\mb{V}=[V_1 \rra V_0]$ as in \Cref{Action of a Lie 2-group on a vector 2-space}. Then by the  usual associated vector bundle construction (see \textbf{Chapter 1},\textbf{ Section 5} of \cite{MR1393940}), we get a pair of vector bundles $\lbrace \pi^{\mb{V}}_i \colon \frac{E_i \times V_i}{G_i} \ra X_i \rbrace_{i=0,1}$, defined by $[p_i,v_i] \mapsto \pi_i(p)$ respectively. It is a straightforward, but tedious verifiation that the pair of  manifolds $\bigg\{ \frac{E_i \times V_i}{G_i} \bigg\}_{i=0,1}$  define a Lie groupoid $\frac{\mb{E} \times \mb{V}}{\mb{G}}:= \big[\frac{E_1 \times V_1}{G_1} \rra \frac{E_0 \times V_0}{G_0} \big]$ with obvious structure maps. We call it an \textit{associated VB-groupoid of $\pi \colon \mb{E} \ra \mb{X}$}.

\begin{remark}\label{general associated groupoid bundle construction}
	One can consider the above construction as a special case of the associated groupoid bundle construction mentioned in the \textbf{Remark 3.13} of \cite{herrera2023isometric}, where instead of a vector 2-space, the authors considered an ordinary Lie groupoid.
\end{remark}
\begin{example}[Adjoint VB-groupoid]
It is easy to verify that the adjoint VB-groupoid $\rm{Ad}(\mb{E})$ of a principal $\mb{G}$-bundle $\pi \colon \mb{E} \ra \mb{X}$,  as defined in the \textbf{Section 4} of \cite{chatterjee2022atiyah}, can be realized as an associated VB groupoid $\pi^{L(\mb{G})} \colon \frac{\mb{E} \times L(\mb{G})}{\mb{G}} \ra \mb{X}$ of $\pi \colon \mb{E} \ra \mb{X}$, with respect to the usual adjoint action of $\mb{G}$ on $L(\mb{G})$ (\Cref{Lie 2-algebra}).
\end{example}
The following observation is immediate.
\begin{proposition}\label{Proposition: Associated VB-groupoidPaper}
	For a Lie 2-group $\mb{G}$, let $(\pi \colon \mb{E} \ra \mb{X}, \mc{C})$ be a  quasi-principal $\mb{G}$-bundle over a Lie groupoid $\mb{X}$. Suppose there is a left action of $\mb{G}$ on a 2-vector space $\mb{V}$. Then the associated VB-groupoid $\pi^{\mb{V}} \colon \frac{\mb{E} \times \mb{V}}{\mb{G}} \ra \mb{X}$  over $\mb{X}$ admits a linear cleavage, 
	\begin{equation}\nonumber
		\begin{split}
			\mc{C}^{\mb{V}} \colon & \biggl( X_1 \times_{s, X_0, \pi^{\mb{V}}_0} \frac{E_0 \times V_0}{G_0}\biggr) \ra \frac{E_1 \times V_1}{G_1}\\
			& \quad \hskip 2cm \big(\gamma, [p,v] \big) \mapsto [\mc{C}(\gamma,p), 1_v].
		\end{split}
	\end{equation}
	Moreover,  if $\mc{C}$ is a unital, then so is $\mc{C}^{\mb{V}}$ and likewise if $\mc{C}$ is a categorical connection then $\mc{C}^{\mb{V}}$ is flat.
\end{proposition}
As a straightforward consequence of \Cref{Proposition: Associated VB-groupoidPaper} and \Cref{T_C}, we obtain a  \textit{2Vect-valued pseudofunctor} corresponding to an associated VB-groupoid of a quasi-principal 2-bundle.
\begin{proposition}\label{KLiegrpdvalued pseudoPaper}
	For a Lie 2-group $\mb{G}$, let $(\pi \colon \mb{E} \ra \mb{X}, \mc{C})$ be a  quasi-principal $\mb{G}$-bundle over a Lie groupoid $\mb{X}$ with a left action of $\mb{G}$ on a 2-vector space $\mb{V}$. Then there is a  \textit{2-Vect-valued pseudofunctor}  $$T_{\mc{C}^{\mb{V}}} \colon \mb{X}^{\rm{op}} \ra {\rm{2}}{\rm{Vect}}.$$
\end{proposition}
As a direct consequence of \Cref{Proposition: Parallel transport on principal 2-bundles} and the traditional notion of induced parallel transport on associated fibre bundles (see \textbf{Chapter (iii)}, \cite{MR1393940}), we get the following:
\begin{proposition}\label{Associated transport along path}
	For a Lie 2-group $\mb{G}$, let $(\pi \colon \mb{E} \ra \mb{X}, \mc{C})$ be a  quasi-principal $\mb{G}$-bundle over a Lie groupoid $\mb{X}$ with a strict connection $\omega \colon T \mb{E} \ra L\mb{G}$. Suppose there is a  left action of $\mb{G}$ on a 2-vector space $\mb{V}$. Then, given a  path $\alpha \colon x \ra y$ in $X_0$ there is an isomorphism of 2-vector spaces $T_{\omega, \mb{V}}^{\alpha} \colon  (\pi^{\mb{V}})^{-1}(x) \ra (\pi^{\mb{V}})^{-1}(y)$ defined as 
	\begin{equation}\nonumber
		\begin{split}
			T_{\omega, \mb{V}}^{\alpha} \colon & (\pi^{\mb{V}})^{-1}(x) \ra (\pi^{\mb{V}})^{-1}(y)\\
			& [p,v] \mapsto [{\rm{Tr}}_{\omega_0}^{\alpha}(p), v],\\
			& [\delta, \zeta] \mapsto [{\rm{Tr}}_{\omega_1}^{u \circ \alpha}(\delta), \zeta].
		\end{split}
	\end{equation}
\end{proposition}
Combining \Cref{KLiegrpdvalued pseudoPaper} and \Cref{Associated transport along path}, we obtain a notion of parallel transport on an associated VB-groupoid of a quasi principal 2-bundle equipped with a strict connection along a lazy Haefliger path.
\begin{definition}\label{Comega associated transportpaper}
	Let a Lie 2-group $\mb{G}$ acts on a 2-vector space $\mb{V}$, and $(\pi \colon \mb{E} \ra \mb{X}, \mc{C})$ be a  quasi-principal $\mb{G}$-bundle over a Lie groupoid $\mb{X}$, with a strict connection  $\omega \colon T \mb{E} \ra L(\mb{G})$. Then the isomorphism of 2-vector spaces $T^{\mb{V}}_{(\Gamma, \mc{C}, \omega)}:= T_{\mc{C}^{\mb{V}}}(\gamma_n^{-1}) \circ  T_{\omega, \mb{V}}^{\alpha_n} \circ\cdots \circ T_{\omega,\mb{V}}^{\alpha_1} \circ  T_{\mc{C}^{\mb{V}}}(\gamma_0^{-1})$  will be called 
	the \textit{$(\mc{C}, \omega)$-parallel transport on the associated VB-groupoid $\pi^{\mb{V}}$  along the lazy $\mb{X}$-path $\Gamma =(\gamma_0, \alpha_1,\gamma_1,\cdots, \alpha_n, \gamma_n)$}.
\end{definition}
\begin{remark}
	Using  \Cref{Definition: Parallel transport of X-paths}, $T^{\mb{V}}_{(\Gamma, \mc{C}, \omega)}$ in the above definition can be expressed in terms of $T_{\Gamma, \mc{C}, \omega}$ as follows:
	\begin{equation}\nonumber
		\begin{split}
			T^{\mb{V}}_{(\Gamma, \mc{C}, \omega)} \colon & (\pi^{\mb{V}})^{-1}(x) \ra (\pi^{\mb{V}})^{-1}(y)\\
			& [p,v] \mapsto [T_{(\Gamma, \mc{C}, \omega)}(p), v],\\
			& [\delta,\zeta] \mapsto [T_{(\Gamma, \mc{C}, \omega)}(\delta), \zeta].
		\end{split}
	\end{equation}
\end{remark}
Suitably adapting  \Cref{Theorem: Parallel transport on 2-bundles} to \Cref{Comega associated transportpaper},
one derives the corresponding parallel transport functor.

\begin{remark}
	Although we have confined our attention to the notion of parallel transport on an associated VB-groupoid of a quasi-principal 2-bundle equipped with a strict connection, it is not difficult to generalize the results obtained in this section for an associated groupoid bundle mentioned in \Cref{general associated groupoid bundle construction}.
\end{remark}
\bibliography{references}

\begin{thebibliography}{10}

\bibitem{MR3107517}
Camilo Arias~Abad and Marius Crainic.
\newblock Representations up to homotopy and {B}ott's spectral sequence for
  {L}ie groupoids.
\newblock {\em Adv. Math.}, 248:416--452, 2013.

\bibitem{Baez_2012}
John Baez, Aristide Baratin, Laurent Freidel, and Derek Wise.
\newblock Infinite-dimensional representations of 2-groups.
\newblock {\em Memoirs of the American Mathematical Society}, 219(1032):0--0,
  2012.

\bibitem{baez2011convenient}
John Baez and Alexander Hoffnung.
\newblock Convenient categories of smooth spaces.
\newblock {\em Transactions of the American Mathematical Society},
  363(11):5789--5825, 2011.

\bibitem{baez2002higher}
John~C Baez.
\newblock Higher yang-mills theory.
\newblock {\em arXiv preprint hep-th/0206130}, 2002.

\bibitem{MR2068522}
John~C. Baez and Alissa~S. Crans.
\newblock Higher-dimensional algebra. {VI}. {L}ie 2-algebras.
\newblock {\em Theory Appl. Categ.}, 12:492--538, 2004.

\bibitem{MR2068521}
John~C. Baez and Aaron~D. Lauda.
\newblock Higher-dimensional algebra. {V}. 2-groups.
\newblock {\em Theory Appl. Categ.}, 12:423--491, 2004.

\bibitem{MR2342821}
John~C. Baez and Urs Schreiber.
\newblock Higher gauge theory.
\newblock In {\em Categories in algebra, geometry and mathematical physics},
  volume 431 of {\em Contemp. Math.}, pages 7--30. Amer. Math. Soc.,
  Providence, RI, 2007.

\bibitem{MR2817778}
Kai Behrend and Ping Xu.
\newblock Differentiable stacks and gerbes.
\newblock {\em J. Symplectic Geom.}, 9(3):285--341, 2011.

\bibitem{bridson2013metric}
Martin~R Bridson and Andr{\'e} Haefliger.
\newblock {\em Metric spaces of non-positive curvature}, volume 319.
\newblock Springer Science \& Business Media, 2013.

\bibitem{Bursztyn2016163}
Henrique Bursztyn, Alejandro Cabrera, and Matias {del Hoyo}.
\newblock Vector bundles over lie groupoids and algebroids.
\newblock {\em Advances in Mathematics}, 290:163--207, 2016.

\bibitem{chatterjee2022atiyah}
Saikat Chatterjee, Adittya Chaudhuri, and Praphulla Koushik.
\newblock Atiyah sequence and gauge transformations of a principal 2-bundle
  over a lie groupoid.
\newblock {\em Journal of Geometry and Physics}, 176:104509, 2022.

\bibitem{MR3126940}
Saikat Chatterjee, Amitabha Lahiri, and Ambar~N. Sengupta.
\newblock Path space connections and categorical geometry.
\newblock {\em J. Geom. Phys.}, 75:129--161, 2014.

\bibitem{MR3213404}
Saikat Chatterjee, Amitabha Lahiri, and Ambar~N. Sengupta.
\newblock Twisted actions of categorical groups.
\newblock {\em Theory Appl. Categ.}, 29:No. 8, 215--255, 2014.

\bibitem{MR3504595}
Saikat Chatterjee, Amitabha Lahiri, and Ambar~N. Sengupta.
\newblock Construction of categorical bundles from local data.
\newblock {\em Theory Appl. Categ.}, 31:Paper No. 14, 388--417, 2016.

\bibitem{MR3521476}
Brian Collier, Eugene Lerman, and Seth Wolbert.
\newblock Parallel transport on principal bundles over stacks.
\newblock {\em J. Geom. Phys.}, 107:187--213, 2016.

\bibitem{colman20111}
Hellen Colman.
\newblock On the 1-homotopy type of lie groupoids.
\newblock {\em Applied Categorical Structures}, 19(1):393--423, 2011.

\bibitem{MR3968895}
Matias del Hoyo and Rui Loja~Fernandes.
\newblock Riemannian metrics on differentiable stacks.
\newblock {\em Math. Z.}, 292(1-2):103--132, 2019.

\bibitem{MR4126305}
Matias del Hoyo and Cristian Ortiz.
\newblock Morita equivalences of vector bundles.
\newblock {\em Int. Math. Res. Not. IMRN}, (14):4395--4432, 2020.

\bibitem{del2008homotopy}
Matias~L del Hoyo.
\newblock On the homotopy type of a cofibred category.
\newblock {\em arXiv preprint arXiv:0810.3063}, 2008.

\bibitem{elgueta2007representation}
Josep Elgueta.
\newblock Representation theory of 2-groups on kapranov and voevodsky's
  2-vector spaces.
\newblock {\em Advances in Mathematics}, 213(1):53--92, 2007.

\bibitem{MR4598921}
Alfonso Garmendia and Sylvie Paycha.
\newblock Principal bundle groupoids, their gauge group and their nerve.
\newblock {\em J. Geom. Phys.}, 191:Paper No. 104865, 21, 2023.

\bibitem{MR3480061}
Gr\'{e}gory Ginot and Mathieu Sti\'{e}non.
\newblock {$G$}-gerbes, principal 2-group bundles and characteristic classes.
\newblock {\em J. Symplectic Geom.}, 13(4):1001--1047, 2015.

\bibitem{MR3696590}
Alfonso Gracia-Saz and Rajan~Amit Mehta.
\newblock Vb-groupoids and representation theory of {L}ie groupoids.
\newblock {\em J. Symplectic Geom.}, 15(3):741--783, 2017.

\bibitem{guruprasad2006closed}
K~Guruprasad and Andr{\'e} Haefliger.
\newblock Closed geodesics on orbifolds.
\newblock {\em Topology}, 45(3):611--641, 2006.

\bibitem{gutt2005poisson}
Simone Gutt, John Rawnsley, and Daniel Sternheimer.
\newblock {\em Poisson geometry, deformation quantisation and group
  representations}.
\newblock Number 323. Cambridge University Press, 2005.

\bibitem{haefliger1971homotopy}
A~Haefliger.
\newblock Homotopy and integrability, 133-175 in: Manifolds-amsterdam 1970, lnm
  197, 1971.

\bibitem{haefliger1982groupoides}
Andr{\'e} Haefliger.
\newblock {\em Groupo{\"\i}des d'holonomie et classifiants}.
\newblock Universit{\'e} de Gen{\`e}ve-Section de math{\'e}matiques, 1982.

\bibitem{haefliger1990orbi}
Andr{\'e} Haefliger.
\newblock Orbi-espaces.
\newblock In {\em Sur les groupes hyperboliques d’apr{\`e}s Mikhael Gromov},
  pages 203--213. Springer, 1990.

\bibitem{MR3556124}
Benjam\'{\i}n~A. Heredia and Josep Elgueta.
\newblock On the representations of 2-groups in {B}aez-{C}rans 2-vector spaces.
\newblock {\em Theory Appl. Categ.}, 31:Paper No. 32, 907--927, 2016.

\bibitem{herreracarmona2023chernweillecomte}
Juan~Sebastian Herrera-Carmona and Cristian Ortiz.
\newblock The chern-weil-lecomte characteristic map for $l_{\infty}$-algebras,
  2023.

\bibitem{herrera2023isometric}
Juan~Sebasti{\'a}n Herrera-Carmona and Fabricio Valencia.
\newblock Isometric lie 2-group actions on riemannian groupoids.
\newblock {\em The Journal of Geometric Analysis}, 33(10):323, 2023.

\bibitem{huan20222representations}
Zhen Huan.
\newblock 2-representations of lie 2-groups and 2-vector bundles, 2022.

\bibitem{iglesias2013diffeology}
Patrick Iglesias-Zemmour.
\newblock {\em Diffeology}, volume 185.
\newblock American Mathematical Soc., 2013.

\bibitem{MR4261588}
Niles Johnson and Donald Yau.
\newblock {\em 2-dimensional categories}.
\newblock Oxford University Press, Oxford, 2021.

\bibitem{kapranov19942}
Mikhail~M Kapranov and Vladimir~A Voevodsky.
\newblock 2-categories and zamolodchikov tetrahedra equations.
\newblock In {\em Proc. Symp. Pure Math}, volume~56, pages 177--260, 1994.

\bibitem{kim2020adjusted}
Hyungrok Kim and Christian Saemann.
\newblock Adjusted parallel transport for higher gauge theories.
\newblock {\em Journal of Physics A: Mathematical and Theoretical},
  53(44):445206, 2020.

\bibitem{MR1393940}
Shoshichi Kobayashi and Katsumi Nomizu.
\newblock {\em Foundations of differential geometry. {V}ol. {I}}.
\newblock Wiley Classics Library. John Wiley \& Sons, Inc., New York, 1996.
\newblock Reprint of the 1963 original, A Wiley-Interscience Publication.

\bibitem{MR2270285}
Camille Laurent-Gengoux, Jean-Louis Tu, and Ping Xu.
\newblock Chern-{W}eil map for principal bundles over groupoids.
\newblock {\em Math. Z.}, 255(3):451--491, 2007.

\bibitem{MR1950946}
Ernesto Lupercio and Bernardo Uribe.
\newblock Loop groupoids, gerbes, and twisted sectors on orbifolds.
\newblock In {\em Orbifolds in mathematics and physics ({M}adison, {WI},
  2001)}, volume 310 of {\em Contemp. Math.}, pages 163--184. Amer. Math. Soc.,
  Providence, RI, 2002.

\bibitem{MR1932333}
Marco Mackaay and Roger Picken.
\newblock Holonomy and parallel transport for abelian gerbes.
\newblock {\em Adv. Math.}, 170(2):287--339, 2002.

\bibitem{MR2157566}
Kirill C.~H. Mackenzie.
\newblock {\em General theory of {L}ie groupoids and {L}ie algebroids}, volume
  213 of {\em London Mathematical Society Lecture Note Series}.
\newblock Cambridge University Press, Cambridge, 2005.

\bibitem{MR1262213}
Kirill C.~H. Mackenzie and Ping Xu.
\newblock Lie bialgebroids and {P}oisson groupoids.
\newblock {\em Duke Math. J.}, 73(2):415--452, 1994.

\bibitem{mackenzie2000integration}
Kirill~CH Mackenzie and Ping Xu.
\newblock Integration of lie bialgebroids.
\newblock {\em Topology}, 39(3):445--467, 2000.

\bibitem{mackenzie1999symplectic}
Kirill Charles~Howard Mackenzie.
\newblock On symplectic double groupoids and the duality of poisson groupoids.
\newblock {\em International Journal of Mathematics}, 10(04):435--456, 1999.

\bibitem{martins2010two}
Jo{\~a}o Martins and Roger Picken.
\newblock On two-dimensional holonomy.
\newblock {\em Transactions of the American Mathematical Society},
  362(11):5657--5695, 2010.

\bibitem{mehta2011double}
Rajan~Amit Mehta and Xiang Tang.
\newblock From double lie groupoids to local lie 2-groupoids.
\newblock {\em Bulletin of the Brazilian Mathematical Society, New Series},
  42(4):651--681, 2011.

\bibitem{MR2166453}
I.~Moerdijk and J.~Mr\v{c}un.
\newblock Lie groupoids, sheaves and cohomology.
\newblock In {\em Poisson geometry, deformation quantisation and group
  representations}, volume 323 of {\em London Math. Soc. Lecture Note Ser.},
  pages 145--272. Cambridge Univ. Press, Cambridge, 2005.

\bibitem{MR1950948}
Ieke Moerdijk.
\newblock Orbifolds as groupoids: an introduction.
\newblock In {\em Orbifolds in mathematics and physics ({M}adison, {WI},
  2001)}, volume 310 of {\em Contemp. Math.}, pages 205--222. Amer. Math. Soc.,
  Providence, RI, 2002.

\bibitem{MR2012261}
Ieke Moerdijk and J.~Mr\v{c}un.
\newblock {\em Introduction to foliations and {L}ie groupoids}, volume~91 of
  {\em Cambridge Studies in Advanced Mathematics}.
\newblock Cambridge University Press, Cambridge, 2003.

\bibitem{MR4037666}
Jeffrey~C. Morton and Roger Picken.
\newblock 2-group actions and moduli spaces of higher gauge theory.
\newblock {\em J. Geom. Phys.}, 148:103548, 21, 2020.

\bibitem{MR2520993}
Urs Schreiber and Konrad Waldorf.
\newblock Parallel transport and functors.
\newblock {\em J. Homotopy Relat. Struct.}, 4(1):187--244, 2009.

\bibitem{schreiber2011smooth}
Urs Schreiber and Konrad Waldorf.
\newblock Smooth functors vs. differential forms.
\newblock {\em Homology, Homotopy and Applications}, 13(1):143--203, 2011.

\bibitem{MR3084724}
Urs Schreiber and Konrad Waldorf.
\newblock Connections on non-abelian gerbes and their holonomy.
\newblock {\em Theory Appl. Categ.}, 28:476--540, 2013.

\bibitem{MR3357822}
Emanuele Soncini and Roberto Zucchini.
\newblock A new formulation of higher parallel transport in higher gauge
  theory.
\newblock {\em J. Geom. Phys.}, 95:28--73, 2015.

\bibitem{MR2119241}
Jean-Louis Tu, Ping Xu, and Camille Laurent-Gengoux.
\newblock Twisted {$K$}-theory of differentiable stacks.
\newblock {\em Ann. Sci. \'{E}cole Norm. Sup. (4)}, 37(6):841--910, 2004.

\bibitem{MR3566125}
David Viennot.
\newblock Non-abelian higher gauge theory and categorical bundle.
\newblock {\em J. Geom. Phys.}, 110:407--435, 2016.

\bibitem{MR2223406}
Angelo Vistoli.
\newblock Grothendieck topologies, fibered categories and descent theory.
\newblock In {\em Fundamental algebraic geometry}, volume 123 of {\em Math.
  Surveys Monogr.}, pages 1--104. Amer. Math. Soc., Providence, RI, 2005.

\bibitem{MR3894086}
Konrad Waldorf.
\newblock A global perspective to connections on principal 2-bundles.
\newblock {\em Forum Math.}, 30(4):809--843, 2018.

\bibitem{MR3917427}
Konrad Waldorf.
\newblock Parallel transport in principal 2-bundles.
\newblock {\em High. Struct.}, 2(1):57--115, 2018.

\bibitem{MR3645839}
Wei Wang.
\newblock On the global 2-holonomy for a 2-connection on a 2-bundle.
\newblock {\em J. Geom. Phys.}, 117:151--178, 2017.

\bibitem{MR2805195}
Christoph Wockel.
\newblock Principal 2-bundles and their gauge 2-groups.
\newblock {\em Forum Math.}, 23(3):565--610, 2011.

\bibitem{MR3529236}
Roberto Zucchini.
\newblock On higher holonomy invariants in higher gauge theory {I}.
\newblock {\em Int. J. Geom. Methods Mod. Phys.}, 13(7):1650090, 59, 2016.

\end{thebibliography}
\bibliographystyle{plain}

\end{document}